\documentclass[a4paper]{article}
\usepackage{amsmath, amssymb, amsthm, amsrefs, charter, enumerate, ifpdf, mathdots, stmaryrd, verbatim}
\usepackage{xypic}
\xyoption{curve}

\ifpdf
  \usepackage[pdftex]{graphicx}
\else
  \usepackage[dvips]{graphicx}
\fi

\title{Algebraic theories and $(\infty,1)$-categories}
\author{James D. Cranch}

\theoremstyle{plain}
\newtheorem{thm}{Theorem}[section]
\newtheorem{prop}[thm]{Proposition}

\newtheorem{lemma}[thm]{Lemma}

\theoremstyle{remark}

\newtheorem{remark}[thm]{Remark}
\newtheorem{claim}{Claim}[thm]

\theoremstyle{definition}
\newtheorem{defn}[thm]{Definition}

\newenvironment{innerproof}[1]%
  {\begin{proof}[#1]
   }%
  {\end{proof}}

\newcommand{\dar}{\ar@{-->}}
\newcommand{\dl}{\ar@{--}}
\newcommand{\Ar} {\ar@{=>}}
\newcommand{\Dar}{\ar@{==>}}
\newcommand{\iar}{\ar@{^{(}->}}
\newcommand{\sear}{\ar@{-{>>}}}
\newcommand{\rar}{\ar@{~>}}

\DeclareMathOperator{\Ob}{Ob}

\DeclareMathOperator{\id}{id}
\DeclareMathOperator{\im}{im}
\DeclareMathOperator{\Fun}{Fun}
\DeclareMathOperator{\Map}{Map}
\DeclareMathOperator{\Mod}{Mod}
\DeclareMathOperator{\Grp}{Grp}
\DeclareMathOperator{\Mon}{Mon}
\DeclareMathOperator{\MonGr}{MonGr}
\DeclareMathOperator{\Hom}{Hom}
\DeclareMathOperator{\pl}{pl}
\DeclareMathOperator{\Ex}{Ex}
\DeclareMathOperator{\Arr}{Arr}
\DeclareMathOperator{\Spaces}{Spaces}
\DeclareMathOperator{\Kan}{Kan}
\DeclareMathOperator{\wKan}{wKan}
\DeclareMathOperator{\colim}{colim}
\DeclareMathOperator{\Theories}{Theories}
\DeclareMathOperator{\Aug}{Aug}
\DeclareMathOperator{\Aut}{Aut}
\DeclareMathOperator{\Alg}{Alg}
\DeclareMathOperator{\Free}{Free}
\DeclareMathOperator{\Fact}{Fact}
\DeclareMathOperator{\Th}{Th}

\newcommand{\isom}{\operatorname{\cong}}
\newcommand{\fib}{\mathrm{fib}}
\newcommand{\N}{\mathbb{N}}
\newcommand{\Z}{\mathbb{Z}}

\newcommand{\cA}{\mathcal{A}}
\newcommand{\cB}{\mathcal{B}}
\newcommand{\cC}{\mathcal{C}}
\renewcommand{\cD}{\mathcal{D}}
\newcommand{\cE}{\mathcal{E}}

\newcommand{\HSp}{\mathcal{H}}
\newcommand{\cI}{\mathcal{I}}
\newcommand{\cJ}{\mathcal{J}}
\newcommand{\cK}{\mathcal{K}}
\renewcommand{\cL}{\mathcal{L}}
\newcommand{\cM}{\mathcal{M}}

\newcommand{\cU}{\mathcal{U}}
\newcommand{\fC}{\mathfrak{C}}
\newcommand{\Fin}{{\mathrm{FinSet}}}
\newcommand{\Fint}{\Fin^\times}
\newcommand{\Fins}{{\Fin_*}}
\newcommand{\Finst}{{\Fin_*^\times}}
\newcommand{\Finop}{{\Fin^\op}}
\newcommand{\Finsop}{{\Fin_*^\op}}
\newcommand{\FSI}{{\Fin^{\isom}}}
\newcommand{\Gr}{\operatorname{Gr}}
\newcommand{\Set}{\mathrm{Set}}

\newcommand{\sSet}{\mathrm{sSet}}
\newcommand{\sCat}{\mathrm{sCat}}
\newcommand{\Cat}{\mathrm{Cat}}
\newcommand{\Cinfty}{\Cat_\infty}
\newcommand{\Cinftypp}{\Cat_\infty^{\mathrm{pp}}}

\newcommand{\TwoSpan}{2\Span}
\newcommand{\Span}{\mathrm{Span}}
\newcommand{\Vect}{\mathrm{Vect}}
\newcommand{\FinSpaces}{\mathrm{FinSpaces}}
\newcommand{\op}{\mathrm{op}}

\newcommand{\Funpp}{\Fun^{\mathrm{pp}}}

\newcommand{\timeso}[1]{\mathop{\times}_{#1}}
\newcommand{\Spant}{\Span^\times}
\newcommand{\tcCt}{\tilde{\cC}^\times}
\newcommand{\cCt}{\cC^\times}
\newcommand{\cDt}{\cD^\times}
\newcommand{\cCot}{\cC^\otimes}
\newcommand{\Algt}{\Alg^\times}
\newcommand{\Algot}{\Alg^\otimes}
\newcommand{\bAlg}{\overline{\Alg}}
\newcommand{\Mont}{\Mon^\times}
\newcommand{\Monot}{\Mon^\otimes}
\newcommand{\bMon}{\overline{\Mon}}

\newcommand{\tp}{\tilde{p}}
\newcommand{\Spanc}{\Span^{\textrm{coll}}}
\newcommand{\Finsc}{\Fins^{\textrm{coll}}}
\newcommand{\limit}{\varprojlim}
\newcommand{\Monoids}{\mathrm{Mon}}
\newcommand{\coh}{\textrm{coh}}

\newcommand{\Ho}{\operatorname{ho}}

\newcommand{\bisSet}{\mathrm{s}^2\mathrm{Set}}

\newcommand{\skel}{\mathrm{skel}}

\newcommand{\cCD}{\cC^{\Delta^1}}

\newcommand{\Din}{\operatorname{Din}}

\newcommand{\GrSpan}{\Gr\Span}
\newcommand{\DR}{{D^\times_+}}
\newcommand{\Spam}{\mathrm{RSpan}}
\newcommand{\MLend}[1]{#1^\sharp}

\newcommand{\GL}{GL}
\newcommand{\FinCat}{\mathrm{FinCat}}
\newcommand{\FinCatInfty}{\FinCat_\infty}
\newcommand{\ThMon}{{\mathrm{Th}_\mathrm{Mon}}}

\newcommand{\rspan}[7]{\xymatrix{{#1}&{#3}\ar[l]_\Delta^{#2}\ar[r]^\Pi_{#4}&{#5}\ar[r]^\Sigma_{#6}&{#7}}}

\newdir{pb}{:(1,-1)@^{|-}}
\def\pb#1{\save[]+<12 pt,0 pt>:a(#1)\ar@{pb{}}[]\restore}

\begin{document}

\begin{titlepage}
\begin{center}

\phantom{blank}~\\
\vspace{6cm}
{\bf \Large Algebraic theories and $(\infty,1)$-categories}\\
\vspace{1cm}
{\bf \large James Donald Cranch}\\
\vspace{1cm} {\small A thesis submitted for the degree of}\\
{\small Doctor of Philosophy}\\
\vspace{4cm}
{\small Department of Pure Mathematics\\ School of Mathematics and Statistics\\
The University of Sheffield}\\
\bigskip
{\small July 2009} \vspace{1cm}
\end{center}
\end{titlepage}

\newpage
\thispagestyle{empty} \phantom{1}
\newpage

\begin{abstract}
Algebraic theories, introduced in the 1960s by Lawvere, are a pleasant approach to universal algebra, encompassing many standard objects of abstract algebra: groups, monoids, rings, modules, algebras, objects with action by a finite group, and so on. In this thesis we adapt this classical notion, obtaining a framework suitable for representing analogous objects from modern algebraic topology, which have algebraic operations that only satisfy axioms up to coherent higher homotopy.

We work in the framework of quasicategories developed by Joyal and Lurie: since these are just a subcategory of simplicial sets, they support both a good conceptual topological framework, and a powerful family of homotopy-theoretic methods.

We provide a general study of quasicategorical theories. Then we introduce one which models $E_\infty$ monoids. In addition, we study distributive laws. This allows us to introduce a model for objects with two $E_\infty$ monoidal structures, one distributing over the other. Thus we obtain a model for $E_\infty$ semiring spaces.

We study grouplike objects, and so define theories modelling grouplike $E_\infty$ monoids and $E_\infty$ ring spaces. The former gives us a new approach to infinite loop space theory (or, in essence, the theory of connective spectra), and the latter gives us a new approach to multiplicative infinite loop space theory (or the theory of connective ring spectra).

These models offer alternatives to approaches considered by Lurie.

We apply this to constructing units of ring spectra, reproving a theorem of Ando, Blumberg, Gepner, Hopkins and Rezk. We apply it also to sketch a construction of the $K$-theory of monoidal quasicategories and ring quasicategories, offering an alternative framework to that provided by Elmendorf and Mandell.

\vspace{6cm} ~
\newpage \thispagestyle{empty} \phantom{1}
\end{abstract}

\tableofcontents

\newpage
\thispagestyle{empty} \phantom{1}
\newpage

\section{An introduction to the language of quasicategories}

This section introduces quasicategories and the basic constructions one makes with them. We claim no originality, and all unattributed material is an account of the philosophy found in \cite{HTT} and in \cite{JoyalTierneyBook}.

\subsection{Structure in homotopy theory}

Homotopy theory can be regarded as the study of spaces up to homotopy equivalence. Frequently, constructions in homotopy theory are only defined up to (canonical) homotopy equivalence, and are invariant up to homotopy equivalence.

This motivates the construction of the \emph{homotopy category of spaces} $\HSp$: one takes a category of homotopically well-behaved spaces (for example, the CW complexes, or the full subcategory of simplicial sets consisting of the Kan complexes), and then identifies homotopic maps between any two spaces.

The homotopy category is an extremely useful and important concept, in many regards, but there are also many aspects of homotopy theory that it fails to record faithfully.

For example, there is a natural notion of \emph{homotopy limits} and \emph{colimits} of spaces. While these have become established as extremely natural concepts in homotopy theory, philosophically comparable to other limits and colimits in category theory, they are not in fact limits and colimits in $\HSp$ (nor in any other category).

Also, while the notion of a commutative monoid object in $\HSp$ --- a space equipped with a product which is commutative and unital up to homotopy --- is  an important one in the theory, it most usually appears as a shadow of a stronger notion, the notion of an $E_\infty$-monoid. This admits no natural category-theoretic description in terms of $\HSp$.

One can come to understand the problem as being the one-dimensionality of the category $\HSp$: as a category, it stores the equivalence classes of spaces, and also equivalence classes of maps, but it fails to record homotopies between maps, even though these can provide important, highly nontrivial information. It also fails to record homotopies between homotopies, and the general notion of \emph{higher homotopy}.

\subsection{Notions of higher category}

A method for dealing with this structure is offered by the language of \emph{higher category theory}.

This exists to study categories whose morphisms have some a notion of morphisms between them.

For example, categories themselves admit functors between them, and thus categories and functors together form a category. However, the functors have natural transformations between them, which form a higher level of structure. We say that categories, functors and natural transformations together form a \emph{2-category}, in which the categories are the \emph{objects} or \emph{0-cells}, the functors are the \emph{morphisms} or \emph{1-cells}, and the natural transformations are the \emph{2-cells}.

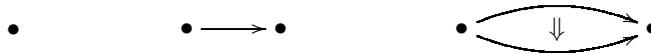
\begin{figure}[htbp]
\begin{displaymath}
\bullet
\hskip 2cm
\xymatrix{\bullet\ar[r]&\bullet}
\hskip 2cm
\xymatrix{\bullet\ar@/_1.8ex/[rr]\ar@/^1.8ex/[rr]&\Downarrow&\bullet}
\end{displaymath}
\caption{A 0-cell, a 1-cell between 0-cells, and a 2-cell between 1-cells}
\end{figure}

This can be extended: there are situations where one might want to consider a notion of \emph{3-cells}: morphisms between 2-cells. Indeed, there are situations where it makes sense to talk of $n$-cells for every natural number $n$: there should be a good notion of $\omega$-categories.

A homotopy between continuous maps can certainly be regarded as a kind of morphism between them; and higher homotopies could in turn be regarded as morphisms between homotopies. This makes the language of higher categories intrinsically relevant to topology.

However, there are problems. Higher category theory is currently hard work: there are several competing definitions of higher categories, with an incomplete theory relating them \cite{LeinsterBook}.

On the other hand, the situations we are mostly interested in lack much of the generality that makes higher category theory difficult. In particular, homotopies are less general than 2-cells in an arbitrary higher category: every homotopy is invertible up to higher homotopies. We say that we are interested in \emph{$(\infty,1)$-categories} (meaning we have cells of every degree, but only the $0$-cells and $1$-cells need not be invertible up to homotopy). In practice we shorten this to \emph{$\infty$-categories} where no confusion is possible.

\subsection{Quasicategories as higher categories}

We recall that an ordinary category $\cC$ has a nerve $N(\cC)$, a simplicial set in which an $n$-cell corresponds to a string of $n$ composable morphisms of $\cC$. The nerve is a functorial construction, and is an equivalence onto its image.

It is standard to expect that, whichever definition of higher categories we use, there is a compatible notion of nerve. In fact, most definitions of higher categories are based on some explicit combinatorial geometric notion of cells, making an appropriate definition of the nerve readily available.

By its very definition, a simplicial set has consistently directed edges (in the sense that higher simplices have their vertices totally ordered by the edges), but simplices have no other orientation data besides this. This means that, being unable to distinguish the direction of higher cells, a nerve is able to store only the data of an $(\infty,1)$-category. However, the nerve should be a complete invariant of $(\infty,1)$-categories, up to some appropriate notion of equivalence.

The basic idea of \emph{quasicategories} is to allow simplicial sets to model $(\infty,1)$-categories, by identifying them with their nerve. The problem then is to choose which simplicial sets should be selected as appropriate.

We recall that nerves of ordinary categories are precisely those simplicial sets $X$ for which all inner horns $\Lambda^n_k\rightarrow X$ (for $0<k<n$) admit a unique extension to an $n$-simplex $\Delta^n\rightarrow X$. Here the simplicial set $\Lambda^n_k$ is the union of all faces of $\Delta^n$ except the $k$th face. In the case of $(\infty,1)$-categories, it is unreasonable to demand uniqueness: there could be several lifts and some homotopy-theoretic interplay between them.

For example, a chain of $1$-cells $x\rightarrow y\rightarrow z$ should certainly be composable to obtain a 2-cell. However, such composites ought to be parametrised by the collection of 2-cells from any such composite to itself: this may very well be nontrivial.

As a result, we can make our central definition (following Joyal \cite{Joyal}, who himself followed Cordier and Porter \cite{Cordier-Porter}, who themselves followed Boardman and Vogt \cite{Boardman-Vogt}): a \emph{quasicategory} consists of a simplicial set $\cC$ such that any inner horn $\Lambda^n_k\rightarrow\cC$ (for $0<k<n$) admits an extension (which is not necessarily unique) to a full $n$-simplex $\Delta^n\rightarrow\cC$.

We adjust our nomenclature as part of a serious effort to treat these as a suitable setting for homotopy-theoretic category theory: an element of $\cC_0$ will be referred to as an \emph{object} or \emph{element}; an element of $\cC_1$ will be referred to as a \emph{morphism}, and an element in $\cC_2$ will be referred to as a \emph{homotopy}.

We define a \emph{functor} between two quasicategories to simply be a map of simplicial sets: the simplicity of this definition is an advert for quasicategories.

An important special case is provided by the case of $(\infty,0)$-categories, also known as $\infty$-groupoids. These have the property that all $k$-cells for $k\geq 1$ are invertible up to homotopy. Notionally, it consists of $0$-cells, homotopies between $0$-cells, homotopies between homotopies between $0$-cells, and so on.

The same analysis, in this situation, tells us that we should lift all horns, and therefore model $\infty$-groupoids by Kan complexes. It is a classical fact that Kan complexes provide a combinatorial model for spaces. This can be made compatible with our intuition: spaces consist of points, paths, paths between paths, and so on.

\subsection{Comparison with other notions}
\label{quasicategory-comparisons}

Another model for $(\infty,1)$-categories is provided by \emph{topological categories}: categories enriched in topological spaces (or in Kan complexes). The idea here is that an $(\infty,1)$-category may be viewed as a category enriched in $(\infty,0)$-categories: there is an $(\infty,0)$-category of cells  (the \emph{homspace}) between any two objects.

This model is more intuitive in many situations, but it is frequently harder to use. It is common to use it on occasion (we shall have occasion to use it later in this thesis); and it is helpful to know that there is a good adjoint pair of comparison functors between quasicategories and categories enriched in Kan complexes (which, by considerable abuse of language, we call \emph{simplicial categories}).

The hard work of the comparison rests in the fact that composition is strictly defined and even strictly associative in a simplicial category, but not in a quasicategory.

The comparison functors are defined using the help of a sequence of simplicial categories $\fC[I]$ associated to any totally ordered set $I$. The first one is $\fC[0]=*$, the point. Then $\fC[01]$ consists of two objects labelled $0$ and $1$, and a single point in the homspace from $0$ to $1$.

$\fC[012]$ has three objects, $0$, $1$ and $2$. The homspaces are given by
$$\fC[012](0,1)=\fC[012](1,2)=*,$$
and $\fC[012](0,2)$ is the interval (with one vertex given by the composite of the morphisms from $0$ to $1$ to $2$).  So this simplicial category is the simplicial category consisting of two composable morphisms, and another morphism homotopic to their composite.

In general, we define $\fC[I]$ to have object set $I$, and morphisms given by:
\begin{displaymath}
\fC[I](i,j) = 
\begin{cases}
\emptyset, & \text{if $i>j$;}\\
*,         & \text{if $i=j$;}\\
E(i,j)     & \text{if $i<j$.}
\end{cases}
\end{displaymath}
Here $E(i,j)$ consists of the `indiscrete simplicial set' whose objects are the subsets $J\subset [i,j]=\{k\vert i\leq k\leq j\}$ with $i,j\in J$ (that is, the nerve of the category with these objects and a single isomorphism between any two of them). This simplicial set can be viewed as a contractible cube of dimension equal to the number of elements of $I$ strictly between $i$ and $j$.

Later we will write $\fC_n$ for the simplicial category $\fC[0\cdots n]$.

Now, using $\Delta^I$ to denote the simplex on $I$ vertices, we can define the \emph{coherent nerve} $N^\coh(\cC)$ of the simplicial category $\cC$ to be the simplicial set
$$\sSet(\Delta^I,N^\coh(\cC))=\sCat(\fC[I],\cC);$$
this simplicial set is shown in \cite{HTT}*{1.1.5.10} to be a quasicategory. It is also shown there that the functor $\fC$ can be extended from the simplices $\Delta^I$ to the whole category of simplicial sets, in order to provide a left adjoint to the coherent nerve functor.

Another popular approach to $(\infty,1)$-categories uses \emph{complete Segal spaces}, which are due to Rezk \cite{Rezk}. One notices that a category consists of a set of objects $\cC_0$ and a set of morphisms $\cC_1$, equipped with certain maps between them. These maps can be characterised by saying that the objects $\cC_n$ form a simplicial set, where we define

$$\cC_n = \underbrace{\cC_1\timeso{\cC_0}\cC_1\timeso{\cC_0}\cdots\timeso{\cC_0}\cC_1\timeso{\cC_0}\cC_1}_{n}.$$
Here these defining maps already use the source and target maps $\cC_1\rightarrow\cC_0$, which are also the two face maps in degree 0. Of course, this is just a restatement of the nerve of a category.

The key idea of a complete Segal space is to generalise this: an $(\infty,1)$-category can have a $(\infty,0)$-category of objects and an $(\infty,0)$-category of morphisms. Thus we model an $(\infty,1)$-category as a certain sort of bisimplicial set: one simplicial direction records the spaces representing these $(\infty,0)$-categories, and the other simplicial direction records the nerve-like structure, as described. This is defined in detail in \cite{Rezk}, and it is shown by Joyal and Tierney \cite{JoyalTierney} that this approach is essentially equivalent to that of quasicategories. The point of view of complete Segal spaces (developed by Barwick) has also received recent attention in approaches by Freed, Hopkins and Lurie to higher cobordism categories \cite{LurieCob}.

\subsection{Examples of quasicategories}

Having done some work to define these objects, it is sensible to ask for good examples. To start with, many finite simplicial sets are quasicategories; these can do service as ``diagram quasicategories'', but are usually of little interest in their own right.

Foremost among the more important examples is $\Spaces$, the \emph{quasicategory of spaces}. This is most easily built indirectly: we take the simplicial category $\Spaces^\Delta$ whose objects are Kan complexes (a natural model of spaces), and whose homspaces $\Spaces^\Delta(X,Y)$ form the Kan complex $\Map(X,Y)$ defined by $\Map(X,Y)_n=\sSet(X\times\Delta^n,Y)$. Then we define $\Spaces=N^\coh(\Spaces^\Delta)$.

This quasicategory is a natural home for a good deal of homotopy theory and algebraic topology.

More introspectively, there is a quasicategory $\Cinfty$ of quasicategories. This can be defined in an analogous manner: we start from the simplicial category $\Cinfty^\Delta$ whose objects are quasicategories and where the homspaces $\Cinfty^\Delta(\cC,\cD)$ consists of the maximal Kan complex contained in $\sSet(\cC,\cD)$ (there is a unique such complex). Studying this is far from purely navelgazing: it will prove directly useful for many of the practical constructions in this thesis.

While frequently useful, $\Cinfty$ does not encapsulate all the structure associated to quasicategories and functors between them. Really, quasicategories should form an $(\infty,2)$-category. Indeed, there is a notion of \emph{natural transformation} between functors between quasicategories: a natural transformation between $F,G:\cC\rightarrow\cD$ is a functor $\Delta^1\times\cC\rightarrow\cD$ which restricts to $F$ and $G$ on $0,1\in\Delta^1$ respectively.

On the other hand, passing to the maximal Kan complex in $\sSet(\cC,\cD)$ is exactly the same as recording only the \emph{homotopies between functors}: those natural transformations which are invertible up to homotopy. This gives us the maximal $(\infty,1)$-category contained in the $(\infty,2)$-category of quasicategories. Some technology for using $(\infty,2)$-categories for this problem has been developed by Lurie \cite{LurieBicat}; various other approaches by other authors are currently rumoured to be in preparation.

Much algebraic topology makes use of spaces only indirectly, preferring the language of \emph{spectra}: these objects both represent generalised cohomology theories, and represent the features of spaces which can be detected by generalised cohomology. The classical introduction to the subject is \cite{Adams}. Most traditional attempts to build models for the category of spectra proceed out of some direct motivation to localise the category of spaces by inverting the suspension functor (since it is an isomorphism on generalised cohomology theories).

Lurie \cite{DAG-I} has demonstrated that one can build a quasicategory of spectra merely by inverting the suspension functor on the quasicategory of spaces in some formal way; we make no direct use of this, but this quasicategory is known to be very close to some quasicategories we construct later.

\subsection{Category theory with quasicategories}
\label{quasicategory-theory}

We now work to demonstrate some ways in which it is convenient to work with quasicategories: we imitate some standard category-theoretic constructions in this context.

The first question is what it should mean for an object $1$ of a quasicategory to be a \emph{terminal object}. According to our motivating philosophy, this should mean some kind of ``homotopy terminal'' object: rather than demanding that the space of morphisms from any other object to $1$ be a point, we should demand that it be a ``homotopy point'': in other words, that it be contractible.

In the language of simplicial categories, this would yield a workable definition as it stands. However, for quasicategories there is a pleasant alternative approach, which generalises to other limits more readily. We recall that in an ordinary category $\cC$, the object $1$ is terminal if the forgetful functor $\cC_{/1}\rightarrow\cC$ is an equivalence (where $\cC_{/1}$ denotes the category of objects of $\cC$ over $1$).

Similarly, the diagram $x\leftarrow z\rightarrow y$ is a product diagram if the diagram of overcategories
$$\cC_{/z}\longrightarrow \cC_{/x,y}$$
is an equivalence.

We can imitate this if we can provide a workable quasicategorical definition of an overcategory. The intuition is that a diagram of shape $L$ over another diagram $f:K\rightarrow\cC$ should just consist of a functor from some $L\star K$ to $\cC$. That simplicial set $L\star K$ should contain copies of $L$ and $K$ (with the copy of $K$ mapped as $f$), together with simplices going from the copy of $L$ to the copy of $K$, expressing the transformation from the diagram to $f$.

Thus we define the join functor $\star:\sSet\times\sSet\rightarrow\sSet$ by
$$(L\star K)_n=\bigsqcup_{i+j=n-1}\left(L_i\times K_j\right).$$
In order to interpret this definition correctly, we must adopt the convention that $i$ or $j$ may attain the value $-1$, and that $X_{-1}$ is a singleton. One should note that this is consistent with the convention that $\Delta^{-1}=\emptyset$; every simplicial set admits a unique map from the empty space.

In words, an $n$-simplex of the join consists of two parts: the early vertices form a simplex in $K$ and the later vertices form a simplex in $L$.

Using this we can readily define an overcategory for any diagram in a quasicategory: if we have $f:K\rightarrow\cC$, then we define the overcategory $\cC_{/f}$ by
$$\sSet(L,\cC_{/f})=\left\{\text{maps $L\star K\rightarrow\cC$ agreeing with $f$ on $K$}\right\}.$$

Definitions of this sort will be seen frequently, and it is well worth spelling out how they work. We can recover an actual simplicial set from this formula. Indeed, since $(\cC_{/f})_n = \sSet(\Delta^n,\cC_{/f})$, we plug in $L=\Delta^n$ to get the $n$-simplices. The right-hand-side of the definition is naturally contravariantly functorial in $X$ so we can recover the simplical structure by plugging in the face and degeneracy maps between simplices.

Any such definition involves an implied claim (usually obvious or immediate to check) that, as a contravariant functor of $X$, the morphisms from $X$ takes colimits to limits; this is the unique property required to show that a functor represents a simplicial set.

In this case, it is normal to denote $\cC_{/f}$ by $\cC_{/K}$ when the map $f$ is obvious from the context.

Using this definition we can say that an object $1$ is \emph{terminal} in a quasicategory $\cC$ if, regarded as a homomorphism $f:*\rightarrow\cC$, the morphism $\cC_{/f}\rightarrow\cC$ is acyclic Kan.

Moreover, we can easily define other limits. A \emph{cone} in a quasicategory is a diagram $1\star K\rightarrow\cC$; it is a \emph{limit cone} if the morphism $\cC_{/(1\star K)}\rightarrow\cC_{/K}$ is acyclic Kan.

This can naturally all be dualised to define the \emph{undercategory} $\cC_{f/}$, \emph{initial objects}, and \emph{colimits} in general. To be explicit, the undercategory is defined by the formula
$$\sSet(X,\cC_{f/})=\left\{\text{maps $K\star X\rightarrow\cC$ agreeing with $f$ on $K$}\right\}.$$

In order to define adjunctions of quasicategories, we shall need a further idea; to motivate it we recall the notion of a Grothendieck fibration from ordinary category theory \cite{Borceux-I}.

These were defined in order to model categories of things depending on some parameter valued in another category. An example is the category of modules over some ring, where objects are pairs $(R,M)$ consisting of a ring and a module over it, and morphisms $(R,M)\rightarrow (S,N)$ consist of a ring map $f:R\rightarrow S$ and a module map $\phi:M\rightarrow f^*N$.

This category should be thought of as fibred over the category of rings. Indeed, it admits a functor to the category of rings, and the preimage of any particular ring (the \emph{fibre}) is the category of modules over that ring. Moreover, the morphisms between modules over different rings encode the pullback of modules.

So Grothendieck fibrations over a category $\cC$ are defined as a special kind of functor $\cE\rightarrow\cC$. The definition is set up so that functors $F:\cC\rightarrow\Cat$ are equivalent to Grothendieck fibrations over $\cC$: the fibre over $x\in\Ob\cC$ in $\cE$ is identified with the image category $F(x)$.

We can imitate this theory perfectly in quasicategories. A \emph{cocartesian fibration} is a map of quasicategories $\pi:\cE\rightarrow\cC$ which is:
\begin{itemize}
\item an \emph{inner fibration} in the sense that, for every square like the following:
\begin{displaymath}
\xymatrix{\Lambda^n_k\ar[d]\ar[r]&\Delta^n\ar[d]\dar[dl]\\
          \cE\ar[r]_\pi&\cC}
\end{displaymath}
(where $0<k<n$) there is a lifting as shown by the dotted arrow.
\item For any morphism $f:x\rightarrow y$ in $\cC_1$ and $X\in\cE_0$ with $\pi(X)=x$, there is a morphism $F\in\cE_1$ with $\pi(F)=f$, and such that the map of undercategories
$$\cE_{F/}\longrightarrow \cE_{X/}\timeso{\cC_{x/}}\cC_{f/}$$
is acyclic Kan.
\end{itemize}
The second part is a natural extension of the ordinary Grothendieck fibration condition: it says more or less that any morphism $g$ out of $X$ in $\cE$ which has $\pi(g)$ factoring through $f$ factors through $F$ in a manner which is unique up to homotopy.

It is a theorem \cite{HTT}*{Section 3.2} that cocartesian fibrations over $\cC$ are equivalent to functors $\cC\rightarrow\Cinfty$. We can dualise the definition of a cocartesian fibration, working with overcategories instead of undercategories, and obtain the definition of a \emph{cartesian fibration}; these are equivalent to functors $\cC^\op\rightarrow\Cinfty$.

With this formalism there is an easy definition of \emph{adjunctions} as maps of quasicategories $\cE\rightarrow\Delta^1$ which are simultaneously cartesian and cocartesian fibrations. Regarded as the former, they define the functor $F:\cC\rightarrow\cD$; regarded as the latter, they define the functor $U:\cD\rightarrow\cC$.

In the case where $\cC$ and $\cD$ are nerves of ordinary categories, we can assemble the appropriate fibration with ease; it is an ordinary category over $\Delta^1$ consisting of $\cC$ over $0$, $\cD$ over $1$, and for $c\in\Ob\cC$ and $d\in\Ob\cD$ we take $\cE(c,d)=\cD(Fc,d)=\cC(c,Ud)$.

Lastly, we make some remarks on homspaces of quasicategories. Of course, one way to recover these is to use the adjoint $\fC$ to the coherent nerve functor, as described  above, to produce a simplicial category, and then take a homspace of that.

In practice, however the adjoint $\fC$ is very unwieldy. Often we need to study only one single homspace of a quasicategory (rather than all of them at once). There are a number of tidy methods for doing this. Our preference is to define $\Hom_\cC(a,b)$ by
$$\sSet(X,\Hom_\cC(a,b)) = \left\{\text{$f:X\times\Delta^1\rightarrow\cC$ such that $f(X,0)=a$, $f(X,1)=b$} \right\}.$$
This is pleasingly symmetric in $a$ and $b$, and quite tractable in practice.

\subsection{Notions of structured quasicategories}

Lurie \cite{DAG-II} has provided a notion of a symmetric monoidal quasicategory. A symmetric monoidal structure on $\cC$ is a cocartesian fibration $\cC^*\rightarrow\Fins$ over the category of based finite sets, such that the fibre over a based finite set $X_+$ is equivalent to $\cC^X$, and such that the various morphisms $X_+\rightarrow 1_+$ model the projections $\cC^X\rightarrow\cC$.

Thus each based finite set $X_+$, we have something equivalent to $\cC^X$. Given a map $f:X_+\rightarrow Y_+$ of based finite sets, a morphism over $f$ in $\cC^*$ represents the way to take products indexed by $f$, to get from an $X$-indexed collection of objects of $\cC$ to a $Y$-indexed collection.

The theory is configured to be as lax as possible, and so it allows products to be defined only up to equivalence: a symmetric monoidal $\infty$-category $\cC$ admits a product functor $\cC\times\cC\rightarrow\cC$, but this need only be defined up to equivalence.

This represents the structure in a pleasing and malleable manner.

However, this is not necessarily the optimal way of presenting the structure. In particular, this observation fails to take into account that a monoidal category has a natural diagonal $\cC\rightarrow\cC\times\cC$, which interacts with the monoidal structure.

The problem is not a novel one. Lurie's definition is a way of importing the theory of $E_\infty$-algebras into the setting of $(\infty,1)$-categories. The language of operads, in general, does not permit reuse of variables: no axioms can be mentioned or theorems deduced with formulae which mention the same variable twice, and therefore one cannot axiomatise the notion of a diagonal.

A similar problem occurs when one tries to define a ring. The problem here is the distributive law: the right-hand side of the identity $a(b+c)=ab+ac$ mentions the variable $a$ twice.

\subsection{Theories}

There is a framework for classical universal algebra which does encompass this structure, that of \emph{algebraic theories}. 

We will explain them with an example: indeed, a running example for us will be the theory of commutative monoids. One forms a category $\Th(\Monoids)$ whose objects are finite sets, and where the morphisms $\Th(\Monoids)(A,B)$ are the morphisms between the free commutative monoid $\Mon(A)$ on $A$ and the free commutative monoid $\Mon(B)$ on $B$.

We will write $\{x_a\}_{a\in A}$ for the generators of $\Mon(A)$, and $\{y_b\}_{b\in B}$ for the generators of $\Mon(B)$. Then a map from $\Mon(A)$ to $\Mon(B)$ sends each $x_a$ to a word in the $\{y_b\}$. Such a word can be represented by a finite set with a map to $B$ (considered up to isomorphism): the number of elements over an element of $B$ describes the number of occurrences of that element in the word. We need to give a word for every element of $A$; this corresponds to a set with a map to $A$ and to $B$: the preimage of an element in $a$ determines a set over $B$ representing the appropriate word. Thus such a map is determined by an isomorphism class of sets over both $A$ and $B$.

We revisit these ideas in Section \ref{homs-in-span}.

It is simple to check that the composite of the maps represented by $A\leftarrow X\rightarrow B$ and $B\leftarrow Y\rightarrow C$ is represented by the pullback of $X$ and $Y$ along~$B$, along with the canonical maps to $A$ and $C$ factoring through $X$ and $Y$ respectively.

This turns out to have the property that the category of product-preserving functors $\Th(\Monoids)^\op\rightarrow\Set$ is equivalent to the category of monoids: given such a functor, we obtain a monoid $M$ by taking the image of the singleton set. The homset $\Th(\Monoids)(A,B)$ parametrises the maps which the theory of monoids forces to exist between $M^A$ and $M^B$.

Given other classes of algebraic objects $M$ whose axioms involve merely maps from $M^A$ to $M^B$, we can follow the same prescription of forming a category whose objects are finite sets and whose homsets are morphisms between the finitely generated free objects on those sets. Then product-preserving functors from that category into $\Set$ will be models of these axioms. See \cite{TTT} for a precise statement.

Detailed references for algebraic theories are \cites{Lawvere, TTT}.

Work of Badzioch \cites{Bad1, Bad2} involves studying a topological analogue of algebraic theories. However, the notions of algebraic theory that he uses are not as lax as can be. As a result, he is able to recognise generalised Eilenberg-Maclane spaces (spaces of the homotopy type of a topological abelian group) using his theories, rather than the laxer notion of an $E_\infty$-space.

One might attempt to define an algebraic theory for $E_\infty$-spaces as follows: given a set $X$, consider the free $E_\infty$-space $E(X)$ on $X$. For the sake of the present argument, we choose to model this as the classifying space of the symmetric monoidal category $\FSI_{/X}$, whose underlying category is the groupoid of sets and isomorphisms over $X$. This can be thought of as modelling the free symmetric monoidal category on generators indexed by $X$: these generators are the inclusions $\{x\}\rightarrow X$ for $x\in X$.

Giving a monoidal functor from $\FSI_{/X}$ to $\FSI_{/Y}$ essentially consists of specifying the destination of each generator: for each $x\in X$, we give a set over $Y$. Similarly to before, this is exactly the data of a diagram $X\leftarrow Z\rightarrow Y$. We call such diagrams \emph{span diagrams}. The novel feature, compared to before, is that we are not passing to isomorphism classes.

This suggests we should be interested in the weak 2-category where 0-cells are finite sets, 1-cells from $X$ to $Y$ are span diagrams $X\leftarrow Z\rightarrow Y$, and 2-cells are isomorphisms of span diagrams, in the natural sense.

By this process we obtain a natural model for monoidal quasicategories, and algebras in them. This is equivalent to Lurie's theory, and so equivalent to the older theory of $E_\infty$-monoids in spaces. We then study the analogous theory for semirings, producing a model for $E_\infty$-semirings. We also study grouplikeness, and so obtain a model for $E_\infty$-groups, and $E_\infty$-ring spaces.

\subsection{Applications}

Firstly, previous models have tended to put the additive and multiplicative structure of an $E_\infty$ ring space on very different footings. The classical topological operad pair uses the little discs operad for the additive structure and the linear isometries operad for the multiplicative structure. These are very different operads: there is no map between them, and so any comparison between their algebras discards all the geometric significance.

In this regard, the various categories of connective strictly commutative ring spectra fare no better: the additive structure is encoded by the spectrum data, but the multiplicative structure is given explicitly.

Even Lurie's handling of the matter (representing an $E_\infty$-algebra as an algebra object in the quasicategory of algebras in the quasicategory of spaces) still has two distinct stages. Even if the machinery used at each stage is identical, the input for the first stage is of a different kind to the input for the second.

On the other hand, our model treats the additive and multiplicative structures analogously and simultaneously, and so one can neatly perform constructions which mix them. One application, described in section \ref{components-and-units}, is a pleasant construction of $\GL_1$ of an $E_\infty$ ring space: this should be an $E_\infty$-monoid, thought of additively, but the old models can only naturally give you one with an action of the multiplicative $E_\infty$-operad.

Lastly, we give a tidy conceptual definition of the $K$-theory of structured categories.

\subsection{Acknowledgements}

While the whole Department of Pure Mathematics combined to make the University of Sheffield a fantastic and friendly place for me to work, I must single out several of its members. Most importantly, I would like to thank Neil Strickland, who has been an inspiring supervisor, a constant help, and a good friend while I was writing this thesis. I am extremely grateful to him; this PhD could not have been without him.

It was David Gepner who first suggested I look at quasicategories, and this turned out to be pivotal. He has offered strong encouragement to me at many points. John Greenlees spent much time struggling to enable me to come to Sheffield, for which I am very grateful; he continued to help me as a colleague and as Head of Department while there. Besides Sheffield's financial support, I am also happy to acknowledge extremely generous funding from the Sims fund of the University of Cambridge.

Fellow Sheffield mathematicians Rajender Adibhatla, Bruce Bartlett, Tony Hignett, and Panagiotis Tsaknias lived with me for two years in the House of Maths and took turns absorbing bad behaviour and bad mathematical ideas; Zacky Choo and Harry Ullman took the easy option of a year each.

I've had fascinating conversations, all too brief, with Andrew Baker, David Barnes, Bruce Bartlett, Clark Barwick, Nathan Bowler, Eugenia Cheng, John Greenlees, Ian Grojnowski, Tony Hignett, Andrey Lazarev, Tom Leinster, Jacob Lurie, John Rognes, Alex Shannon, Johann Sigurdsson, Vic Snaith, Andy Tonks, Sarah Whitehouse, Simon Willerton, and certainly many more people besides.

From among people who aren't mathematicians, I've also been blessed with too many tremendous friends to count, who have succeeded in keeping me very happy indeed. The first aiders of Cambridge and Sheffield Universities have given me a great hobby. Of the countless dozens of others, I should mention and thank in particular Nic and Ogg Cranch, John Eaton, Jo Harbour, Tom Jackman and Charlotte Squires-Parkin. 

Lastly, I should thank my parents, whose r\^ole in making it possible for me to write this thesis goes far beyond the obvious.

I thus dedicate this thesis to my father, David Donald Cranch, who has steadily and lovingly encouraged my interest in mathematics from the very beginning.

\pagebreak

\section{Preliminaries on 2-categories and quasicategories}
\label{2-cats-quasicats}

The first aim of this section is to compare various notions of 2-category, in order to match Jacob Lurie's definition of a 2-category \cite{HTT} with the classical notions.

There are several classical notions, with varying levels of strictness and laxity: as might be expected, it is simpler to construct the laxer versions, and simpler to use the stricter versions in constructions. At one end is the notion of a weak 2-category, and at the other is the notion of a strict 2-category \cite{Borceux-I}.

There is little essential difference, insofar as the work of Street and his coauthors \cites{Gordon-Power-Street, Street} (proved also in \cite{Leinster}) says that any weak 2-category can be replaced with an equivalent strict 2-category. A strict 2-category is exactly the same thing as a category enriched in categories, and we use this identification in what follows.

The second aim is to prove some basic results on quasicategories, which will be useful later on.

We use quasicategorical terminology without apology, even for arbitrary simplicial sets. Thus a 0-cell will often be called an \emph{object} of a simplicial set, and a 1-cell will often be called a \emph{morphism}.

Accordingly, when we use the word ``space'', we mean Kan complex.

\subsection{Quasicategories and $(2,1)$-categories}
\label{two-one-categories}

Let $\cC$ be a strict 2-category with all 2-cells invertible (that is, a category enriched in groupoids).

We can define its nerve in two steps. First we form a simplicial category $\bar\cC$ with $\Ob\bar\cC=\Ob\cC$, and $\bar\cC(x,y)=N\cC(x,y)$. This is a category enriched in Kan complexes, and is thus suitable for the coherent nerve construction described in \cite{HTT}, giving as our final definition that $N\cC=N^\coh(\bar\cC)$.

It is worth expanding this definition a little. We recall the definition of the simplicial categories $\fC_n$ from \cite{HTT}*{1.1.5} (we also defined them in \ref{quasicategory-comparisons}), then $N\cC_n=N^\coh(\bar\cC)_n=\sCat(\fC_n,\bar\cC)$. This lets us prove:

\begin{prop}
An $n$-cell in $N(\cC)_n$ consists of
\begin{itemize}
\item an $n+1$-tuple $X_0,\ldots,X_n$ of objects of $\cC$,
\item morphisms $f_{ij}:X_i\rightarrow X_j$ of $\cC$ for all $i<j$,
\item 2-cells $\theta_{ijk}:f_{jk}\circ f_{ij}\Rightarrow f_{ik}$ of $\cC$ for all $i<j<k$,
\end{itemize}
such that for any $i<j<k<l$ there is an identity on 2-cells:
$$\theta_{ijl}\circ(\theta_{jkl}*\id(f_{ij}))=\theta_{ikl}\circ(\id(f_{kl})*\theta_{ijk}):f_{kl}\circ f_{jk}\circ f_{ij}\Rightarrow f_{il}.\qedhere$$
\end{prop}
\begin{proof}
As $\Ob\fC_n=\{0,\ldots,n\}$, a map of simplicial categories $\fC_n\rightarrow\bar\cC$ certainly distinguishes objects $X_0,\ldots,X_n$.

The $0$-simplices of homspaces of $\fC_n(i,j)$ correspond to subsets of the interval $\{i,\ldots,j\}$ containing both $i$ and $j$, and composition is by disjoint union. Thus they are generated under composition by the minimal subsets $\{i,j\}$. These give us the morphisms $f_{ij}:X_i\rightarrow X_j$.

The $1$-simplices of homspaces of $\fC_n(i,j)$ correspond to (the opposites of) inclusions of pairs of subsets of $\{i,\ldots,j\}$ containing both $i$ and $j$. These are generated by inclusions $\{i,k,j\}\leftarrow\{i,j\}$ under horizontal and vertical composition, providing the maps $\theta_{ijk}:f_{jk}\circ f_{ij}\Rightarrow f_{ik}$ of $\cC$.

The interchange law for horizontal and vertical composition gives us the specified identity, arising from the agreement of the composite inclusions
\begin{align*}
\{i,k,l,j\}\longleftarrow&\{i,k,j\}\longleftarrow\{i,j\},\quad\text{and}\\
\{i,k,l,j\}\longleftarrow&\{i,l,j\}\longleftarrow\{i,j\}.
\end{align*}
This identity generates all $2$-cells in $\fC_n(i,j)$, under composition.

As $\bar\cC(i,j)$ is the nerve of a groupoid, a map $\fC_n(i,j)$ is uniquely specified by its effect on the 1-skeleton, so there are no further data or identities.
\end{proof}

We refer to the identity as the \emph{compatibility condition}, and since 2-cells are invertible we can write it graphically: it says that the pasting of the following diagram is the identity 2-cell.
\begin{displaymath}
\xymatrix{X_0\ar[r]\ar@/^16pt/[rr]_{\Uparrow}\ar@/_26pt/[rrr]^{\Downarrow}\ar@/^26pt/[rrr]_{\Uparrow}&X_1\ar[r]\ar@/_16pt/[rr]^{\Downarrow}&X_2\ar[r]&X_3}
\end{displaymath}

We also get the following basic coherence result, which is obvious from the description above.
\begin{prop}
A lax functor $F:\cC\rightarrow\cD$ between bicategories with all 2-cells invertible yields (via passing to strict 2-categories) a map of quasicategories $N(F):N(\cC)\rightarrow N(\cD)$ between their nerves.
\end{prop}
\begin{proof}
We can replace $F$ with an equivalent functor of strict 2-categories, and then use the naturality of the nerve construction considered above.
\end{proof}

Also, this construction agrees with the construction of the nerve of a category.
\begin{prop}
Let $\cC$ be a category, regarded as a bicategory with only identities for 2-cells. Then $N(\cC)$ is the ordinary nerve of $\cC$.\qedhere
\end{prop}

Now, the nerve $N(\cC)$ should be thought of as a model for $\cC$ in the world of quasicategories. Thus, we should expect it to be a (2,1)-category in the sense discussed above. This means that all all extensions of maps $\Lambda^n_k\rightarrow\cC$ to $k$-cells are unique for $n\geq 3$: its cells in degrees 3 and over are determined by those in lower degrees. The facts support our intuition:
\begin{prop}
\label{nerves-of-bicats}
The nerve $N(\cC)$ is a $(2,1)$-category.
\end{prop}
\begin{proof}
Suppose given an inner horn inclusion $\Lambda^n_k\rightarrow N(\cC)$ for $n\geq 3$, and $0<k<n$. We can recover all the 1-cells from this data: the 1-cell $X_i\rightarrow X_j$ for $i<j$ will be given by the face numbered $\alpha$ for any $\alpha\notin\{i,j,k\}$.

If $n=3$, then without loss of generality, $k=1$ (as the case $k=2$ is dealt with in a symmetric manner). We then have the following diagram:
\begin{displaymath}
\xymatrix{X_0\ar[r]\ar@/^16pt/[rr]_{\Uparrow}\ar@/_26pt/[rrr]^{\Downarrow}&X_1\ar[r]\ar@/_16pt/[rr]^{\Downarrow}&X_2\ar[r]&X_3}
\end{displaymath}
This leaves us just missing the 2-cell $\theta_{023}:f_{23}\circ f_{02}\Rightarrow f_{03}$. But, since all 2-cells are invertible, we can take this to be the composite of all the 2-cells in the diagram above. In symbols, we define
$$\theta_{023}=\theta_{013}\circ(\theta_{123}*\id(f_{01}))\circ(\id(f_{23})*\theta_{012}^{-1}),$$
and this clearly fulfils the compatibility condition. This choice is clearly forced, arising as it does by solving the compatibility condition for $\theta_{023}$, and this means the extension is unique.

If $n\geq 4$, then all 2-cells are determined uniquely (indeed, $\theta_{hij}$ will be defined by face $\alpha$, for any $\alpha\notin\{h,i,j,k\}$). However, if $n=4$ there are some compatibility conditions which are not forced by the faces, and we must check that they hold.

For calculations, we omit the identity parts of our 2-cells. Then all composites are vertical composites, so we do not bother writing the $\circ$. There are five compatibility conditions coming from the faces:
\begin{align*}
\theta_{134}\theta_{123}&=\theta_{124}\theta_{234}\tag{face 0}\\
\theta_{034}\theta_{023}&=\theta_{024}\theta_{234}\tag{face 1}\\
\theta_{034}\theta_{013}&=\theta_{014}\theta_{134}\tag{face 2}\\
\theta_{024}\theta_{012}&=\theta_{014}\theta_{124}\tag{face 3}\\
\theta_{023}\theta_{012}&=\theta_{013}\theta_{123}\tag{face 4}\\
\end{align*}
Also, $\theta_{012}$ and $\theta_{234}$ commute. We can see this using the interchange law:
\begin{align*}
  \theta_{012}\theta_{234}
&=(\id(f_{24})*\theta_{012})\circ(\theta_{234}*\id(f_{02}))\\
&=(\id(f_{24})\circ\theta_{234})*(\theta_{012}\circ\id(f_{02}))
 =\theta_{234}\theta_{012}.
\end{align*}

For horn inclusions $\Lambda^4_1\rightarrow N(\cC)$, we have all coherence conditions except the one arising from face 1, and must show that from the others. But we have:
\begin{align*}
\theta_{034}\theta_{023}
&=(\theta_{014}\theta_{134}\theta_{013}^{-1})(\theta_{013}\theta_{123}\theta_{012}^{-1})\qquad\text{(faces 2 and 4)}\\
&=\theta_{014}\theta_{134}\theta_{123}\theta_{012}^{-1}\\
&=\theta_{014}(\theta_{124}\theta_{234})\theta_{012}^{-1}\qquad\text{(face 0)}\\
&=(\theta_{024}\theta_{012})\theta_{234}\theta_{012}^{-1}\qquad\text{(face 3)}\\
&=\theta_{024}\theta_{234}\qquad\text{(since $\theta_{012}$ and $\theta_{234}$ commute).}
\end{align*}

For horn inclusions $\Lambda^4_2\rightarrow N(\cC)$, we have all coherence conditions except the one from face 2. Similarly, we have:
\begin{align*}
\theta_{034}\theta_{013}
&=(\theta_{024}\theta_{234}\theta_{023}^{-1})(\theta_{023}\theta_{012}\theta_{123}^{-1})\qquad\text{(faces 1 and 4)}\\
&=\theta_{024}\theta_{012}\theta_{234}\theta_{123}^{-1}\qquad\text{(since $\theta_{012}$ and $\theta_{234}$ commute)}\\
&=(\theta_{014}\theta_{124})\theta_{234}\theta_{123}^{-1}\qquad\text{(face 3)}\\
&=\theta_{014}(\theta_{134}\theta_{123})\theta_{123}^{-1}\qquad\text{(face 0)}\\
&=\theta_{014}\theta_{134}
\end{align*}

Horn inclusions $\Lambda^4_3\rightarrow N(\cC)$ can be dealt with by an argument symmetric to that used for horn inclusions $\Lambda^4_1\rightarrow\cC$.

The fact that all structure is determined means that the extension is unique.

If $n\geq 5$, then nothing need be checked: the compatibility conditions on $X_g$, $X_h$, $X_i$ and $X_j$ will be fulfilled by face $\alpha$, for any $\alpha\notin\{g,h,i,j,k\}$.
\end{proof}

Using this nerve construction, in the sequel we shall abuse terminology systematically, and confuse a strict 2-category with its nerve $(2,1)$-category.

\subsection{Fibrations and extension properties of $(n,1)$-categories}
\label{n-1-categories}

In this section we prove some properties of Lurie's model for $(n,1)$-categories, from \cite{HTT}*{subsection 2.3.4}: these are those $(\infty,1)$-categories which admit all inner horn extensions $\Lambda^m_k$ uniquely where $m>n$.

It follows immediately from the definition \cite{HTT}*{2.3.4.9} that an $(n,1)$-category has at most one extension along $\partial\Delta^m\rightarrow\Delta^m$ for $m>n$; here's a strengthening of that statement:

\begin{prop}
\label{partial-delta-n-cat}
An $(n,1)$-category $\cC$ has unique liftings for $\partial\Delta^m\rightarrow\Delta^m$ where $m\geq n+2$.
\end{prop}
\begin{proof}
We can restrict the map $\partial\Delta^m\rightarrow\cC$ to a map $\Lambda^m_1\rightarrow\cC$, and lift that uniquely to a map $\Delta^m\rightarrow\cC$. This is the only candidate for a lifting; we must prove that it is compatible with the given map on all of $\partial\Delta^m$: that is, show that it agrees on the $1$st face.

But these two $m-1$-cells certainly agree on the boundary of the $1$st face (which is isomorphic to $\partial\Delta^{m-1}$) and thus agree.
\end{proof}

In a similar vein is this:
\begin{prop}
\label{outer-horn-n-cat}
An $(n,1)$-category $\cC$ has unique liftings for outer horns $\Lambda^m_0\rightarrow\Delta^m$ and $\Lambda^m_m\rightarrow\Delta^m$ where $m>n+2$.
\end{prop}
\begin{proof}
We can uniquely extend a map $\Lambda^m_0\rightarrow\cC$ to a map $\partial\Delta^m\rightarrow\cC$ using Proposition \ref{partial-delta-n-cat} on the $0$th face. Then we can uniquely extend that to a map $\Delta^m\rightarrow\cC$ using Proposition \ref{partial-delta-n-cat} again.

The case of $\Lambda^m_m$ is symmetrical.
\end{proof}

The special case of ordinary categories will be of utility later:
\begin{prop}
 \label{nerve-outer-horns}
 The nerve of a category $N\cC$ has unique liftings for outer horns $\Lambda^n_0$ and $\Lambda^n_n$ whenever $n\geq 4$.\qed
\end{prop}

The following proposition reduces the work necessary to show that a map of $(n,1)$-categories is an acyclic Kan fibration:
\begin{prop}
\label{acyclic-kan}
A functor $\cC\rightarrow\cD$ of $(n,1)$-categories automatically has the right lifting property with respect to the maps $\partial\Delta^m\rightarrow\Delta^m$ for $m\geq n+2$.
\end{prop}
\begin{proof}
Proposition \ref{partial-delta-n-cat} gives a map $\Delta^m\rightarrow\cC$, and by \cite{HTT}*{2.3.4.9}, this is consistent with the given map $\Delta^m\rightarrow\cD$.
\end{proof}

We can say useful things about inner fibrations.
Let $F:\cC\rightarrow\cD$ be a functor between $(n,1)$-categories. We have the following simple criterion for being an inner fibration:
\begin{prop}
\label{inner-fibs}
The functor $F$ is an inner fibration if and only if it has the right lifting property for inner horns $\Lambda^m_k\rightarrow\Delta^m$ for $0<k<m\leq n$.\qed
\end{prop}

In particular, this gives the following simple criterion for $(2,1)$-categories: $F$ is an inner fibration if, for every pair of diagrams
\begin{displaymath}
\vcenter{
\xymatrix{&y\ar[dr]^h&\\
          x\ar[ur]^f&&z}}
\text{in $\cC$,\hskip 1cm and}
\vcenter{
\xymatrix{&y'\ar[dr]^{h'}\Ar^{k'}[d]&\\
          x'\ar[ur]^{f'}\ar[rr]_{g'}&&z'}}
\text{in $\cD$,}
\end{displaymath}
such that $F(f)=f'$ and $F(h)=h'$, there is a 1-cell $g:x\rightarrow z$ and 2-cell $k:h\circ f\Rightarrow g$ such that $F(g)=g'$ and $F(k)=k'$.

We now switch our attention to the more intricate notion of a cartesian fibration. These are analogues of the classical notion of a Grothendieck fibration of categories. They are morphisms of simplicial sets which describe a family of quasicategories varying in a contravariant functorial manner over a base quasicategory. Following Lurie \cite{HTT}, we make the following definition:

\begin{defn}
\label{defn-cartesian-fibration}
A \emph{cartesian fibration} $p:\cC\rightarrow\cD$ of quasicategories is a functor which is both an inner fibration and is such that, for every 1-morphism $f:x\rightarrow y$ (meaning a 1-cell $f\in\cD_1$ with $d_0f=x$ and $d_1f=y$) and every lift $\tilde y$ of $y$ to $\cC$ (meaning an 0-cell $\tilde y\in\cC_0$ with $p(\tilde y)=y$), there is a $p$-cartesian morphism $\tilde f$ in $\cC$ which maps to $f$ under $p$.

In turn, a \emph{$p$-cartesian morphism} $f:a\rightarrow b\in\cC_1$ is one such that the natural map
$$L_f:\cC_{/f}\longrightarrow \cC_{/y}\timeso{\cD_{/y}}\cD_{/f},$$
where $y=f(b)$, is an acyclic Kan fibration.

There is a dual notion of a \emph{cocartesian morphism} and a \emph{cocartesian fibration}: a cocartesian fibration describes a family of quasicategories varying covariantly functorially over a base quasicategory. Given $p:\cC\rightarrow\cD$, a cocartesian morphism in $\cC$ is a cartesian morphism for $p^\op:\cC^\op\rightarrow\cD^\op$, and $p$ is a cocartesian fibration if $p^\op$ is a cartesian fibration.

\end{defn}

Lurie proves that overcategories of $(n,1)$-categories are $(n,1)$-categories \cite{HTT}*{Lemma 1.2.17.10}. The class of $(n,1)$-categories is not closed under fibre products. But the following lemma does most of the work for us:

\begin{prop}
The class of simplicial sets which are the coskeleton of their $k$-skeleton is closed under all limits.
\end{prop}
\begin{proof}
The $n$-skeleton functor $\skel_n$ visibly preserves limits, and the $n$-coskeleton functor preserves limits since it is right adjoint to $\skel_n$. Given this, this category is closed under limits.
\end{proof}

Indeed, more is true:
\begin{prop}
For any $0\leq i\leq n$, the class of simplicial sets with unique liftings for maps $\Lambda^n_i\rightarrow\Delta^n$ is closed under limits.
\end{prop}
\begin{proof}

The class of such simplicial sets is closed under products since a lifting for the product is just the product of liftings of the factors; we'll verify it for fibre products too.

So if $X$, $Y$ and $Z$ are simplicial sets with unique extensions for the map $\Lambda^n_i\rightarrow\Delta^n$, then I claim that $X\timeso{Z}Y$ has unique liftings for it too. Indeed, a map $\Lambda^n_i\rightarrow X\timeso{Z}Y$ consists of maps $\Lambda^n_i\rightarrow X,Y$ whose composites with the maps from $X$ and $Y$ to $Z$ agree.

These extend uniquely to maps $\Delta^n\rightarrow X,Y$. Their composites with the maps to $Z$ are both extensions of our map $\Lambda^n_i\rightarrow Z$. But such extensions are unique, and so they agree. Thus these maps assemble to a unique extension $\Delta^n\rightarrow X\timeso{Z}Y$.
\end{proof}

These combine to prove the following:
\begin{prop}
Let $p:\cC\rightarrow\cD$ be a map between $(n,1)$-categories. A morphism $f:x\rightarrow y$ in $\cC_1$ is $p$-cartesian if and only if the morphism
$$\cC_{/f}\longrightarrow\cC_{/y}\timeso{\cD_{/py}}\cD_{/pf}$$
has the right lifting property for all maps $\partial\Delta^m\rightarrow\Delta m$ for $m\leq n+1$.
\end{prop}
\begin{proof}
By the results above, both sides are $(n,1)$-categories; we thus apply Proposition \ref{acyclic-kan} to show the higher lifting conditions are automatic.
\end{proof}

\subsection{Over-and-overcategories}

We prove some results about Joyal's construction of overcategories, described in \cite{HTT}.

Firstly, we study the naturality of overcategories. Recall that, given a map $f:K\rightarrow\cC$ of simplicial sets, the overcategory $\cC_{/f}$ is defined up to isomorphism by the property that $\sSet(X,\cC_{/f})$ is defined to be the collection of diagrams
\begin{displaymath}
\xymatrix{K\ar[r]^i\ar[dr]_{f}&X\star K\ar[d]\\
                              &\cC,}
\end{displaymath}
where $i$ is the natural inclusion of $K$ into $X\star K$. This is indeed a valid definition of $\cC_{/f}$; we're describing a functor $\sSet^\op\rightarrow\Set$, which takes colimits to limits. Thus we get a simplicial set: the $n$-simplices are obtained by evaluating on $\Delta^n$.

We need to study iterating the construction of overcategories.

Given a map $g:L\rightarrow\cC_{/f}$ (which corresponds to a map which we denote $\tilde g:L\star K\rightarrow\cC$), we can then form $(\cC_{/f})_{/g}$.

Maps $\sSet(X,(\cC_{/f})_{/g})$ correspond to diagrams
\begin{displaymath}
\xymatrix{L\ar[r]^i\ar[dr]_{g}&X\star L\ar[d]\\
                              &\cC_{/f}}
\end{displaymath}
which in turn correspond to diagrams
\begin{displaymath}
\xymatrix{K\ar[r]^i\ar[dr]_f&L\star K\ar[d]_{\tilde g}\ar[r]^i&X\star L\star K\ar[dl]\\
                          &\cC}
\end{displaymath}
However, the left-hand triangle provides no information, so we have proved:

\begin{prop}
\label{over-and-over}
An overcategory of an overcategory is an example of an overcategory. What we mean is that
$(\cC_{/f})_{/g}\isom \cC_{/\tilde g}$.\qed
\end{prop}

In particular, this means that limits in overcategories $\cC_{/f}$ are just a special case of limits in $\cC$: a limit of $g$ in $\cC_{/f}$ is just a limit of $\tilde g$ in $\cC$.

\subsection{Overcategories and limits of simplicial sets}

In this section we study the relationship between Joyal's over construction (described in \cite{HTT}*{Lemma 1.2.9.2}), and limits of simplicial sets. The result is that taking overcategories commutes with taking limits of simplicial sets, in the following sense:
\begin{prop}
\label{limits-and-overs}
  Suppose we have an (ordinary) finite category $D$, to be thought of as a diagram category, and a diagram $F:D\rightarrow\sSet$. Suppose also that we have a cone on it: a simplicial set $K$ and a natural transformation $\theta:K\Rightarrow F$ to $F$ from the constant functor at $K$.

  We then get a map $\bar\theta:K\rightarrow\limit F$ from $K$ to the limit of the diagram $F$.

  We then have that the over construction commutes with limits in the sense that
  $$(\limit F)_{/\bar\theta}\isom{\limit}_x(F(x)_{/\theta_x}).$$
\end{prop}
\begin{proof}
We consider maps from a fixed simplicial set $Y$; it's then just a straightforward check:
\begin{align*}
       \sSet(Y,(\limit F)_{/\bar\theta}) 
\isom &\sSet_\theta(Y\star K,\limit F)\\
\isom &{\limit}_{x\in D}\sSet_{\theta_x}(Y\star K,F(x))\\
\isom &{\limit}_{x\in D}\sSet(Y,F(x)_{/\theta_x})\\
\isom &\sSet(Y,{\limit}_{x\in D}(F(x)_{/\theta_x}).\qedhere
\end{align*}
\end{proof}

We continue this analysis to derive a corresponding result for finite products and cartesian morphisms:
\begin{prop}
\label{prods-and-carts}
 If $q_1:\cC_1\rightarrow\cD_1$ and $q_2:\cC_2\rightarrow\cD_2$ are maps of quasicategories, then, defining $q=q_1\times q_2:\cC_1\times\cC_2\rightarrow\cD_1\times\cD_2$, the $q$-cartesian morphisms (as defined in subsection \ref{defn-cartesian-fibration}) are exactly the products of $q_1$-cartesian morphisms and $q_2$-cartesian morphisms.
\end{prop}
\begin{proof}
 We must relate a pullback of overcategories of products to a product of pullbacks of overcategories. The pullbacks and products commute, as usual; Proposition \ref{limits-and-overs} provides that the formation of products and of overcategories commute.
\end{proof}

\subsection{Functoriality of overcategories}

In this section, we detour for a moment to show how an edge $f:x\rightarrow y$ in a quasicategory $\cC$ yields a morphism of overcategories $\cC_{/x}\rightarrow\cC_{/y}$. This is helpful for understanding overcategories better.

Given a quasicategory $\cC$, consider the map $p:\cCD\rightarrow\cC$ induced by evaluation of the terminal vertex of $\Delta^1$.

I claim firstly that this stores all the overcategories:
\begin{prop}
The fibres $p^{-1}(x)$ of $p$ are equivalent to $\cC_{/x}$.
\end{prop}
\begin{proof}
Maps $K\rightarrow p^{-1}(x)$ are maps $K\times\Delta^1\rightarrow\cC$ sending $K\times 1$ to $x$, or equivalently are maps $(K\times\Delta)/(K\times 1)\rightarrow\cC$ pointed at $x$. But there is a map $(K\times\Delta^1)/(K\times 1)\rightarrow K\star 1$, so there is a map $\cC_{/x}\rightarrow p^{-1}(x)$.

Moreover, since this map $(K\times\Delta^1)/(K\times 1)\rightarrow K\star 1$ is a strong deformation retract. We immediately get equivalences of homspaces. The required result then follows from the equivalence of simplicial categories and quasicategories. 
\end{proof}

Now, this map classifies functoriality of overcategories. In order to demonstrate that, we need an intermediate result:
\begin{prop}
\label{cylinders-cat-anodyne}
For $0<k<n$, the inclusions $(\{1\}\times\Delta^n)\cup(\Delta^1\times\Lambda^n_k)\rightarrow(\Delta^1\times\Delta^n)$ are compositions of inclusions of inner horns.
\end{prop}
\begin{proof}
We label the vertices of $\Delta^1\times\Delta^n$ as $0,1,\ldots,n,0',1',\ldots,n'$. A simplex in $\Delta^1\times\Delta^n$ is determined by its vertices, and thus the nondegenerate simplices can be written as $a_0\cdots a_pb'_0\cdots b'_q$ where $a_0<\cdots<a_p\leq b_0<\cdots<b_q$.

The left-hand side $(\{1\}\times\Delta^n)\cup(\Delta^1\times\Lambda^n_k)$ contains all the cells of $\Delta^1\times\Delta^n$ except those which involve every number from 0 to $n$ except possibly $k$, and which contain at least one unprimed vertex.

We make a series of horn extensions as follows.

First we extend over $00'1'\cdots \hat k'\cdots n'$ (where the hat denotes omission). Since every face but $01'\cdots\hat k'\cdots n'$ is present, this is an inner horn extension of shape $\Lambda^n_1\rightarrow\Delta^n$.

Next we extend over $00'1'\cdots n'$. Since every face but $01'\cdots n'$ is present, this is an inner horn extension of shape $\Lambda^{n+1}_1\rightarrow\Delta^{n+1}$. Then we extend over $011'2'\cdots n'$. Since every face but $012'\cdots n'$ is present, this is an inner horn extension of shape $\Lambda^{n+1}_2\rightarrow\Delta^{n+1}$.

Proceeding inductively, as $i$ increases from $2$ to $n$, we then extend $01\cdots ii'\cdots n'$, which is an inner horn extension of shape $\Lambda^{n+1}_i$ since it's missing only the face $01\cdots(i-1)i'\cdots n'$. This gives all the missing maximal simplices, so we're done.
\end{proof}

Using that, here is the result we were aiming for:
\begin{prop}
\label{overcats-are-functorial}
The map $p$ is a cocartesian fibration.
\end{prop}
\begin{proof}
To show that $p$ is an inner Kan fibration, we must provide extensions as follows:
\begin{displaymath}
\xymatrix{\Lambda^n_k\ar[r]\ar[d]&\Delta^n\ar[d]\dar[dl]\\
  \cCD\ar[r]&\cC,}
\end{displaymath}
or, equivalently,
\begin{displaymath}
\xymatrix{\left(\Lambda^n_k\times\Delta^1\right)\cup\left(\Delta^{n}\times\{1\}\right)\ar[r]\ar[dr]&\Delta^n\times\Delta^1\dar[d]\\
  &\cC.}
\end{displaymath}
This lifting problem is solved by Proposition \ref{cylinders-cat-anodyne}.

Let $f:x\rightarrow y$ be an edge in $\cC$, and let $\alpha:a\rightarrow x$ be a vertex of $\cCD$. We need a cocartesian lift $\phi:\Delta^1\rightarrow\cCD$. We use the following:
\begin{displaymath}
\xymatrix{
  a\ar[d]_\alpha\ar[r]\ar[dr]&a\ar[d]\\
  x\ar[r]_f                  &y,}
\end{displaymath}
where the bottom-left cell is a composition of $\alpha$ and $f$, and the top-right cell is a degeneracy of a face of the bottom-left cell.

Now, given this choice of $\phi$, we must show that the map $\cCD_{\phi/}\rightarrow\cCD_{\alpha/}\timeso{\cC_{x/}}\cC_{f/}$ is acyclic Kan.

A map $\Delta^0\rightarrow\cCD_{\alpha/}\timeso{\cC_{x/}}\cC_{f/}$ can be rewritten using the structural adjunctions as a diagram
$$\left((\Delta^0\star\Delta^0)\times\Delta^1\right)\cup\left((\Delta^1\star\Delta^0)\times\{1\}\right)\longrightarrow\cC$$
restricting to $\alpha$ and $f$ on the left-hand factors of the join.

Writing $A_0=a$, $X_0=x$, and $X_1=y$, this gives us the bottom and back faces of the following diagram:
\begin{displaymath}
\xymatrix{
A_0\dar[dr]\ar[dd]_\alpha\ar[rr]&&A_2\ar[dd]\\
&A_1\dar[ur]\dar[dd]&\\
X_0\ar[dr]_f\ar'[r][rr]&&X_2\\
&X_1\ar[ur]&\\}
\end{displaymath}
To extend these data to a diagram $\Delta^0\rightarrow\cCD_{\phi/}$, we can use Proposition \ref{cylinders-cat-anodyne} with $k=1$, $n=2$.

We solve the other lifting problems in a uniform manner.

For $n\geq 1$, a diagram
\begin{displaymath}
\xymatrix{\partial\Delta^n\ar[r]\ar[d]&\Delta^n\ar[d]\\
\cCD_{\phi/}\ar[r]&\cCD_{\alpha/}\timeso{\cC_{x/}}\cC_{f/}}
\end{displaymath}
gives us a map
$$\left((\Delta^1\star\partial\Delta^n)\times\Delta^1\right)\cup\left((\{0\}\star\Delta^n)\times\Delta^1\right)\cup\left((\Delta^1\star\Delta^n)\times\{1\}\right)\longrightarrow\cC,$$
or equivalently
$$\left(\Lambda^{2+n}_0\times\Delta^1\right)\cup\left(\Delta^{2+n}\times\{1\}\right)\longrightarrow\cC.$$
extending $\phi$ on the left-hand term.

However, because of the degeneracy in the definition of $\phi$, this is equivalent to a map
$$\left(\Lambda^{2+n}_1\times\Delta^1\right)\cup\left(\Delta^{2+n}\times\{1\}\right)\longrightarrow\cC,$$
which extends straight away using Proposition \ref{cylinders-cat-anodyne} to the required map
$$\Delta^{2+n}\times\Delta^1\longrightarrow\cC.\hfil\qedhere$$
\end{proof}

The results of this section can of course be dualised: the map $\cCD\rightarrow\cC$ given by evaluation of the initial vertex of $\Delta^1$ has fibres which are the undercategories of $\cC$, and this map is a cartesian fibration.

\subsection{Limits in undercategories}

In this section, we show how quasicategorical limits in undercategories are related to limits in the original quasicategory.

\begin{prop}
\label{forgetting-from-undercat-preserves-limits}
Let $\cC$ be any quasicategory with limits, and let $f:D\rightarrow\cC$ be any diagram in it. The ``forgetful'' map $\cC_{D/}\rightarrow\cC$ preserves limits.
\end{prop}
\begin{proof}
Suppose we have a diagram of shape $K$ in $\cC_{D/}$. Postcomposition with the forgetful map gives a diagram $K\rightarrow\cC$, which admits a limit $1\star K\rightarrow\cC$; and the map $\cC_{/(1\star K)}\rightarrow\cC_{/K}$ is acyclic Kan.

Our diagram $K\rightarrow\cC_{D/}$ is equivalent to a diagram $D\rightarrow\cC_{/K}$. By the acyclic Kan condition, this gives us a map $D\rightarrow\cC_{/(1\star K)}$, or equivalently, $1\star K\rightarrow\cC_{D/}$.

We must merely show that this is indeed a limit: that $\cC_{D//(1\star K)}\rightarrow\cC_{D//K}$ is acyclic Kan. However, given $I\rightarrow J$ a cofibration, there is a bijective correspondence between squares of the two following sorts:
\begin{displaymath}
 \xymatrix{I\ar[d]\ar[r]&J\ar[d]\\
           \cC_{D//(1\star K)}\ar[r]&\cC_{D//K}}
\qquad\text{and}\qquad
 \xymatrix{D\star I\ar[d]\ar[r]&D\star J\ar[d]\\
           \cC_{/(1\star K)}\ar[r]&\cC_{/K}}.
\end{displaymath}
Since $\cC_{/(1\star K)}\rightarrow\cC_{/K}$ is acyclic Kan, we have a diagonal filler on the right, which gives us one on the left, too.
\end{proof}

By a straightforward dualisation, of course we also get:
\begin{prop}
\label{forgetting-from-overcat-preserves-colimits}
Let $\cC$ be any quasicategory and $f:D\rightarrow\cC$ any diagram in it. Then the forgetful map $\cC_{/D}\rightarrow\cC$ preserves colimits.
\end{prop}

Lastly, we have the following useful result:
\begin{prop}
 \label{interchange-of-limits}
 If $\cC$ is a complete quasicategory, and $f:X\times Y\rightarrow \cC$ is a diagram in $\cC$, then we have the usual interchange-of-limits isomorphisms
 $$\lim_X\lim_Y\isom\lim_{X\times Y}\isom\lim_Y\lim_X.$$
\end{prop}
\begin{proof}
 This follows from \cite{HTT}*{Prop 4.2.2.7}.
\end{proof}

\subsection{Limits and colimits in $\Spaces$}

This subsection serves two purposes. Firstly, it gives a straightforward construction of homotopy pullbacks in $\Spaces$. Then it gives a couple of basic properties of colimits in the quasicategory of spaces; they are both recognisable results in the discrete case, where they reduce to results about ordinary colimits in the ordinary category of sets.

\begin{defn}
\label{homotopy-pullbacks-spaces}
We use the following a natural model for the quasicategorical pullback: we write $E3$ for the standard contractible simplicial set on three points $l$, $m$ and $r$ (with one $n$-simplex for each $(n+1)$-tuple of vertices). Then we define
$$\cC_1\times^h_\cE\cC_2 = (\cC_1\times\cC_2)\times_{(\cE\times\cE)}\Map(E3,\cE).$$
Here the morphism $\Map(E3,\cE)\rightarrow \cE\times\cE$ is given by evaluation on $l$ and $r$.
\end{defn}

We also write down the structure maps of the limiting cone:
\begin{displaymath}
\xymatrix{\cC_1\times^h_\cE\cC_2\ar[d]_{p_2}\ar[dr]^f\ar[r]^{p_1}&\cC_1\ar[d]\\
          \cC_2\ar[r]&\cE.}
\end{displaymath}
The maps $p_1$ and $p_2$ are the evident projections, and $f$ is induced by the map $\Map(E3,\cE)\rightarrow\cE$ given by evaluation at $m$. The homotopies between the composites $\cC_1\times^h_\cE\cC_2\rightarrow\cE$ are induced by the equivalences $l\isom m$ and $m\isom r$ in $E3$. This is of course merely a simplicial variant of the standard topological construction of the homotopy pullback \cite{Hirschhorn}*{18.1.7}.

Now we turn to colimits. We start with an easy observation:
\begin{prop}
\label{colims-prods-commute}
Colimits in $\Spaces$ commute with products.
\end{prop}
\begin{proof}
We invoke \cite{HTT}*{Corollary 4.2.4.8} to show that it suffices to do this in the simplicial category $\Spaces^\Delta$, for arbitrary coproducts and homotopy pushouts.

Both of these are easy checks. Indeed,
$$\coprod_{a\in A}(Z\times X_a)\isom Z\times\coprod_{a\in A}X_a.$$
Also, since the formation of mapping cylinders commutes with products, homotopy coequalisers do also.
\end{proof}

This allows us to prove:
\begin{prop}
\label{products-and-colimits}
Let $F:K\rightarrow\Spaces$ and $G:L\rightarrow\Spaces$ be diagrams in the quasicategory of spaces. Then the diagram
$$F\times G:K\times L\longrightarrow\Spaces\times\Spaces\stackrel{\mathrm{prod}}{\longrightarrow}\Spaces$$
has colimit given by $\colim(F\times G)=\colim(F)\times\colim(G)$.
\end{prop}
\begin{proof}
This is an easy calculation, using \ref{colims-prods-commute} twice:
\begin{align*}
\colim(F\times G) &= \colim_{k\in K}\colim_{l\in L}\left(F(k)\times G(l)\right) \\
                  &= \colim_{k\in K}\left(F(k)\times\colim_{l\in L}G(l)\right) \\
                  &= \colim_{k\in K}\left(F(k)\times\colim(G)\right) \\
                  &= \left(\colim_{k\in K}F(k)\right)\times\colim{G} \\
                  &= \colim(F)\times\colim(G).\qedhere
\end{align*}
\end{proof}

\subsection{Mapping cylinders in quasicategories}
\label{mapping-cylinders-quasicategories}

As discussed above, cocartesian fibrations over $\cC$ are equivalent to functors $\cC\rightarrow\Cinfty$. In the case where $\cC=\Delta^1$, we shall later have need of a direct way of replacing functors between quasicategories $F:\cA\rightarrow\cB$ with cocartesian fibrations $\cE\rightarrow\Delta^1$. This approach is, in fact, a special case of Lurie's \emph{relative nerve} construction \cite{HTT}*{3.2.5}; however it may nevertheless be helpful to have a self-contained account of it.

Firstly, we note that maps $\Delta^n\rightarrow\Delta^1$ are classified by the preimages of the two vertices of $\Delta^1$. Thus we write $\Delta^{I\sqcup J}$ for a simplex with the implied map to $\Delta^1$ sending $I$ to $0$ and $J$ to $1$.

We define our model $p:\cE\rightarrow\Delta^1$ by giving that
\begin{displaymath}
\left\{\text{maps}\quad\vcenter{\xymatrix{\Delta^{I\sqcup J}\ar[r]\ar[dr]&\cE\ar[d]\\&\Delta^1}}\right\}
= \left\{\text{diagrams}\quad\vcenter{\xymatrix{\Delta^I\ar[d]\ar[r]^i&\Delta^{I\sqcup J}\ar[d]\\\cA\ar[r]_F&\cB}}\right\},
\end{displaymath}
where the map $i$ is that induced by the evident inclusion $I\rightarrow I\sqcup J$.

We now show by parts that this serves for us. To start with, it is straightforward to check that the preimages of the vertices are isomorphic to $\cA$ and $\cB$ respectively.

\begin{prop}
\label{mapping-cylinders-1}
The map $p:\cE\rightarrow\Delta^1$ is an inner fibration.
\end{prop}
\begin{proof}
We need to show that an inner horn $\Lambda^{I\sqcup J}_k\rightarrow\cE$ extends to a full simplex $\Delta^{I\sqcup J}\rightarrow\cE$. Since the preimages of both vertices are quasicategories,  we need only concern ourselves with the case where $I$ and $J$ are both nonempty.

In either case, the faces we have include a full map $\Delta^I\rightarrow\cA$; this merely leaves us with an inner horn extension $\Delta^{I\sqcup J}\rightarrow\cB$, which is possible as since $\cB$ is a quasicategory.
\end{proof}

\begin{prop}
\label{mapping-cylinders-2}
For any element $a\in\cA_0$, there is a $p$-cocartesian morphism of $\cE_1$ whose 0th vertex is $a$, and which lies over the nontrivial 1-cell of $\Delta^1$.
\end{prop}
\begin{proof}
We chose the 1-cell $\alpha$ given by
\begin{displaymath}
\xymatrix{\Delta^0\ar[r]^0\ar[d]_a&\Delta^1\ar[d]^a\\
 \cA\ar[r]_F&\cB.}
\end{displaymath}

Now we go on to show that this is indeed $p$-cocartesian: that the morphism
$$\cE_{\alpha/}\longrightarrow \cE_{a/}\timeso{\cE}\cB$$
is acyclic Kan.

A diagram 
\begin{displaymath}
\xymatrix{\partial\Delta^n\ar[d]\ar[r]&\Delta^n\ar[d]\\
\cE_{\alpha/}\ar[r]&\cE_{a/}\timeso{\cE}\cB}
\end{displaymath}
unravels to give a diagram
\begin{displaymath}
\xymatrix{&1\star(\emptyset\star\partial\Delta^n)\ar[dl]\ar[dr]&\\
1\star(1\star\partial\Delta^n)\ar[dr]&&1\star(\emptyset\star\Delta^n)\ar[dl]\\
&\cE.&}
\end{displaymath}
Since these agree on the part mapping to $T$, the problem is to extend a map from $(\Delta^1\star\partial\Delta^n)\cup(1\star\Delta^n)\isom\Lambda^{2+n}_0$ to a map from $\Delta^{2+n}$. However, the edge between the first two vertices is a degeneracy, so the given map factors through $\Delta^{1+n}$ and can be extended to $\Delta^{2+n}$ via the map $\Delta^{2+n}\rightarrow\Delta^{1+n}$ which collapses the first two vertices.
\end{proof}

We can now prove:
\begin{prop}
\label{mapping-cylinders}
The map $p:\cE\rightarrow\Delta^1$ is an cocartesian fibration.
\end{prop}
\begin{proof}
We have just shown in Proposition \ref{mapping-cylinders-1} that $p$ is an inner fibration; there remains the question of cocartesian lifts. We only need to provide cocartesian lifts over nonidentity cells of $\Delta^1$; and Proposition \ref{mapping-cylinders-2} does this.
\end{proof}

\subsection{Functoriality of limits and colimits}

In this subsection we address the question of how colimits and limits in a quasicategory $\cC$ are functorial in the diagrams. We treat only colimits in detail; the case of limits is dual.

We also restrict ourselves to finite colimits (that is, colimits of diagrams given by a finite simplicial set). Actually some restriction is essential: for set-theoretic reasons we must restrict ourselves to small colimits, and in any particular situation we must restrict to some collection of limits that we can guarantee will exist. But the use of finite colimits is an arbitrary choice, governed by our later applications.

For example, we might want to build a functor $(\FinCatInfty)_{/\cC}\rightarrow\cC$ (where $\FinCatInfty$ is the quasicategory of finite quasicategories) taking finite diagrams in $\cC$ to their colimits. We can do this, although it turns out to be not the ideal approach to the situation.

An object of $(\FinCatInfty)_{/\cC}$ is a functor of quasicategories $D_0\rightarrow\cC$, and a colimit is a left Kan extension to a functor $(D_0\star *)\rightarrow\cC$.

We are equipped to prove the following:
\begin{prop}
\label{1cells-between-colimits}
Let $\cC$ be a category with all finite colimits. Then a morphism between elements of $(\FinCatInfty)_{/\cC}$ yields a morphism between their colimits.
\end{prop}
\begin{proof}
A morphism of $(\FinCatInfty)_{/\cC}$ is a 2-simplex in $\Cinfty$, which is a map $\fC_2\rightarrow\Cinfty^\Delta$ describing categories $D_0$, $D_1$ and $\cC$ and maps between them. All the data is determined by the map $\fC_2(0,2)\rightarrow\Cinfty^\Delta(D_0,\cC)$, which expands to a map
$$M = (D_0\times\Delta^1)\amalg_{(D_0\times\{1\})}(D_1\times\{1\})\longrightarrow \cC,$$
from the standard mapping cylinder construction of $D_0\rightarrow D_1$, into $\cC$.

Suppose we take a left Kan extension to a functor
$$M'=((D_0\star *)\times\Delta^1)\amalg_{((D_0\star*)\times\{1\})}((D_1\star *)\times\{1\})\longrightarrow \cC,$$
from the mapping cylinder of the cocones on $D_0$ and $D_1$, into $\cC$. This is a full subcategory inclusion, so such an extension exists, by \cite{HTT}*{Corollary 4.3.2.14}.

I claim firstly that $*\times\{0\}$ is sent to $\colim D_0$. This follows immediately from the definition of a Kan extension along a full subcategory inclusion: it is the colimit of the diagram
$$M_{/(*\times\{0\})}\rightarrow\cC.$$
But the overcategory of $*\times\{0\}$ is just $D_0\times\{0\}$.

Similarly, I claim that $*\times\{1\}$ is sent to $\colim D_1$. The overcategory of $*\times\{1\}$ is the whole mapping cylinder $M$. However, the inclusion of $D_1\times\{1\}$ is cofinal. Indeed, this follows from Joyal's quasicategorical version of Quillen's Theorem A \cite{HTT}*{Theorem 4.1.3.1}. To employ Theorem A, we are required to show that for every object $x\in M_0$, the category $(D_1\times\{1\})\times_M(M_{x/})$ is weakly contractible. However, it always has an initial object, given by $(x,\id_x)$ if $x\in D_1\times\{1\}$ and $(f(x), x\rightarrow f(x))$ if $x\in D_0\times\{0\}$ (where $f$ is the functor $D_0\rightarrow D_1$).

Hence the image of $*\times\{1\}$ is $\colim D_1$, since cofinal maps preserve colimits.

Thus the image in $\cC$ of the morphism $(*\times\{0\})\rightarrow(*\times\{1\})$ provides the required morphism.
\end{proof}

The following proposition generalises this, in the same way.
\begin{prop}
\label{ncells-between-colimits}
In the situation of the preceding Proposition \ref{1cells-between-colimits}, an $n$-cell in $(\FinCatInfty)_{/\cC}$ yields an $n$-cell between their colimits.
\end{prop}
\begin{proof}
In general, an $n$-simplex of $(\FinCatInfty)_{/\cC}$ is an $(n+1)$-simplex of $\Cinfty$; it is determined by the map $\fC_n(0,n)\rightarrow\Cinfty^\Delta(D_0,\cC)$, which consists of a map to $\cC$ on the pushout of $D_0\times(\Delta^1)^n$ with all the maps
$$\left(D_0\times(\Delta^1)^{n-i}\times\{1\}^i\right) \longrightarrow \left(D_i\times(\Delta^1)^{n-i}\times\{1\}^i\right)$$
for $i$ from $1$ to $n$.

As in the proof of the above proposition, we take the left Kan extension to a map from the same construction with $D_i$ replaced by $(D_i\star *)$ wherever it occurs. As before, this extension exists, and using the methods of the previous proposition we can show that
$$*\times\{0\}^{n-i}\times\{1\}^i$$ is sent to $\colim D_i$. Since these form the vertices of an $n$-simplex in $*\times(\Delta^1)^n$, we have exhibited the required $n$-cell.
\end{proof}
Moreover, since there is a contractible space of choices of Kan extensions, these $n$-cells are unique. Restricting to subcubes of $(\Delta^1)^n$ also provides appropriate face maps between our cells.

This immediately allows us the following:
\begin{thm}
\label{colimit-functor}
Using the axiom of choice and induction on $n$-simplices, we can build a colimit functor $(\FinCatInfty)_{/\cC}\rightarrow\cC$.\qed
\end{thm}
Since it is a brutal use of the axiom of choice we shall avoid using it directly.

\section{Generalities on algebraic theories}
\label{theories-generalities}

In this section we generalise the notion of an algebraic theory, to the setting of quasicategories. We base our account mostly on that given by Borceux \cite{Borceux-II} for the classical case.

We will have need of the quasicategory $\Spaces$. By this we mean the quasicategory obtained as the coherent nerve of the simplicial category of Kan complexes (together with their mapping complexes), as is used in \cite{HTT}*{1.2.16.1}. However, there are other natural constructions of equivalent quasicategories, just as there are several natural model categories Quillen equivalent to the standard model structure on topological spaces. Of course it will not matter which is used.

We regard $\Set$ as being the full subquasicategory on the discrete spaces; this is evidently equivalent to the standard notion. We shall use the adjective \emph{discrete} frequently to describe phenomena which occur over $\Set$ rather than the whole of $\Spaces$.

We should also say, once and for all, what we mean by this:
\begin{defn}
\label{full-subquasicategory}
A \emph{full subquasicategory} of a quasicategory, is a maximal subquasicategory with its set of $0$-cells. We shall also call this \emph{$1$-full}; an $n$-full subquasicategory is a maximal subquasicategory with that particular set of $k$-cells for all $k<n$.
\end{defn}

\subsection{Theories and models}

\begin{defn}
\label{theory-defn}
An \emph{algebraic theory} is a quasicategory $T$ together with a product-preserving, essentially surjective functor $\Finop\rightarrow T$.
\end{defn}

By product-preserving, I mean ``taking finite product diagrams to finite product diagrams''; I suppose this is the standard meaning, but I am departing from tradition in making this plain.

Usually we shall abuse notation and write just $\cC$ for the algebraic theory, leaving the morphism  from $\Finop$ out of the notation.

\begin{defn}
A morphism of algebraic theories is a diagram of product-preserving functors
\begin{displaymath}
 \xymatrix{&\Finop\ar[dl]\ar[dr]&\\
           T\ar[rr]&&S.}
\end{displaymath}
\end{defn}
It follows immediately from the definition that a morphism $T\rightarrow S$ is essentially surjective.

We can define a quasicategory $\Theories$ of algebraic theories to be the $2$-full subquasicategory of $\left(\Cinfty\right)_{\Finop/}$ on the theories and morphisms of theories.

Theories are not much use without introducing a notion of model:
\begin{defn}
A \emph{model} (in spaces) of a theory $T$ is a product-preserving functor $\phi:T\rightarrow\Spaces$. The quasicategory of models of $\cC$ is the full subquasicategory $\Mod(T)$ of $\Map(T,\Spaces)$ on the product-preserving objects.

In exactly the same way, if $\cU$ is any category with all finite products, we define the quasicategory of \emph{models of $T$ in $\cU$} to be the quasicategory $\Mod(T,\cU)$ of product-preserving functors $T\rightarrow\cU$ .
\end{defn}

By abuse of notation, we write $1$ for the image of the singleton set under the functor $\Finop\rightarrow T$. We sometimes say that $\phi(1)$ is the \emph{underlying object} of the model, and that giving such a functor $\phi$ is equipping $\phi(1)$ with a \emph{$T$-structure}.

We now observe that this does indeed generalise Lawvere's original notion (which is discussed in \cite{Lawvere}). To do this, we introduce terminology for this special case:
\begin{defn}
An algebraic theory $\Finop\rightarrow T$ is \emph{discrete} if $T$ is, in fact, an ordinary category.

If a theory $T$ is discrete, we say that a model $M$ of $T$ is \emph{discrete} if it is valued in sets (regarded as a subcategory of spaces).
\end{defn}

The rationale is that, since the subquasicategory of simplicial sets on the discrete objects is equivalent to the ordinary category of sets, if $T$ is an ordinary category, then a functor (of quasicategories) $T\rightarrow\Set\subset\Spaces$ is just an ordinary functor.

So an algebraic theory in the sense of Lawvere, which is an ordinary category $T$ equipped with a product-preserving, essentially surjective functor $\Finop\rightarrow T$, is the same thing as a discrete algebraic theory in the sense defined here. Moreover, a model of an algebraic theory in the sense of Lawvere is the same thing as a discrete model of the corresponding discrete algebraic theory.

A morphism $f:T\rightarrow U$ of theories induces a functor $f^*:\Mod(T)\rightarrow\Mod(U)$ by precomposition.

\label{models-in-quasicategories}

Once could study models of a theory $T$ in quasicategories simply by using $\Mod(T,\Cinfty)$ as defined above. Usually, we will employ an equivalent but more easily manipulated definition:
\begin{defn}

A \emph{model of $\cC$ in quasicategories} is a cocartesian fibration over $T$ that is classified by a product-preserving functor. We call such cocartesian fibrations \emph{productive}.

\end{defn}

We also need to deal with maps between models:
\begin{defn}
We define the quasicategory $\Mod^\fib(T)$ of models in quasicategories of a theory $T$. This is the subquasicategory of the overcategory $(\Cinfty)_{/T}$, consisting of all those cells whose vertices are productive cocartesian fibrations over $T$, and whose edges are product-preserving functors taking cartesian morphisms to cartesian morphisms.
\end{defn}

A morphism $f:T\rightarrow U$ of theories induces a functor $f^*:\Mod^\fib(U)\rightarrow\Mod^\fib(T)$ induced by pulling back the cocartesian fibrations.  By the results of \cite{HTT}*{Section 2.4.2 and Chapter 3}, there is an equivalence between $\Mod^\fib(T)$ and the previously defined notion $\Mod(T,\Cinfty)$.

\subsection{Multisorted theories}

Occasionally, one has need to consider axioms for algebraic structures with several underlying objects, and maps between them.

Accordingly, we define:
\begin{defn}
 Let $X$ be a set. A \emph{multisorted theory} with sorts indexed by $X$ consists of a quasicategory $T$ together with a product-preserving essentially surjective functor $(\Finop)^X\rightarrow T$.

 We refer to multisorted theories with sorts indexed by $\{1,\ldots,n\}$ as being \emph{$n$-sorted theories}.

 A \emph{model} in $\cC$ of a multisorted theory $T$ with sorts indexed by $X$ is a product-preserving functor from $T$ to $\cC$.
\end{defn}

By way of trivial example, if $\Finop\rightarrow T_1,\ldots,T_n$ are theories, then the product
$$\Finop^n\longrightarrow T_1\times\cdots\times T_n$$
is the $n$-sorted theory whose models are tuples consisting of a model of each of the theories $\{T_i\}$:
$$\Mod(T)=\Mod(T_1)\times\cdots\times\Mod(T_n).$$

In the main, the basic results for single-sorted theories carry over to $n$-sorted theories as one would expect, and we shall not write them out.

We can regard all the categories of multisorted theories as forming subcategories of the quasicategory $\Cinftypp$ of quasicategories with all finite products, product-preserving functors, and homotopies between them.  In particular, $\Mod(T;\cU)=\Cinftypp(T,\cU)$.

\begin{prop}
\label{models-has-an-adjoint}
Fix a quasicategory $\cU$ with finite products. The quasifunctor $\Funpp(-,\cU):(\Cinftypp)^\op\rightarrow\Cinfty$, which assigns to each theory its category of models, has a left adjoint.
\end{prop}
\begin{proof}
Our proof proceeds by exhibiting an adjunction in detail. However, I consider that this motivating argument is considerably more enlightening. The idea is that
$$\Funpp(T,\Fun(\cC,\cU))\isom\Fun(\cC,\Funpp(T,\cU))$$
since products are computed pointwise. This means that
$$\Theories^\op(\Fun(\cC,\cU),T)\isom\Cinfty(\cC,\Mod(T;\cU)).$$
which is exactly the equivalence on homspaces required for an adjunction.

By \cite{HTT}*{Section 5.2}, an adjunction is represented by a cartesian and cocartesian fibration over $\Delta^1$.

Now the maps $\Delta^n\rightarrow\Delta^1$ are described by the preimages of the vertices: they are equivalent to decompositions $\Delta^n=\Delta^i\star\Delta^j$ where $i,j\geq-1$ and $i+j=n$. So we can define a simplicial set $\cD$ over $\Delta^1$ by giving a compatible set of homsets $\sSet_{\Delta^1}(X\star Y,\cD)$.

We define $\cD$ by letting $\sSet_{\Delta^1}(X\star Y,\cD)$ consist of maps $c:X\rightarrow\Cinfty$ and $a:Y^\op\rightarrow\Cinftypp$, together with a map $f:(X\times Y^\op)\star 1\rightarrow\Cinfty$, which are equipped with a natural equivalence with $(c\times a):X\times Y^\op\rightarrow\Cinfty$ when restricted to $X\times Y^\op$, and which send the extra point $1$ to $U$.

Writing $\pi$ for the projection $\cD\rightarrow\Delta^1$, we easily see that $\pi^{-1}(0)=\Cinfty$ and $\pi^{-1}(1)=\Cinftypp$. We must show that $\pi$ is a bicartesian fibration, to show that it represents an adjunction.

We split this into two parts. Firstly we show that $\pi$ has the inner Kan lifting property:
\begin{claim}
The morphism $\pi$ is an inner fibration.
\end{claim}
\begin{innerproof}{Proof of claim}
For greater flexibility, we index our simplices by finite linearly ordered sets in this argument.

So, given finite linearly ordered sets $I$ and $J$, and $k$ some internal element of the concatenation $I\sqcup J$ , we must provide a lifting
\begin{displaymath}
\xymatrix{\Lambda^{I\sqcup J}_k\ar[d]\ar[r]&\Delta^{I\sqcup J}\ar[d]\dar[dl]\\
  \cD\ar[r]&\Delta^1.}
\end{displaymath}

If either $I$ or $J$ have no elements, this clearly reduces to the statement that the preimages $\Cinfty$ and $(\Cinftypp)^\op$ of the endpoints of $\Delta^1$ are both quasicategories. 

Supposing otherwise, we assume without loss of generality that $k\in I$ (the case $k\in J$ is symmetrical). Observing that
$$\Lambda^{I\sqcup J}_k = (\Lambda^I_k\star\Delta^J)\cup_{(\Lambda^I_k\star\partial\Delta^J)}(\Delta^I\star\partial\Delta^J),$$
we get that a morphism $f:\Lambda^{I\sqcup J}_k\rightarrow\cD$ consists of maps $c:\Delta^I_k\rightarrow\Cinfty$, $a:(\Delta^J)^\op\rightarrow\Cinftypp$, and a map
$$\left((\Lambda^I_k\times\Delta^J)\cup_{(\Lambda^I\times\partial\Delta^J)}(\Delta^I\times\partial\Delta^J)\right)\star 1\longrightarrow\Cinfty.$$
Using \cite{JoyalTierneyBook}*{3.2.2}, we see that the inclusion $(\Lambda^I_k\times\Delta^J)\cup(\Delta^I\times\partial\Delta^J)\rightarrow\Delta^I\times\Delta^J$ is anodyne: it's a composite of horn extensions. Tracing the argument carefully (using that $k$ is not the initial object of $I$) we see that no horn extensions of shape $\Lambda^r_0\rightarrow\Delta^r$ are required, even in the case where $k$ is terminal in $I$. Since we are doing that extension working over $U$, only inner horn extensions are needed.
\end{innerproof}

And now secondly we show the existence of cartesian and cocartesian lifts. Since there is only one nontrivial 1-cell $01\in\Delta^1_1$, we must merely show:
\begin{claim}
For any object $A\in(\Cinftypp)_0$, there is a cartesian morphism of $\cD$ over $01$ with target $A$; for any object $C\in(\Cinfty)_0$, there is a cocartesian morphism of $\cD$ over $01$ with source $C$.
\end{claim}
\begin{innerproof}{Proof of claim}
Given $C\in(\Cinfty)_0$, we must give a cocartesian 1-cell in $\cD$ from it which lies over the nondegenerate 1-cell of $\Delta^1$; we take the cell consisting of $c=C\in(\Cinfty)_0$, $a=\Fun(C,\cU)\in(\Cinftypp)_0$, and $f\in(\Cinfty)_1$ representing the evaluation map $C\times\Fun(C,\cU)\rightarrow \cU$.

Similarly, given $A\in(\Cinftypp)_0$, the cartesian 1-cell in $\cD$ consists of $c=\Funpp(A,\cU)\in(\Cinfty)_0$, $a=A\in(\Cinftypp)_0$, and $f\in(\Cinfty)_1$ representing the evaluation map $\Funpp(A,\cU)\times A\rightarrow \cU$.

The proofs that these are indeed cocartesian and cartesian respectively are very similar. We aim to show that the morphism
$$\cD_{(c,a,f)/}\longrightarrow\cD_{c/}\timeso{\cD}\Cinftypp$$
is acyclic Kan. This unravels to the requirement that we can extend two compatible maps $\Delta^n\rightarrow{\Cinfty}_{/\cU}$ and $\partial\Delta^n\rightarrow{\Cinfty}_{/(\cC\times\Fun(\cC,\cU)\rightarrow \cU)}$ to a map $\Delta^n\rightarrow{\Cinfty}_{/(\cC\times\Fun(\cC,\cU)\rightarrow \cU)}$, with a requirement that all the maps we supply are product-preserving.

That we can do so follows immediately from the adjunction (in the quasicategorical sense) of the functors $(\cC\times-)$ and $\Fun(\cC,-)$ for $n\geq 1$, and is a quick check in the case $n=0$.
\end{innerproof}
This completes the proof.
\end{proof}

As an immediate corollary, we get
\begin{prop}
\label{models-colimits-to-limits}
The ``models'' functor $\Mod(-,\cU)$ takes colimits of theories to limits of their quasicategories of models.\qed
\end{prop}
We shall show in Proposition \ref{theories-cocomplete} that colimits of theories exist; and thus this will be a helpful tool.

\subsection{Properties of quasicategories of models}

These categories of models have good properties:
\begin{prop}
 \label{models-limits}
 If $\cC$ is a theory, then the quasicategory $\Mod(T)$ is complete, with limits computed pointwise.
\end{prop}
\begin{proof}
 The quasicategory $\Fun(T,\Spaces)$ is complete, with limits computed pointwise. By Proposition \ref{interchange-of-limits} showing that limits can be interchanged, the limit of a diagram from $\Mod(T)$ is again in $\Mod(T)$, and is thus the limit in $\Mod(T)$.
\end{proof}

\begin{prop}
 \label{models-functor-limits}
 Given a morphism of theories $f:T\rightarrow U$, the pullback functor $f^*:\Mod(U)\rightarrow\Mod(T)$ preserves limits.
\end{prop}
\begin{proof}
 The pullback functor $\Map(U,\Spaces)\rightarrow\Map(T,\Spaces)$ evidently preserves limits, since they're defined pointwise. The result follows, since limits in $\Mod(T)$ are just limits in $\Map(T,\Spaces)$ (and the same for $U$), and this pullback functor restricts to our desired one.
\end{proof}

We recall from \cite{HTT}*{Section 5.3.1} the notion of a filtered simplicial set. This is equivalent for having liftings for all maps $A\rightarrow A\star 1$, where $A$ is the nerve of a finite poset. A filtered colimit is then just a colimit on a filtered diagram.

\begin{prop}
 The category $\Mod(T)$ has filtered colimits, which are computed pointwise.
\end{prop}
\begin{proof}
 This is the same argument as \ref{models-limits}, using \cite{HTT}*{Prop 5.3.3.3}, saying that filtered colimits commute with limits.
\end{proof}

Now we wish to study push-forwards of models, showing that taking left Kan extensions provides a left adjoint to the pullback functor. This will require some work; we subdivide it into two major parts.

We show that this is plausible:
\begin{prop}
\label{left-kan-preserves-prods}
Given a morphism $f:T\rightarrow U$ of theories, and a model $G:T\rightarrow\Spaces$, the left Kan extension of $G$ along $f$ preserves products and is thus a model of $U$.
\end{prop}
\begin{proof}
The left Kan extension is given by
$$(f_*G)(x)=\colim\left(T\timeso{U}U_{/x}\longrightarrow T\stackrel{G}{\longrightarrow}\Spaces\right).$$

We must show that $(f_*G)(1)\isom 1$ and $f_*(G)(x\times y)\isom f_*(G)(x)\times f_*(G)(y)$.

In both cases we show that there is a natural map from the colimit diagrams which define each side, which is cofinal (in the sense of Joyal, written up by Lurie \cite{HTT}*{4.1}), and thus there is an equivalence between them.

In the first case, we have
$$(f_*G)(1)=\colim\left(T\isom T\timeso{U}U_{/1}\longrightarrow T\longrightarrow\Spaces\right).$$

It is easy to see that the inclusion of the terminal object $(1,\id_1)$ into $T$ is cofinal. Indeed, by Joyal's characterisation of cofinal maps \cite{HTT}*{4.1.3.1}, we must show that $1\times_TT_{1/}$ is weakly contractible. This is clear: it has an initial object $1$.

Thus $1\rightarrow T$ induces an isomorphism of colimits. This terminal object is sent to $1\in T_0$ and thence to $1\in\Spaces_0$. This proves the first case.

In the second case, $f_*(G)(x\times y)$ is given by the colimit
$$(f_*G)(x\times y)=\colim\left(T\timeso{U}U_{/x\times y}\longrightarrow T\stackrel{G}{\longrightarrow}\Spaces\right).$$
There is a functor
$$\left(T\timeso{U}U_{/x}\right)\times\left(T\timeso{U}U_{/y}\right)\longrightarrow T\timeso{U}U_{/x\times y},$$
which sends
$$((t_1,f(t_1)\rightarrow x),(t_2,f(t_2)\rightarrow y))\longmapsto(t_1\times t_2,f(t_1\times t_2)\rightarrow x\times y)$$
in the evident way.

According to \cite{HTT}*{4.1.3.1}, to show this map is cofinal we need to show that, for any $(t,f(t)\rightarrow x\times y)\in\left(T\timeso{U}U_{/x\times y}\right)_0$, the simplicial set
$$\left(\left(T\timeso{U}U_{/x}\right)\times\left(T\timeso{U}U_{/y}\right)\right)\timeso{T\timeso{U}U_{/x\times y}}\left(T\timeso{U}U_{/x\times y}\right)_{(t,f(t)\rightarrow x\times y)/}$$
is weakly contractible.

This simplicial set is isomorphic to
$$\left(T^2\timeso{T}T_{t/}\right)\timeso{U^2\timeso{U}U_{f(t)/}}\left(\left(U_{/x}\times U_{/y}\right)\timeso{U_{/x\times y}}U_{f(t)//x\times y}\right),$$
which is the quasicategory of pairs $a,b\in T$ equipped with maps $t\rightarrow a\times b$, and 2-cells $f(t)\rightarrow f(a\times b)\rightarrow x\times y$.
But this quasicategory has an evident terminal object $\Delta:t\rightarrow t\times t$ and $f(t)\rightarrow f(t\times t)\rightarrow x\times y$, which makes it weakly contractible.

And the colimit of $\left(T\timeso{U}U_{/x}\right)\times\left(T\timeso{U}U_{/y}\right)$ is indeed $f_*(G)(x)\times f_*(G)(y)$, by Proposition \ref{products-and-colimits}. This completes the proof.
\end{proof}

Now we can finish the job:
\begin{prop}
\label{models-pushforward}
 Given a morphism $f:T\rightarrow U$ of theories, the pullback functor $f^*:\Mod(U)\rightarrow\Mod(T)$ has a left adjoint $f_*$, given by left Kan extension.
\end{prop}
\begin{proof}
We could do this simply by restricting the standard adjunction between $\Fun(T,\Spaces)$ and $\Fun(U,\Spaces)$ given by composition and left Kan extension. However, we build an adjunction by hand to make more of the structure visible.

First, we use $f^*$ to define a cocartesian fibration $\Mod(T/U)\rightarrow\Delta^1$, as described in subsection \ref{mapping-cylinders-quasicategories} (it is a cocartesian fibration, as proved in Proposition \ref{mapping-cylinders}).

We need to show that it is also cartesian, so it represents an adjunction. We have observed it to be an inner fibration already (in Proposition \ref{mapping-cylinders-1}; we just need to demonstrate the existence of cartesian lifts for edges. The simplicial set $\Delta^1$ only has one degenerate 1-cell; it is only over that cell that the problem is not vacuous.

Given $A\in\Mod(T)_0$, we take a left Kan extension of $A$ along $f$, given by
$$(f_*A)(x)=\colim\left(T\timeso{U}U_{/x}\longrightarrow T\stackrel{A}{\longrightarrow}\Spaces\right),$$
where we identify objects of $T$ and of $U$ for brevity. This is product-preserving by Proposition \ref{left-kan-preserves-prods}.

Our cartesian lift $\alpha$ shall have this as its zero vertex, so we must exhibit a morphism $f_*f^*A\rightarrow A$. This is provided by the universal property of the colimit.

We must now show that this 1-cell $\alpha$ from $f^*A$ to $A$ is cartesian. That means showing that the projection 
$\Mod(T/U)_{/\alpha}\rightarrow\Mod(T/U)\timeso{\Delta^1}\Delta^0$ is acyclic Kan.

Unpacking the definitions, this morphism is the evident projection
$$\Mod(U)_{/f_*A}\timeso{\Mod(T)_{/f^*f_*A}}\Mod(T)_{/(f^*f_*A\rightarrow A)} \longrightarrow \Mod(T)_{/A};$$
we can show that this is acyclic Kan by working pointwise and using the acyclic Kan condition of the colimit.
\end{proof}

Note that this gives us a notion of a free model $T(X)$ of a theory $T$ on a space $X$: a space can be viewed as a model of the initial theory $\Finop$, and we can use the push-forward associated to the morphism of theories $\Finop\rightarrow T$.

\label{pointed-theories}

A theory is said to be \emph{pointed} if it has a zero object: an object $0$ which is both initial and terminal. This is standard categorical terminology, and is also justified by the following proposition:
\begin{prop}
If $T$ is a pointed theory, the terminal model (the model given by the constant $1$ functor) is a zero object in the category of models. In particular, any model $T\rightarrow\cU$ factors through $\cU_{1/}$. 
\end{prop}
\begin{proof}
This is straightforward.
\end{proof}

We write $\Theories_*$ for the quasicategory of pointed theories: this is the full subquasicategory of $\Theories$ whose objects are the pointed theories.

The following is a valuable structure theorem for pointed theories.
\begin{prop}
\label{finsop-is-universal}
 $\Finsop$ is the initial pointed theory. In particular, any pointed theory $\Finop\rightarrow T$ factors as a composite $\Finop\stackrel{+}{\rightarrow}\Finsop\rightarrow T$, where the right-hand map preserves products and initial objects.
\end{prop}
We defer its proof until Section \ref{proof-finsop-is-universal}: by then we will have developed machinery to understand the situation better.

\subsection{Structure on algebraic theories}

In this section we show that the quasicategory of theories is complete (Proposition \ref{theories-complete}), which is straightforward, and that it is cocomplete (Proposition \ref{theories-cocomplete}), which is much harder.

In order to prove the latter result, we introduce a good deal of machinery. Intrinsic in this machinery is the ability to take the free theory on some fairly general collection of data, but we apply it only to take the free theory on a diagram consisting of other theories. Thus we anticipate that the methods introduced here could be used to prove other theorems of this general character.
\begin{prop}
\label{theories-complete}
 The quasicategory of theories is complete.
\end{prop}
\begin{proof}
 An $I$-shaped diagram in theories yields an underlying diagram $F:I\rightarrow\Cinfty$. This neglects the functors from $\Finop$; for each $i$, we write $p(i)$ for the map $\Finop\rightarrow F(i)$. This is classified by a Cartesian fibration $X\rightarrow I^\op$. Lurie's model, which we recall from \cite{HTT}*{3.3.3}, for the limit of this diagram (in $\Cinfty$) is the quasicategory of Cartesian sections $I^\op\rightarrow X$ (that is, the quasicategory of sections which take 1-cells to Cartesian 1-cells).

 We consider the full subquasicategory of this on the objects $s:I^\op\rightarrow X$ for which there is a finite set $A$ such that $s(i)=p(i)(A)$, that is, those which act diagonally on objects.

 Any cone over $F$ in $\Theories$ acts diagonally on the objects, up to equivalence, since the maps commute with the structure maps. Hence the universal property of the product in $\Cinfty$ gives us a universal property for this subobject in $\Theories$.
\end{proof}

Now we turn our attention to showing that theories have all colimits. This will require some technical work, and we build up to the proof slowly.

The plan is as follows: Lurie has proved that the quasicategory of quasicategories is cocomplete. Thus, for any simplicial set $D$, the colimit $\colim_{\Theories}(D)$ of a diagram in theories factors uniquely through the colimit $\colim_{\Cinfty}(D)$ in the quasicategory of quasicategories. Indeed, we should expect it to be the universal quasicategory with a functor from $\colim_{\Cinfty}(D)$ such that the images of all the product cones in elements of $D$ are product cones.

Consider the quasicategory $(\Cinfty)_{(1\star D)/}$ of quasicategories with a map from $1\star D$. We are interested in the full subcategory $(\Cinfty)^{\lim}_{(1\star D)/}$ with objects the quasicategorical limit cones $(1\star D)\rightarrow\cC$.

Our first step is this:
\begin{prop}
\label{inclusion-of-limits-into-cones}
The inclusion functor
$$F:(\Cinfty)^{\lim}_{(1\star D)/}\longrightarrow(\Cinfty)_{(1\star D)/}$$
preserves all limits.
\end{prop}
\begin{proof}
By \cite{HTT}*{4.4.2.6}, it suffices to show it preserves all products and pullbacks.

In this proof we write $*$ for the terminal simplicial set, to avoid overuse of the symbol $1$.

Given a set of quasicategories and maps $\{*\star D\rightarrow \cC_\alpha\}_{\alpha\in A}$, all of them limit cones, then the diagonal map $(*\star D)\rightarrow\prod_{\alpha\in A}\cC_\alpha$ can easily be shown to be a limit cone.

Now, we have to deal with pullbacks of quasicategories; we recall the setup of Definition \ref{homotopy-pullbacks-spaces}.

Now, suppose we have limit cones $(*\star D)\rightarrow\cC_1,\cC_2,\cE$. We then have a diagonal map $(*\star D)\rightarrow\cC$, and must show that this too is a limit cone. Suppose we have a cofibration $I\rightarrow J$; we must show that there are liftings
\begin{displaymath}
\xymatrix{I\ar[r]\ar[d]&J\ar[d]\dar[dl]\\
          \cC_{/(*\star D)}\ar[r]&\cC_{/D},}
\end{displaymath}
or equivalently that there are extensions
\begin{displaymath}
\xymatrix{(I\star *\cup J\star\emptyset)\star D\ar[r]\ar[dr]&(J\star *\star D)\dar[d]\\
          &\cC,}
\end{displaymath}
provided that the restruction to $*\star D$ is the given cone.

Unravelling using the definition of $\cC$, we are demanding extensions
\begin{displaymath}
\xymatrix{
(I\star *\cup J\star\emptyset)\star D\ar[rr]\ar[dd]\ar[dr]&&\cC_1\ar[dd]\\
&J\star *\star D\dar[ur]\ar[dd]&\\
((E2\times I)\star *\cup(E2\times J)\star\emptyset)\star D\ar[rr]\ar[dr]&&\cE\\
&(E2\times J)\star *\star D\dar[ur]&\\
(I\star *\cup J\star\emptyset)\star D\ar[rr]\ar[dr]\ar[uu]&&\cC_2\ar[uu]\\
&J\star *\star D\dar[ur]\ar[uu]&}
\end{displaymath}
We can extend the top and bottom without difficulty, using that the maps $(*\star D)\rightarrow\cC_1,\cC_2$ are limit cones. This leaves us with an extension problem
\begin{displaymath}
\xymatrix{
\left((E2\times I)\star *\cup(E2\times J)\star\emptyset\cup (\{0,1\}\times J)\star *\right)\star D\ar[r]\ar[dr]&(E2\times J)\star *\star D\dar[d]\\
&\cE}
\end{displaymath}
which is readily checked to be a right lifting against a cofibration, and so follows from the fact that $(*\star D)\rightarrow\cE$ is a limit cone.

This completes the proof.
\end{proof}

Now we consider the diagram
\begin{displaymath}
\xymatrix{(\Cinfty)^{\lim}_{(1\star D)/}\ar[r]^F\ar[dr]_{\sim}&(\Cinfty)_{(1\star D)/}\ar[d]\\
&(\Cinfty)_{D/}.}
\end{displaymath}
The diagonal map is an acyclic Kan fibration, since every diagram naturally has a contractible space of limits.

Note that $\Cinfty$ is a presentable category \cite{LurieBicat}*{Remark 1.2.11}, and \cite{HTT}*{5.5.3.11} shows that undercategories of presentable categories are presentable. Thus all the categories in the diagram are presentable.

Also, the proof of Proposition \ref{inclusion-of-limits-into-cones} demonstrates that colimits in $(\Cinfty)_{D/}$ and $(\Cinfty)_{(1\star D)/}$ are computed in $\Cinfty$, and thus (using that the diagonal map is an equivalence) all three functors preserve colimits.

Accordingly, we can apply Lurie's Adjoint Functor Theorem \cite{HTT}*{5.5.2.9} to show the following:
\begin{prop}
\label{forcing-a-limit}
The functor $F:(\Cinfty)^{\lim}_{(1\star D)/}\longrightarrow(\Cinfty)_{(1\star D)/}$ admits a left and a right adjoint.\qed
\end{prop}

A straightforward consequence of the existence of a left adjoint is that, for every quasicategory $\cC$ and map $1\star D\rightarrow\cC$, there is a universal quasicategory $\cC\rightarrow\cC'$ such that the composite $1\star D\rightarrow\cC'$ is a limit cone, in the sense that
$$(\Cinfty)_{(1\star D\rightarrow\cC\rightarrow\cC')/}\timeso{(\Cinfty)_{(1\star D\rightarrow\cC)/}}(\Cinfty)^{\lim}_{(1\star D\rightarrow\cC)/}\longrightarrow(\Cinfty)_{(1\star D\rightarrow\cC)/}$$
is acyclic Kan.

Indeed, the morphism $\cC\rightarrow\cC'$ is just the unit of the adjunction.

Now, suppose we have a diagram $K\rightarrow\Theories$. We will construct a colimit. Firstly, the extension $K\rightarrow\Cinfty$ has a colimit $K\star 1\rightarrow\Cinfty$. Transfinitely enumerate the finite product diagrams as $\{f_\alpha:(1\star X_\alpha)\rightarrow K_{s(\alpha)}\}$. With this notation, we prove the result we were aiming for:

\begin{prop}
\label{theories-cocomplete}
The quasicategory of theories is cocomplete.
\end{prop}
\begin{proof}
We provide a colimit for any diagram $F:K\rightarrow\Theories$.

Firstly, we can obtain from our diagram $F$ a diagram $F':1\star K\rightarrow\Cinfty$, sending $1$ to $\Finop$. We take the colimit of that, using \cite{HTT}*{3.3.4}. We claim that the resulting colimit cocone has essentially surjective structure maps.

 Indeed, any object in $\colim(F')$ is in the essential image of $F'(z)$ for some $z\in 1\star K$: the structure maps are jointly essentially surjective. (To prove this, it is quick to verify that the essential image of $1\star K$ in $\colim(F')$ satisfies the colimit property, and is thus all of it).

 However, any $a\in F'(z)_0$ is the essential image of some $A\in(\Finop)_0$. Also, there is an equivalence (induced by the image of the 2-cell $(1\star\{z\}\star 1)$ in $\Cinfty$) between the image of $a$ and the image of $A$ in $\colim(F')$. So every structure map has the same essential image: they're all essentially surjective.

Now, we will manufacture a colimit in $\Theories$.

We start with $X_0=\colim_{\Cinfty}(F')$. We can transfinitely enumerate the finite product diagrams in the quasicategory $\Finop$ as
$$\{f_\alpha:(1\star D_\alpha)\rightarrow\Finop\}_{\alpha<\kappa}$$
for some ordinal $\kappa$, where $D_\alpha$ is a discrete simplicial set, and $z_\alpha\in K_0$. We choose to do this with redundancy: we want each individual product diagram to appear infinitely many times and be cofinal in $\kappa$. One straightforward way to ensure this is to enumerate them without repetition with ordertype $\lambda$, then take $\kappa=\lambda\omega$ and repeat our list $\omega$-many times.

Our aim is to produce $\colim_{\Theories}(F')$ by starting from $X_0$ and extending all the maps $f_\alpha$ to limit cones.

Proposition \ref{forcing-a-limit} supplies us with a quasicategory $X_1$ and a unit map $u$ such that the composite
$$(1\star D_0)\longrightarrow X_0\stackrel{u}{\longrightarrow}X_1$$
is a limit cone.

Similarly, we produce $X_2$ from $X_1$ by using the adjunction to provide a quasicategory from which the map from $1\star D_1$ is a limit cone. Then we proceed by transfinite induction, extending to limit ordinals by taking the filtered colimits (in $\Cinfty$) of the preceding quasicategories:
$$X_{\lim_i(\alpha_i)}=\lim_i(X_{\alpha_i}).$$

The resulting quasicategory $X_\kappa$ is our colimit. We have several checks to make to show this to be the case.

Firstly, it is necessary to show we haven't enlarged our quasicategory in an unacceptable manner:
\begin{claim}
All the structure maps $X_\alpha\rightarrow X_\beta$ for $\alpha<\beta$ are essentially surjective.
\end{claim}
\begin{innerproof}{Proof of claim}
It suffices to show both that the successor maps $X_\alpha\rightarrow X_{\alpha+1}$ are essentially surjective, and also that a colimit of shape $\omega$ of essentially surjective maps is essentially surjective.

A similar argument works for both. We can show for each that a failure to be essentially surjective would violate the universal property: that the image of the morphism would provide a smaller object with the same property.

Indeed, if $X_\alpha\rightarrow X_{\alpha+1}$ was not essentially surjective, the image $X'_{\alpha+1}\subset X_{\alpha+1}$ would result in the nontrivial factorisation
$$(1\star D_\alpha)\longrightarrow X'_{\alpha+1}\longrightarrow X_{\alpha+1},$$
and the left-hand map can easily be checked to be a product cone. This contradicts the universal property of the adjunction.

Similarly, if a colimit
$$Z_0\longrightarrow Z_1\longrightarrow Z_2\longrightarrow\cdots\longrightarrow Z_\omega$$
of essentially surjective maps of quasicategories is not essentially surjective, then the essential image factors the structure maps of the colimit nontrivially, which contradicts the universal property of the colimit.
\end{innerproof}

Secondly, we need to show it is indeed a theory, and that the structure maps we've defined are maps of theories. We've done essential surjectivity already, so we just need the following:
\begin{claim}
The defined maps $\Finop\rightarrow X_\kappa$ preserve all finite products.
\end{claim}
\begin{innerproof}{Proof of claim}
Given a limit cone $1\star D\rightarrow\Finop$, we must show that the composite $1\star D\rightarrow X_\kappa$ is a limit cone too. Given a lifting problem for a cofibration $\partial\Delta^n\rightarrow\Delta^n$ as follows:
\begin{displaymath}
\xymatrix{\partial\Delta^n\ar[r]\ar[d]&\Delta^n\ar[d]\dar[dl]\\
          (X_\kappa)_{/(1\star D)}\ar[r]&(X_\kappa)_{/D},}
\end{displaymath}
we rewrite it as
\begin{displaymath}
\xymatrix{(\partial\Delta^n\star 1\star D)\cup(\Delta^n\star\emptyset\star D)\ar[r]\ar[dr]&\Delta^n\star 1\star D\dar[d]\\
          &X_\kappa.}
\end{displaymath}
Since the simplicial set $(\partial\Delta^n\star 1\star D)\cup(\Delta^n\star\emptyset\star D)$ is finite, the map from it to $X_\kappa$ factors through some $X_\lambda$ for which $1\star D\rightarrow X_\lambda$ is a product cone (since, by construction, the set of such ordinals $\lambda$ is cofinal in $\kappa$).

The required extension exists in that $X_\lambda$ and thus also in $X_\kappa$.
\end{innerproof}

Lastly, of course, we need to verify the universal property of a colimit.
\begin{claim}
$X_\kappa$ is universal among theories under $F$.
\end{claim}
\begin{innerproof}{Proof of claim}
We need to show that the functor $\Theories_{(F\star 1)/}\rightarrow\Theories_{F/}$ is acyclic Kan. Suppose given a cofibration $I\rightarrow J$; we have a lifting problem
\begin{displaymath}
\xymatrix{I\ar[r]\ar[d]&J\ar[d]\\
\Theories_{(F\star 1)/}\ar[r]&\Theories_{F/}.}
\end{displaymath}

It suffices to consider cofibrations $\partial\Delta^n\rightarrow\Delta^n$. We consider the $n=0$ and $n\geq 1$ separately.

If $n=0$, our cofibration is $\emptyset\rightarrow 1$: we have a cone $K\star 1\rightarrow\Theories$ describing a theory $T$ under $F$; the aim is to factor it through $X_\kappa$.

Since $X_0$ is the colimit of $F$ in $\Cinfty$, we have a diagram $K\star 1\star 1\rightarrow\Cinfty$ factoring our cone through $X_0$. Working under $F$, since the maps $F(x)\rightarrow T$ are product-preserving, we can factorise this successively through the $X_\lambda$ to get an essentially surjective, product-preserving map $X_\kappa\rightarrow T$ under $F$ as required.

Now, in case $n\geq 1$, we have compatible functors $K\star 1\star\partial\Delta^n\rightarrow\Theories$ and $K\star\emptyset\star\Delta^n\rightarrow\Theories$, with the middle $1$ sent to $X_\kappa$. Equivalently, this is a diagram $K\star\partial\Delta^{1+n}\rightarrow\Theories$ and we need to extend it to $K\star\Delta^{1+n}\rightarrow\Theories$.

Since $X_0$ is the colimit of $F$, we can extend this the underlying diagram $K\star\Delta^{1+n}\rightarrow\Cinfty$ to a diagram $K\star 1\star\Delta^{1+n}\rightarrow\Cinfty$, with the middle $1$ sent to $X_0$. Using the universal property of the adjunction and the colimiting property, we can extend this to a map $K\star N(\kappa+1)\star\Delta^{1+n}\rightarrow\Cinfty$, where $N(\kappa+1)$ denotes the nerve of the ordinal $\kappa+1$ viewed as a poset (that is, as the poset of ordinals less than or equal to $\kappa$), and where the ordinal $\lambda$ is sent to $X_\lambda$.

The terminal vertex of $N(\kappa+1)$ and the initial vertex of $\Delta^{1+n}$ are both sent to $X_\kappa$. Moreover, by construction, they are identical under $F$, and so by construction the edge between them is the identity. Restriction to $K\star\{\kappa\}\star\Delta^n\rightarrow\Cinfty$ thus gives us the required diagram in $\Cinfty$; since all the edges were present already and were morphisms of $\Theories$, this is also a diagram in $\Theories$.
\end{innerproof}
This completes the proof.
\end{proof}

We can do similar things with this method:
\begin{prop}
\label{cinftypp-cocomplete}
The quasicategory $\Cinftypp$ of quasicategories with all finite products, and product-preserving functors between them, has colimits.
\end{prop}
\begin{proof}
The proof of the preceding proposition generalises in a straightforward manner. Since our diagrams are no longer cones of essentially surjective maps under $\Finop$, we need to consider products in all the diagrams and force their images to all be products (whereas before it sufficed to consider only those in $\Finop$). We no longer need to ensure essential surjectivity, but we do however need to provide limits for all the new objects introduced. So we intersperse the operations which force cones to be product cones with operations that adjoin new products for the objects (using the methods of \cite{HTT}*{5.3.6}).
\end{proof}

\subsection{Free models on sets}

Free models for a theory are shown to exist by Proposition \ref{models-pushforward}; this section records a more explicit, less involved construction of free models on finite sets.

Let $T$ be a theory. We suppose given a functorial model $\Map_T(-,-):T^\op\times T\rightarrow\Spaces$ for the homspaces in $T$. Such models are shown to exist and are discussed further in \cite{HTT}*{Section 1.2.2}.

The structure map $\Finop\rightarrow T$ is equivalent to a functor $\Fin\rightarrow T^\op$, and we can compose this with the homspace functor to get a map
$$\Free:\Fin\rightarrow\Fun(T,\Spaces).$$
In other words, we take $\Free_X(Y)=\Map_T(X,Y)$.

The functor $\Free_X$ is product-preserving since $\Map_T(X,-)$ is, so we actually get a functor
$$\Free:\Fin\rightarrow\Mod(T).$$

This behaves as we would hope:
\begin{prop}
The functor $\Free_X$ is indeed the free model of $T$ on $X$.
\end{prop}
\begin{proof}
For any model $A$ of $T$, we have a natural equivalence
$$\Spaces(X,A(1))\isom\Mod(T)(\Free_X,A);$$
This is a straightforward exercise using the quasicategorical Yoneda lemma of \cite{HTT}*{Section 5.1}.
\end{proof}
In particular, this agrees with the more general construction of Proposition~\ref{models-pushforward}.

We can also prove:
\begin{prop}
A theory $T$ is equivalent to the opposite of the full subquasicategory of $\Mod(T)$ on the free models on finite sets.
\end{prop}
\begin{proof}
The Yoneda embedding used above is full and faithful; and the functor is evidently essentially surjective on objects.
\end{proof}

\section{Lawvere symmetric monoidal categories}

In this section we introduce a quasicategory $\Span$, a theory for commutative monoid objects, and we study its properties. We first motivate this by explaining the failure of the most naive approach.

\subsection{The theory of strict commutative monoids}

Given a collection of objects in a commutative monoid, we can define operations which form more objects by adding some copies of the originals. Here is an example of such an operation:
$$(a,b,c)\longmapsto(b+a,c,0,a+a+a+b+c).$$
Indeed, since addition of finite collections is the only operation in a commutative monoid, we rather expect that all operations natural in the monoid should take this form. But we should find out how to describe operations in this form in general.

We are adding up some copies of the things we started with. We can regard this as a two-stage process: first we make copies, then we add. So we can factor this operation as
$$(a,b,c)\longmapsto(b,a,c,a,a,a,b,c)\longmapsto(b+a,c,0,a+a+a+b+c).$$

In general we can associate natural operations on a monoid $M$ to maps of finite sets:
\begin{itemize}
\item Given a map $f:X\leftarrow U$ of sets, we can produce a \emph{copying} map $\Delta_f:M^X\rightarrow M^U$ via $(\Delta_fA)_u = A_{f(u)}$.
\item Given a map $g:U\rightarrow Y$ of sets, we cap produce an \emph{addition} map $\Sigma_g:M^U\rightarrow M^Y$ via $(\Sigma_gA)_y = \sum_{g(u)=y}A_u$.
\end{itemize}
Of course, we can compose these, and so given any diagram of finite sets
$$X\stackrel{f}{\longleftarrow}U\stackrel{g}{\longrightarrow}Y$$
we get an operation $\Sigma_g\circ\Delta_f$, which sends $M^X\rightarrow M^Y$ via $$(\Sigma_g\circ\Delta_f)(A)_y = \sum_{g(u)=y}A_{f(u)}.$$

We refer to a diagram of this shape as a \emph{span diagram}. It certainly seems natural to suggest that span diagrams should give all the natural operations on a commutative monoid. However, span diagrams $X\leftarrow U\rightarrow Y$ and $X\leftarrow U'\rightarrow Y$ yield identical operations if they are isomorphic in the sense that
\begin{displaymath}
\xymatrix@R-20pt{&U\ar[dr]\ar[dd]^\wr&\\
X\ar[ur]\ar[dr]&&Y\\
&Y'\ar[ur]&}
\end{displaymath}
commutes.

Moreover, we can compose span diagrams: we use pullbacks:
\begin{displaymath}
\left(\vcenter{\xymatrix@R-8pt@C-16pt{&U\ar[dl]\ar[dr]&\\X&&Y}}\right) \circ
\left(\vcenter{\xymatrix@R-8pt@C-16pt{&V\ar[dl]\ar[dr]&\\Y&&Z}}\right) = 
\left(\vcenter{\xymatrix@R-8pt@C-16pt{&U\timeso{Y}V\ar[dl]\ar[dr]&\\X&&Z}}\right).
\end{displaymath}
It can quickly be checked that this is the right thing to do, using the formula above for $\Sigma_g\circ\Delta_f$.

These span diagrams thus form a category $\ThMon$. Hence it is reasonable to believe that this is the theory of commutative monoids, and in fact this is a classical result.

We might aim to apply this to homotopy theory: Badzioch \cite{Bad1} has shown that the models of $\ThMon$ in $\Spaces$ are exactly the generalised Eilenberg--Mac Lane spaces. His paper works in ordinary category theory, but a note observes that all results carry through in the world of simplicial categories.

We might hope for a different theory: frequently, a more useful notion of commutative monoid in $\Spaces$ is the notion of $E_\infty$-monoid \cite{Adams-ILS}. This is a monoid which is commutative only up to coherent homotopy. So we might ask, how might we change $\ThMon$ in order to get this more nuanced theory?

To find an answer, we must realise that we lose valuable information when we pass to isomorphism classes of span diagrams to form the category $\ThMon$. In particular, multiplying two elements in an $E_\infty$-monoid corresponds to the span diagram
$$\{1,2\}\stackrel{=}{\longleftarrow}\{1,2\}\longrightarrow\{1\}.$$
This span diagram has a nontrivial automorphism, exchanging the $1$ and $2$ in the central part of the span. This is important: without it, there is only one way to do the multiplication. However, with it, and with the homotopies induced by the automorphism, we hope there would be a space of ways of multiplying two elements which is contractible but has a free $C_2$-action, as is required by a $E_\infty$-monoid action.

Our task therefore is to build an object like $\ThMon$ but which remembers the automorphisms of span diagrams.

\subsection{The bicategory $\TwoSpan$}
\label{span-bicategory}

Throughout this subsection we assume given a canonical, functorial choice of pullbacks of finite sets.

We introduce a bicategory $\TwoSpan$ of span diagrams. A 0-cell of $\TwoSpan$ is a finite set. A 1-cell from $X_0$ to $X_1$ is a span diagram $X_0\leftarrow Y\rightarrow X_1$ of finite sets. A 2-cell between diagrams $X_0\leftarrow Y\rightarrow X_1$ and $X_0\leftarrow Y'\rightarrow X_1$ is an isomorphism $f:Y\stackrel{\sim}{\rightarrow}Y'$ fitting into a diagram as follows:

\begin{displaymath}
\xymatrix@R-20pt{&Y\ar[dl]\ar[dr]\ar[dd]_f^{\wr}&\\
          X_0&&X_1\\
          &Y'.\ar[ul]\ar[ur]}
\end{displaymath}
2-cells compose in the obvious way; 1-cells compose by taking pullbacks: the composite of $X_0\leftarrow X_{01}\rightarrow X_1$ and $X_1\leftarrow X_{12}\rightarrow X_2$ is given by $X_0\leftarrow X_{02}\rightarrow X_2$, where $X_{02}$ is the following pullback:
\begin{displaymath}
\xymatrix{&&X_{02}\pb{270}\ar[dl]\ar[dr]&&\\
          &X_{01}\ar[dl]\ar[dr]&&X_{12}\ar[dl]\ar[dr]&\\
          X_0&&X_1&&X_2.}
\end{displaymath}

It is a simple exercise to show that this gives a bicategory.

We will later have cause to use a generalisation of this notion. Given a category $\cC$ which has all pullbacks, and a functorial choice of pullbacks, we can define the bicategory $\TwoSpan(\cC)$ of spans in $\cC$: 0-cells are objects of $\cC$, 1-cells are span diagrams in $\cC$, and 2-cells are isomorphisms of spans. So, in this notation, our category $\TwoSpan$ is $\TwoSpan(\Fin)$.

\subsection{Equivalences in $\TwoSpan$}

We work with the weak 2-category of spans $\TwoSpan(\cC)$, where $\cC$ is any category whatsoever. First we prove a more-or-less standard lemma of ordinary category theory:

\begin{prop}
\label{pb-se}
Pullbacks of split epimorphisms are split epimorphisms.
\end{prop}
\begin{proof}
Suppose given a diagram
\begin{displaymath}
 \xymatrix{A\ar[r]^k  \ar[d]_h&B\ar[d]^g\\
           C\sear[r]_f        &D,}
\end{displaymath}
where the bottom morphism $f$ is a split epimorphism: a morphism such that there is $f':D\rightarrow C$ with $ff'=1_D$.

This affords us a map $f'g:B\rightarrow C$. Now, we have $f(f'g) = g1_B$, and so, by the definition of the pullback, there is a map $k':B\rightarrow A$ with $kk'=1_B$, as required.
\end{proof}

This allows us to prove an important structural result for span categories:

\begin{prop}
\label{span-isos}
Objects $X,Y\in\Ob\TwoSpan(\cC)$ are equivalent if and only if they are isomorphic as objects of $\cC$.
\end{prop}
\begin{proof}
Given two isomorphic objects in $\cC$, any span of isomorphisms between them forms an equivalence in $\TwoSpan(\cC)$.

The data of an equivalence consists of 1-cells $X\stackrel{a}{\leftarrow}U\stackrel{b}{\rightarrow}Y$ and $Y\stackrel{c}{\leftarrow}V\stackrel{d}{\rightarrow}Y$, fitting into diagrams
\begin{displaymath}
 \xymatrix{&&X\pb{270}\ar[dl]^p\ar[dr]_q\ar@/_2ex/[ddll]_{=}\ar@/^2ex/[ddrr]^{=}&&\\
           &U\ar[dl]^{a}\ar[dr]^{b}&&V\ar[dl]_{c}\ar[dr]_{d}&\\
           X&&Y&&X,}
\end{displaymath}
\begin{displaymath}
 \xymatrix{&&Y\pb{270}\ar[dl]^r\ar[dr]_s\ar@/_2ex/[ddll]_{=}\ar@/^2ex/[ddrr]^{=}&&\\
           &V\ar[dl]^{c}\ar[dr]^{d}&&U\ar[dl]_{a}\ar[dr]_{b}&\\
           Y&&X&&Y.}
\end{displaymath}
The maps $a$, $b$, $c$ and $d$ are split epimorphisms (by inspecting the left and right composites in each diagram). But this means that $p$, $q$, $r$ and $s$ are also split epimorphisms, by Proposition \ref{pb-se}.

However $p$, $q$, $r$ and $s$ are also split monomorphisms (by inspection of the left and right composites), and thus isomorphisms (since if a morphism is left and right invertible, the inverses agree). This clearly means that $a$, $b$, $c$ and $d$ are isomorphisms, and thus that $X$ and $Y$ are isomorphic.
\end{proof}

Continuing the analysis, we have the following proposition:
\begin{prop}
\label{span-auts}
Let $\cC$ be a category. The group of isomorphism classes of automorphisms of $X$ in $\TwoSpan(\cC)$ is isomorphic to $\Aut_\cC(X)$.
\end{prop}
\begin{proof}
By \ref{span-isos}, any automorphism of $X$ in $\TwoSpan(\cC)$ looks like
$$X\stackrel{\sim}{\longleftarrow}X'\stackrel{\sim}{\longrightarrow}X.$$
But such a span diagram is uniquely isomorphic to exactly one of the form
$$X\stackrel{=}{\longleftarrow}X\stackrel{\sim}{\longrightarrow}X;$$
this proves the claim.
\end{proof}

\subsection{The $\Span$ quasicategory}
\label{span-quasicategory}

We now define a quasicategory $\Span$, one of the principal objects of study of this thesis, which is isomorphic to the nerve of the bicategory $\TwoSpan$.

Define $C_n$ to be the poset of nonempty subintervals $(i,i+1,\ldots,j)$ in $[n]=(0,\ldots,n)$, equipped with the \emph{reverse} inclusion ordering. We regard $C_n$ as a category.

The poset of nonempty subintervals of a totally ordered set is a functorial construction, so the collection $C=\{C_n\}$ forms a cosimplicial object in categories as we vary over all finite totally ordered sets $(0,\ldots,n)$.

This enables us to define a simplicial set, which we shall soon prove (in Proposition \ref{Span-structure}) to be a quasicategory:
\begin{defn}
We define the \emph{span quasicategory} $\Span(\cC)$ to be the simplicial set whose $n$-cells $\Span_n$ are the collection of functors $F:C_n\rightarrow\cC$ from $C_n$ to the category $\cC$, with the condition that, if $I$ and $J$ are two nonempty intervals in $[n]$ with nonempty intersection, the diagram
\begin{displaymath}
\xymatrix{
F(I\cup J)\pb{315}\ar[r]\ar[d]&F(I)\ar[d]\\
F(J)\ar[r]&F(I\cap J)}
\end{displaymath}
is a pullback.
\end{defn}
We refer to this condition later as the \emph{pullback property}. The collection $\Span(\cC)$ is indeed a simplicial set, since $C$ is a cosimplicial category, and taking faces and degeneracies preserves the pullback property.

If $Y:C_n\rightarrow\cC$ is an $n$-cell of the quasicategory $\Span(\cC)$, then we will write $Y_{ij}$ for $Y((i,\ldots,j))$ and $Y_i$ for $Y((i))$. If $i\leq i'\leq j'\leq j$, then we write $Y_{ij\rightarrow i'j'}$ for the structure map $Y_{ij}\rightarrow Y_{i'j'}$ induced by the inclusion.

Now, we have our formal statement:
\begin{prop}
\label{Span-structure}
Suppose $\cC$ is any ordinary category with pullbacks. Then we have an isomorphism of simplicial sets $N(\TwoSpan(\cC))\isom\Span(\cC)$. Thus $\Span(\cC)$ is a $(2,1)$-category and in particular (as suggested in the definition above) a quasicategory.
\end{prop}
\begin{proof}
We refer back to Section \ref{two-one-categories} for notation on bicategories. Given an $n$-cell $\{X_i,f_{ij},\theta_{ijk}\}\in N(\TwoSpan(\cC))_n$, we associate an $n$-cell $Y\in\Span(\cC)_n$.

We take $Y_i=X_i$ for all $i$. Further, we take $Y_{ij}$ to be the middle part of the span 1-cell given by $f_{ij}$, so we have a diagram $Y_i\leftarrow Y_{ij}\rightarrow Y_j$ for all $i<j$.

What is more, the 2-cell $\theta_{ijk}$ gives us a diagram as follows:
\begin{displaymath}
\xymatrix{&&X_{ik}\ar[dddll]\ar[dddrr]\ar[d]^{\wr}&&\\
          &&\bullet\pb{270}\ar[dl]\ar[dr]&&\\
          &X_{ij}\ar[dl]\ar[dr]&&X_{jk}\ar[dl]\ar[dr]&\\
          X_i&&X_j&&X_k,}
\end{displaymath}
for every $i<j<k$.

However, such diagrams are in 1-1 correspondence with diagrams
\begin{displaymath}
\xymatrix{&&X_{ik}\pb{270}\ar[dl]\ar[dr]&&\\
          &X_{ij}\ar[dl]\ar[dr]&&X_{jk}\ar[dl]\ar[dr]&\\
          X_i&&X_j&&X_k,}
\end{displaymath}
where the composite $X_i\leftarrow X_{ik}\rightarrow X_k$ is the given span diagram for $i<k$.

The compatibility condition gives all the other pullbacks, and the functoriality of the maps $X_{ij}\rightarrow X_{ij'}$ and $X_{ij}\rightarrow X_{i'j}$. 

This construction is reversible (and naturally commutes with faces and degeneracies) so we get an isomorphism of simplicial sets.

Proposition \ref{nerves-of-bicats} gives that $\Span(\cC)$ is thus a $(2,1)$-category (and a quasicategory in particular).
\end{proof}
As in section \ref{span-bicategory}, we write simply $\Span$ for the quasicategory $\Span(\Fin)$.

For example, an element of $\Span(\cC)_0$ is just an object $X_0$ of $\cC$, and an element of $\Span(\cC)_1$ consists of a diagram $\{X_0\leftarrow X_{01}\rightarrow X_1\}$ in $\cC$. We will specify them by writing them in this way.

In general, it is easy to see that an element of $\Span(\cC)_n$ is given by objects $X_{ij}$ of $\cC$ for $0\leq i\leq j\leq n$, and morphisms $X_{ij\rightarrow (i+1)j}$ and $X_{i\rightarrow i(j-1)}$, such that all diagrams of the form
\begin{displaymath}
\xymatrix{&X_{(i+1)(j-1)}&\\
          X_{i(j-1)}\ar[ur]&&X_{(i+1)j}\ar[ul]\\
          &X_{ij}\ar[ul]\ar[ur]&}
\end{displaymath}
commute, and are pullbacks.

We write $\hat C_n$ for the sub-poset of $C_n$ consisting of the proper subintervals of $\{0,\ldots,n\}$, so that $C_n\isom 1\star\hat C_n$ Then the following is an easy consequence of this description:
\begin{prop}
For $n\geq 2$, an $n$-cell of $\Span(\cC)$, regarded as a functor $C_n\rightarrow\cC$, is in fact a limit cone of its restriction to $\hat C_n$.\qed
\end{prop}

\subsection{Structure in $\Span$}
\label{span-structure}

We have a pair of functors
$$L:\Fins\longrightarrow\Span,\qquad R:\Finsop\longrightarrow\Span,$$
where $\Fins$ is the category of pointed finite sets, regarded as a bicategory with only trivial 2-cells (and thus as a quasicategory).

These are defined as follows. On 0-cells, $L(X)=R(X)=X$. Given $f:X_+\rightarrow Y_+$, we have
$$L(f)=\left(X\leftarrow f^{-1}(Y)\stackrel{f}{\rightarrow}Y\right)\qquad\text{and}\qquad R(f)=\left(Y\stackrel{f}{\leftarrow}f^{-1}(Y)\rightarrow X\right).$$

The composition 2-cell for $X\stackrel{f}{\rightarrow}Y\stackrel{g}{\rightarrow}Z$ under $L$ is given by the span diagram
\begin{displaymath}
\xymatrix@!R=8mm@!C=8mm{
&&(g\circ f)^{-1}(Z)\pb{270}\ar[dl]\ar[dr]&&\\
&f^{-1}(Y)\ar[dl]\ar[dr]&&g^{-1}(Z)\ar[dl]\ar[dr]&\\
X&&Y&&Z}
\end{displaymath}

and the one under $R$ is defined symmetrically to this.

There are simpler functors $\bar L:\Fin\rightarrow\Span$ and $\bar R:\Fin^\op\rightarrow\Span$. These can be defined as $\bar L=L\circ\pl$ and $\bar R=R\circ\pl^\op$, where $\pl$ is the ``plus functor'' given by adding a disjoint basepoint: $X\mapsto X_+$.

We observe that these functors satisfy
\begin{displaymath}
\bar L(f)=\left\{\xymatrix{X&\ar[l]_{\id}X\ar[r]^f&Y}\right\},\quad
\bar R(f)=\left\{\xymatrix{Y&\ar[l]_fX\ar[r]^{\id}&X}\right\}.
\end{displaymath}

In a similar fashion, there are functors $\bar L:\cC\rightarrow\Span(\cC)$ and $\bar R:\cC^\op\rightarrow\Span(\cC)$, for any category $\cC$. They are evidently faithful, and according to Proposition \ref{span-isos}, if $\bar L(f)$ or $\bar R(f)$ is an equivalence then $f$ is an isomorphism.

We move on to considering products in the quasicategory $\Span$.

\begin{prop}
The quasicategory $\Span$ has finite products. The product of objects $A$ and $B$ is $A\sqcup B$.
\end{prop}
\begin{proof}
Recall the definition of limits in quasicategories: if $f:K\rightarrow\Span$ is a morphism of simplicial sets, then a limit of $f$ is a terminal object of the over-category $\Span_{/f}$, given by
$$\left(\Span_{/f}\right)_n=\left\{\text{maps $\Delta_n\star K\rightarrow\Span$ which extend $f$}\right\}.$$

For us, $K=2=\{0,1\}$, with $f(0)=A$ and $f(1)=B$. Thus
\begin{eqnarray*}
\left(\Span_{/f}\right)_n&=&\left\{\text{maps $\Delta_n\star 2\rightarrow\Span$ which extend $f$}\right\}\\
&=&\big\{(X,Y)\in\Span_{n+1}^2\vert
             d_n X=d_n Y,\\
& &\hskip 25mm d_0d_1\cdots d_{n-1}X=A,\\
& &\hskip 25mm d_0d_1\cdots d_{n-1}Y=B\big\}\\
&=&\big\{\text{$X,Y:C_{n+1}\rightarrow\Fin$ with pullback property, such that}\\
& &\hskip 7mm\text{$X|_{C_n}=Y|_{C_n}$, $X(n+1)=A$ and $Y(n+1)=B$.}\big\}
\end{eqnarray*}

We now specify the object $P$ of $(\Span_{/f})_0$ which we claim is the product: it consists of the object $A\sqcup B\in\Span_0$, with projection maps $\{A\sqcup B\leftarrow A\rightarrow A\}$ and $\{A\sqcup B\leftarrow B\rightarrow B\}$.

We need to show that it is a strongly final object in $\Span_{/f}$. This means showing that any diagram $F:\partial\Delta^n\rightarrow\Span_{/f}$ with $F(n)=P$ extends to a diagram $\Delta^n\rightarrow\Span_{/f}$. This will be a straightforward, but notationally heavy, check.

Define $\hat C'_{n+1}$ to be the poset of subintervals of $\{0,\ldots,n+1\}$ that do not contain all of $\{0,\ldots,n\}$ (with the reverse inclusion order). Restricting to $\{0,\ldots,n\}$, we recover the poset $\hat C_n$ of proper subintervals of $\{0,\ldots,n\}$ defined in Section \ref{span-quasicategory}.

The simplicial structure on $\Delta^n$ guarantees that maps $\partial\Delta^n\rightarrow\Span_{/f}$ assemble to form diagrams $X,Y:\hat C'_{n+1}\rightarrow\Fin$ with the pullback property, such that $X|_{\hat C_n}=Y|_{\hat C_n}$, $X(n)=Y(n)=A\sqcup B$, $X(n,n+1)=X(n+1)=A$ and $Y(n,n+1)=Y(n+1)=B$.

For example, if $n=3$ the diagram $X$ is as follows:
\begin{displaymath}
\xymatrix@!R=6mm@!C=6mm{
&&&&&X(1,4)\ar[dl]\ar[dr]&&&\\
&&X(0,2)\ar[dl]\ar[dr]&&X(1,3)\ar[dl]\ar[dr]&&X(2,4)\ar[dl]\ar[dr]&&\\
&X(0,1)\ar[dl]\ar[dr]&&X(1,2)\ar[dl]\ar[dr]&&X(2,3)\ar[dl]\ar[dr]&&A\ar[dl]\ar[dr]&\\
X(0)&&X(1)&&X(2)&&A\sqcup B&&A,}
\end{displaymath}
and the diagram $Y$ is similar.

We can extend these to $C_{n+1}$ by defining
\begin{eqnarray*}
X(0,n)  &=&{\lim}_{\hat C_n}X,\\
Y(0,n)  &=&{\lim}_{\hat C_n}Y,\\
X(0,n+1)&=&{\lim}_{\hat C'_n}X,\\
Y(0,n+1)&=&{\lim}_{\hat C'_n}Y.
\end{eqnarray*}
We clearly have $X|_{C_n}=Y|_{C_n}$, have $X(n+1)=A$ and $Y(n+1)=B$ by definition, and it is quick to check the pullback property.
\end{proof}

The same proof suffices to prove the following:
\begin{prop}
\label{prod-span-c}
For any category $\cC$ with finite coproducts and finite limits, finite products in $\Span(\cC)$ exist, and are given on objects by coproducts in $\cC$. The inclusion maps are defined analogously to the case $\cC=\Fin$ above.\qed
\end{prop}

As an important corollary, we have:
\begin{prop}
The functor $R$ makes the category $\Span$ into an algebraic theory, as introduced in Definition \ref{theory-defn}.\qed
\end{prop}
Accordingly, since $\Span$ was motivated by the desire to produce a quasicategorical version of the theory of monoids, we define:
\begin{defn}
Let $\cC$ be a quasicategory with finite products. A \emph{Lawvere monoid object} in $\cC$ is a model of $\Span$ in $\cC$: a product-preserving functor $\Span\rightarrow\cC$.
\end{defn}

Also, since $\Span$ is self-opposite, we have
\begin{prop}
The category $\Span(\cC)$ has coproducts, which agree with products.\qed
\end{prop}
An immediate consequence of this is that the theory $\Span$ is pointed, as defined in Section \ref{pointed-theories}; various important consequences of this are given there too.

\subsection{The category $\Spant$}

Now we introduce a category $\Spant$. Using the notation of subsection \ref{span-quasicategory}, we define $\Spant=\Span(\Arr(\Fin))$, where $\Arr(\Fin)$ is the category of arrows in $\Fin$.

So an $n$-cell of $\Spant$ is a pair of span diagrams $\{X_{ij}\},\{Y_{ij}\}\in\Span_n$ with maps $f_{ij}:X_{ij}\rightarrow Y_{ij}$, which commute with all the structure maps. 

Equivalently, it's a natural transformation between functors $X,Y:C_n\rightarrow \Fin$, where both $X$ and $Y$ have the pullback property.

There's a 2-functor $p:\Spant\rightarrow\Span$ coming from the functor $\Arr(\Fin)\rightarrow\Fin$ which sends $(X\rightarrow Y)$ to $Y$. According to the description above, this sends a morphism of span diagrams to the codomain.

Now, we want to study this functor. First we find a good supply of $p$-cartesian morphisms (as introduced in Definition \ref{defn-cartesian-fibration} above).

\begin{prop}
\label{spant-span-cartesian-morphisms}
Any 1-cell of $\Spant$ of the form
\begin{displaymath}
\xymatrix{Y_0\ar[d]&Y_{01}\ar[l]_{=}\ar[r]\ar[d]\pb{315}&Y_1\ar[d]\\
          X_0      &X_{01}\ar[l]    \ar[r]              &X_1.}
\end{displaymath}
ie. which has the top left map the identity and right-hand square a pullback, is $p$-cartesian.
\end{prop}
\begin{proof}
By Proposition \ref{acyclic-kan}, there are four checks to make on the functor
$$\Spant_{/f}\longrightarrow(\Spant_{/y})\timeso{\Span_{/py}}(\Span_{/pf})$$
to show that it is an acyclic Kan fibration: we must check it has the right lifting property with respect to $\partial\Delta^m\rightarrow\Delta^m$ for $m\leq 3$. We are using the notation $y$ for the 0-cell of $\Spant$ given by $Y_4\rightarrow X_4$.

Firstly, we show the existence of liftings for $\emptyset\rightarrow\Delta^0$.

Given a diagram like the following, which represents a $0$-cell of $\Spant_{/y}\timeso{\Span_{/py}}\Span_{/pf}$,
\begin{displaymath}
\xymatrix{
  &&Y_{24}\ar[dldl]\ar[drdr]\ar[d]&&\\
  &&X_{24}\pb{270}\ar[dl]\ar[dr]&&\\
  Y_2\ar[d]&X_{23}\ar[dl]\ar[dr]&&X_{34}\ar[dl]\ar[dr]&Y_4\ar[d]\\
  X_2&&X_3&&X_4,}
\end{displaymath}
we can fill it in to form a full 0-cell of $\Spant_{/f}$ as follows:
\begin{displaymath}
\xymatrix{
  &&Y_{24}\ar@<-2ex>[dldl]\dar[dl]^{=}\ar@<2ex>[drdr]\dar[dr]\ar[d]&&\\
  &Y_{24}\dar[d]\dar[dl]\dar[dr]&X_{24}\pb{270}\ar[dl]\ar[dr]&Y_{34}\dar[d]\dar[dl]^>>>{=}\dar[dr]&\\
  Y_2\ar[d]&X_{23}\ar[dl]\ar[dr]&Y_{34}\dar[d]&X_{34}\ar[dl]\ar[dr]&Y_4\ar[d]\\
  X_2&&X_3&&X_4,}
\end{displaymath}
and this is the required lifting.

Next, a diagram
\begin{displaymath}
\xymatrix{
\partial\Delta^1\ar[r]\ar[d]&\Spant_{/f}\ar[d]\\
\Delta^1\ar[r]&(\Spant_{/y})\timeso{\Span_{/py}}(\Span_{/pf})}
\end{displaymath}
gives us a configuration of $Y$'s as follows:
\begin{displaymath}
\xymatrix{
&&&Y_{14}\ar[dl]\ar[dr]\ar@<2ex>[drdr]\pb{270}&&&\\
&&Y_{13}\ar[dl]\ar@<2ex>[drdr]&&Y_{24}\ar[dl]\ar[dr]\pb{270}&&\\
&Y_{12}\ar[dl]\ar[dr]&&Y_{23}\ar[dl]\ar[dr]&&Y_{34}\ar[dl]_{=}\ar[dr]&\\
Y_1&&Y_2&&Y_3&&Y_4,}
\end{displaymath}
where all squares commute and are pullbacks. There is also a full diagram of $X_{ij}$'s, and maps $Y_{ij}\rightarrow X_{ij}$. The parallel morphisms $Y_{14}\rightarrow Y_{34}$ do not have to agree \emph{prima facie}, but the composites $Y_{14}\rightarrow Y_4$ do agree. This maps to a complete span diagram of $X$'s in the obvious way.

However, since $Y_{34}$ is a pullback, the parallel morphisms into it do commute (since the two composites into $Y_4$ and $X_{34}$ do agree).

The maps $Y_{14}\rightarrow Y_{13}$ and $Y_{24}\rightarrow Y_{23}$ are isomorphisms, since they're pullbacks of an isomorphism. This allows us to define a map $Y_{13}\rightarrow Y_{23}$, which makes the resulting top and left squares into pullbacks. Finally, the resulting parallel pair of morphisms $Y_{13}\rightarrow Y_3$ agree, since they are isomorphic to the pair considered earlier.

Now we brace ourselves and consider liftings for $\partial\Delta^2\rightarrow\Delta^2$. Here the morphism $\partial\Delta^2\rightarrow\Spant_{/f}$ gives us a diagram like
\begin{displaymath}
\xymatrix{
&&&&Y_{04}\ar[dl]\ar[dr]\ar@<2ex>[drdr]\pb{270}&&&&\\
&&&Y_{03}\ar[dl]\ar[dr]\ar@<2ex>[drdr]\pb{270}&&Y_{14}\ar[dl]\ar[dr]\pb{270}&&&\\
&&Y_{02}\ar[dl]\ar[dr]\ar@<2ex>[drdr]\pb{270}&&Y_{13}\ar[dl]\ar[dr]\pb{270}&&Y_{24}\ar[dl]\ar[dr]\pb{270}&&\\
&Y_{01}\ar[dl]\ar[dr]&&Y_{12}\ar[dl]\ar[dr]&&Y_{23}\ar[dl]\ar[dr]&&Y_{34}\ar[dl]\ar[dr]&\\
Y_0&&Y_1&&Y_2&&Y_3&&Y_4}
\end{displaymath}
Here all squares are pullbacks, but it is not given that the parallel pairs agree. However, the morphism $\Delta^2\rightarrow\Spant_{/y}$ gives us exactly this necessary extra coherence data, completing this check.

Lastly, it is straightforward to check that, given a lifting problem for $\partial\Delta^3\rightarrow\Delta^3$, all data is given and is coherent: we get a complete span diagram.
\end{proof}

\begin{prop}
\label{span-cart-fib}

The map $p:\Spant\rightarrow\Span$ is a cartesian fibration.

\end{prop}
\begin{proof}
Firstly, we show that the map is an inner fibration. By Proposition \ref{inner-fibs}, we need only check horn extensions for $\Lambda^2_1\rightarrow\Delta^2$.
This gives us the following diagram:
\begin{displaymath}
\xymatrix{
&X_{01}\ar[dl]\ar[dr]\ar[d]&Y_{02}\ar[dl]\ar[dr]\pb{270}&X_{12}\ar[dl]\ar[dr]\ar[d]&\\
X_0\ar[d]&Y_{01}\ar[dl]\ar[dr]&X_1\ar[d]&Y_{12}\ar[dl]\ar[dr]&X_2\ar[d]\\
Y_0&&Y_1&&Y_2.}
\end{displaymath}
This can be filled in to a full map of span diagrams by taking $X_{02}$ to be the pullback of $X_{01}\rightarrow X_1\leftarrow X_{12}$; this maps to $Y_{02}$ in an appropriate manner.

Given a 1-cell $X_3\leftarrow X_{34}\rightarrow X_4$ of $\Span$ (the numbering will make sense later) and an 0-cell $Y_4\rightarrow X_4$ of $\Spant$, we need to find a $p$-cartesian morphism of $\Spant$ which restricts to these two.

But we can define $Y_{34}$ to form a 1-cell of $\Spant$ as follows:
\begin{displaymath}
\xymatrix{Y_{34}\ar[d]&Y_{34}\ar[l]_{=}\ar[r]\ar[d]\pb{315}&Y_4\ar[d]\\
          X_3         &X_{34}\ar[l]    \ar[r]              &X_4.}
\end{displaymath}
This is $p$-cartesian by Proposition \ref{spant-span-cartesian-morphisms} above.
\end{proof}

This construction is compatible with the construction by Lurie \cite{DAG-III} of the cartesian fibration $\Fint\rightarrow\Fin$, in the following sense:
\begin{prop}
There is a commuting diagram
\begin{displaymath}
\xymatrix{\Finst\ar[r]^{L^\times}\ar[d]&\Spant\ar[d]\\
          \Fins \ar[r]_{L}             &\Span.}
\end{displaymath}
Here the functor $L$ is as defined in Section \ref{span-structure}.
\end{prop}

\begin{proof}
We must define $L^\times$; we do this by analogy with the construction of $L$ above. As Lurie defines it, $\Finst$ is essentially the category of monomorphisms of pointed finite sets. This admits a natural functor into the category of arrows of pointed finite sets. Given a diagram of arrows of pointed finite sets, we can perform $L$ on it levelwise (that is, take $L$ of the domains and $L$ of the codomains). Each gives us a diagram in $\Span$, and there is a map between them, which means it assembles to a diagram in $\Spant$.

It is then immediate that the square in the Proposition commutes.
\end{proof}

\subsection{Cartesian morphisms for $\Spant\rightarrow\Span$}

In this section we classify all morphisms which are $p$-cartesian, where $p:\Spant\rightarrow\Span$ is the natural projection map.

For convenience of notation, we will work with the equivalent notion in the opposite categories: classifying $p^\op$-cocartesian morphisms where $p^\op$ is the corresponding morphism ${\Spant}^\op\rightarrow\Span^\op$.

In the proof of Proposition \ref{span-cart-fib}, we showed that a 1-cell $F\in{\Spant}^\op_1$ given by
\begin{displaymath}
  \xymatrix{X_0\ar[d]&X_{01}\ar[l]_{\lambda^X_{01}}\ar[r]^{\rho^X_{01}}\ar[d]&X_1\ar[d]\\
            Y_0      &Y_{01}\ar[l]^{\lambda^Y_{01}}\ar[r]_{\rho^Y_{01}}      &Y_1}
\end{displaymath}
is $p^\op$-cocartesian if the morphism $\lambda^X_{01}$ is an isomorphism, and if the right-hand square is a pullback square.

We write $T_F$ for ${\Spant}^\op_{F_0/}\times_{(\Span^\op_{X_0/})}\Span^\op_{X/}$.

The argument depends on the diagrams used in the proof of Proposition \ref{span-cart-fib}. We will take to drawing the bottom part of a span upside-down: this will simplify the diagrams in practice.

\begin{prop}
\label{cart-pb-surj}
If $F$ is $p^\op$-cocartesian, then the natural map $X_{01}\rightarrow X_0\times_{X_0}X_{01}$ is surjective.
\end{prop}
\begin{proof}
  Given an element $(x,y)\in X_0\times_{Y_0}Y_{01}$, the solid arrows of the following diagram describe a cell $\Delta^0\rightarrow T_F$:
\begin{displaymath}
 \xymatrix{&&1\ar@/_2ex/[ddll]_x\dar[dl]\ar[ddrr]^=\ar@/^1.7ex/[ddddd]^y&&\\
           &X_{01}\ar[ddd]\ar[dl]\ar[dr]&&&\\
           X_0\ar[d]&&X_1\ar[d]&&1\ar[d]\\
           Y_0&&Y_1&&1\\
           &Y_{01}\ar[ul]\ar[ur]&&Y_1\ar[ul]\ar[ur]&\\
           &&Y_{01}\ar[ul]\ar[ur]&&}
\end{displaymath}
We are assuming that an extension to a cell $\Delta^0\rightarrow\Spant_{F/}$ exists; this provides us with the dotted arrow $1\rightarrow X_{01}$: an element of $X_{01}$ which maps to $(x,y)$. This proves surjectivity.
\end{proof}

\begin{prop}
\label{cart-topmap-surj}
If $F$ is $p^\op$-cocartesian, then the map $\rho_{01}^X:X_{01}\rightarrow X_1$ is surjective.
\end{prop}
\begin{proof}
Suppose this is not the case: that $x\in X_1$ has no preimage in $X_{01}$.

We consider a lifting problem for $\partial\Delta^1\rightarrow\Delta^1$ along $\Spant_{/F}\rightarrow T_F$. The data of such a situation is specified by solid arrows of the following diagram:
\begin{displaymath}
 \xymatrix{
&&&0\ar[dl]\ar[dr]&&&\\
&&0\ar[dl]\ar[dr]&&1\ar@/^1.7ex/[ddll]\dar[dl]\ar[dr]&&\\
&X_{01}\ar[dl]\ar[dr]\ar[ddd]&&0\ar[dl]\ar[dr]&&1\ar[dl]\ar[dr]&\\
X_0\ar[d]&&X_1\ar[d]&&1\ar[d]&&1\ar[d]\\
Y_0&&Y_1&&1&&1\\
&Y_{01}\ar[ul]\ar[ur]&&Y_1\ar[ul]\ar[ur]&&1\ar[ul]\ar[ur]&\\
&&Y_{01}\ar[ul]\ar[ur]&&Y_1\ar[ul]\ar[ur]&&\\
&&&Y_{01}.\ar[ul]\ar[ur]&&&}
\end{displaymath}
By hypothesis, all the squares in each half are pullbacks.

Since $F$ is assumed to be $p^\op$-cocartesian, an extension exists along the dotted line: a contradiction.
\end{proof}

\begin{prop}
\label{cart-topmap-inj}
If $F$ is $p^\op$-cocartesian, then the map $\rho_{01}^X:X_{01}\rightarrow X_1$ is injective.
\end{prop}
\begin{proof}
Suppose not: that there is $x\in U_1$ with $P={\rho_{01}^X}^{-1}(x)$ a set of size at least 2. Then there is a nontrivial automorphism $\alpha$ of $P$.

We now consider the following lifting problem for $\partial\Delta^1\rightarrow\Delta^1$ along $\Spant_{F/}\rightarrow T_F$, where $i$ is the inclusion $P\rightarrow X_{01}$, and the top parallel collection of morphisms need not commute:
\begin{displaymath} 
\xymatrix{
&&&P\ar[dl]^\alpha\ar@/_1.7ex/[ddll]_i\ar[dr]&&&\\
&&P\ar[dl]^i\ar[dr]&&1\ar@/^1.7ex/[ddll]^(.75){x}\ar[dr]&&\\
&X_{01}\ar[dl]\ar[dr]\ar[ddd]&&1\ar[dl]_x\ar[dr]&&1\ar[dl]\ar[dr]&\\
X_0\ar[d]&&X_1\ar[d]&&1\ar[d]&&1\ar[d]\\
Y_0&&Y_1&&1&&1\\
&Y_{01}\ar[ul]\ar[ur]&&Y_1\ar[ul]\ar[ur]&&1\ar[ul]\ar[ur]&\\
&&Y_{01}\ar[ul]\ar[ur]&&Y_1\ar[ul]\ar[ur]&&\\
&&&Y_{01}.\ar[ul]\ar[ur]&&&}
\end{displaymath}
Again, all squares are pullbacks. By assumption this lifts to a complete diagram $\Delta^1\rightarrow\Spant_{F/}$, meaning that $i\alpha=i$, meaning that $\alpha$ is trivial: a contradiction.
\end{proof}

\begin{prop}
\label{cart-pb-inj}
If $F$ is $p^\op$-cocartesian, then the natural map $X_{01}\rightarrow X_0\times_{Y_0}Y_{01}$ is injective.
\end{prop}
\begin{proof}
Let $(x,y)$ be any element of $X_0\times_{Y_0}Y_{01}$, and let $a,a'$ be two elements of the preimage. We consider another lifting problem for $\partial\Delta^1\rightarrow\Delta^1$ along $\Spant_{F/}\rightarrow T_F$, where again the top parallel collection of morphisms need not commute:
\begin{displaymath} 
\xymatrix{
&&&1\ar[dl]\ar@/_1.7ex/[ddll]_{a'}\ar[dr]&&&\\
&&1\ar[dl]^a\ar[dr]&&1\ar@/^1.7ex/[ddll]\ar[dr]&&\\
&X_{01}\ar[dl]\ar[dr]\ar[ddd]&&1\ar[dl]\ar[dr]&&1\ar[dl]\ar[dr]&\\
X_0\ar[d]&&X_1\ar[d]&&1\ar[d]&&1\ar[d]\\
Y_0&&Y_1&&1&&1\\
&Y_{01}\ar[ul]\ar[ur]&&Y_1\ar[ul]\ar[ur]&&1\ar[ul]\ar[ur]&\\
&&Y_{01}\ar[ul]\ar[ur]&&Y_1\ar[ul]\ar[ur]&&\\
&&&Y_{01}.\ar[ul]\ar[ur]&&&}
\end{displaymath}
The fact that appropriate pullback squares exist follows from Propositions \ref{cart-topmap-surj} and \ref{cart-topmap-inj}. Since $F$ is assumed $p^\op$-cocartesian, the lifting gives us that $a=a'$.
\end{proof}

\begin{thm}
 If $F$ is given by
\begin{displaymath}
  \xymatrix{X_0\ar[d]&X_{01}\ar[l]_{\lambda^X_{01}}\ar[r]^{\rho^X_{01}}\ar[d]&X_1\ar[d]\\
            Y_0      &Y_{01}\ar[l]^{\lambda^Y_{01}}\ar[r]_{\rho^Y_{01}}      &Y_1,}
\end{displaymath}
then it is $p$-cartesian if and only if the right-hand square is a pullback and $\lambda^X_{01}$ is an isomorphism. 
\end{thm}
\begin{proof}
 One direction is Proposition \ref{span-cart-fib}, the other is jointly implied by Propositions \ref{cart-pb-surj}, \ref{cart-topmap-surj}, \ref{cart-topmap-inj}, and \ref{cart-pb-inj}.

 We note that we have not used the lifting condition for $\partial\Delta^2\rightarrow\Delta^2$, and deduce that it is automatically satisfied in the presence of the others: this is apparently not otherwise clear.
\end{proof}

\subsection{Lawvere symmetric monoidal structures}
\label{lawvere-sym-mon-structures}

Given a quasicategory $\cC$ with cartesian products, we shall produce a model of $\Span$ in quasicategories (as defined in subsection \ref{models-in-quasicategories}).

First we define an auxiliary category $\tcCt$. For $K\rightarrow\Span$, we define $\tcCt$ to be the simplicial set represented by the following functor in $K$:
$$\Hom_{\Span}(K,\tcCt)=\Hom(K\timeso{\Span}\Spant,\cC).$$
(It is straightforward to check that this functor does indeed preserve colimits.)

This has the following important structural property: 
\begin{prop}
The projection $\tp:\tcCt\rightarrow\Span$ is a cocartesian fibration.
\end{prop}
\begin{proof}
This is analogous to \cite{DAG-III}*{Proposition 2.8.1}, but, finding that argument a little concise, we fill in the details.

We have
$$\Hom_\Span(K,\tcCt)=\Hom(K\timeso{\Span}\Spant,\cC)=\Hom_\Span(K\timeso{\Span}\Spant,\cC\times\Span).$$

The map $p:\Spant\rightarrow\Span$ was shown to be a cartesian fibration in Proposition \ref{span-cart-fib}. Since the map $\cC\rightarrow 1$ is evidently a cocartesian fibration, the projection $q:\cC\times\Span\rightarrow\Span$ is also a cocartesian fibration (by \cite{HTT}, 2.3.2.3).

These two maps satisfy the hypotheses for $p$ and $q$ respectively in \cite{HTT}*{Lemma 3.2.2.13}, and so the proposition is proved.
\end{proof}

We can describe the fibre $\tcCt_A$ of $\tcCt$ over a finite set $A\in\Span_0$:
\begin{prop}
\label{fibre-description}
 $$\tcCt_A=\Hom(\Span^A, \cC).$$
\end{prop}
\begin{proof}
 We have:
\begin{align*}
 (\tcCt_A)_n&=\left\{\text{maps}\vcenter{\xymatrix{\Delta^n\ar[rr]\ar[dr]_A&&\tcCt\ar[dl]\\&\Span&}}\right\}\\
            &=\Hom(\Delta^n\timeso{\Span}\Spant,\cC)\\
            &=\Hom(\Delta^n\times\Span(\Fin_{/A}),\cC),
\end{align*}
and so $\tcCt_A$ can be identified with the simplicial set of functors, from the category $\Span(\Fin_{/A})$ of spans of finite sets over $A$, into $\cC$. But since $\Span(\Fin_{/A})=\Span(\Fin^A)=\Span^A$, we get:
\begin{align*}
 \tcCt_A&=\Hom(\Span(\Fin_{/A}), \cC)\\
        &=\Hom(\Span^A,\cC).\qedhere
\end{align*}
\end{proof}

This description allows us to analyse the $\tp$-cocartesian morphisms in $\tcCt$:
\begin{prop}
  Let $\alpha$ be a morphism in $\tcCt$, with image the 1-cell $X\stackrel{f}{\leftarrow}Z\stackrel{g}{\rightarrow}Y$ in $\Span$. Then $\alpha$ is $\tp$-cocartesian if and only if, for every $U\subset Y$, the morphism $\alpha$ takes the 1-cell
\begin{displaymath}
 \xymatrix{g^*U\ar[d]&g^*U\ar[l]\ar[d]\ar[r]&U\ar[d]\\
           X         &Z   \ar[l]      \ar[r]&Y}
\end{displaymath}
of $\Delta^1\timeso{\Span}\Spant$ to an equivalence in $\cC$.
 
\end{prop}
\begin{proof}
First, we study the morphisms which are cocartesian with respect to the projection functor $q:\cC\times\Span\rightarrow\Span$.

But Proposition \ref{prods-and-carts} makes it clear that these are the products of the morphisms which are cocartesian for $\cC\rightarrow 1$ and those which 
are cocartesian for the identity on $\Span$. By \cite{HTT}*{Remark 2.3.1.4}, the former morphisms are the equivalences in $\cC$, and the latter are all the morphisms in $\Span$.

So the $q$-cocartesian morphisms are those which are an equivalence on the left factor.

With these preliminaries, the result follows by invoking \cite{HTT}*{Lemma 3.2.2.13}.
\end{proof}

Since Proposition \ref{fibre-description} describes the fibre of $\cC\rightarrow\Span$ over an object $A\in\Span_0\isom\Fin_0$ as the category of functors $\Span^A\rightarrow\cC$. We can thus define $\cCt$ to be the full subcategory whose objects are all the product-preserving functors $\Span^A\rightarrow\cC$ for all $A\in\Span_0$.

We define $p:\cCt\rightarrow\Span$ to be the restriction of $\tp$ to $\cCt$, and begin to amass good properties of this functor.
\begin{prop}
\label{cocart-from-prods}
The projection $p:\cCt\rightarrow\Span$ is a cocartesian fibration, with the same cocartesian morphisms as $\tcCt\rightarrow\Span$.
\end{prop}
\begin{proof}
The map $p$ is evidently an inner fibration, since it's a restriction of an inner fibration to a full subcategory.

  So we just need to demonstrate that, given a span $\alpha=(X\stackrel{f}{\leftarrow}Z\stackrel{g}{\rightarrow}Y)$ a 1-cell in $\Span$, and an element $F$ of the fibre $\cCt_Y$, there is cocartesian lift of $\alpha$ in $\tcCt$, whose left-hand vertex $G$ is an element of the fibre $\cCt_X$.

But, in terms of the description of the fibre we are given, if
$$F:\Span(\Fin_{/Y})\rightarrow\cC$$
is product-preserving, then the vertex of the natural lift is given by
$$G(A\rightarrow X)=F(f^*A\rightarrow Y),$$
which is clearly a product-preserving functor $\Span(\Fin_{/X})\rightarrow\cC$.
\end{proof}

\begin{prop}
\label{lifting-prop}
Assume that $\cC$ has finite products. The projection $p:\cCt\rightarrow\Span$ has the property that, given a coproduct diagram in $\Span$, consisting of $A\sqcup B\leftarrow A\rightarrow A$ and $A\sqcup B\leftarrow B\rightarrow B$, the corresponding functors realise $p^{-1}(A\sqcup B)$ as the product of $p^{-1}(A)$ and $p^{-1}(B)$.
\end{prop}
\begin{proof}
The fibre $\cCt_A$ over a finite set $A$ is the quasicategory of product-preserving functors $\Span^A\rightarrow\cC$. If $\cC$ has products, this is isomorphic to $\cC^A$.

It is quick to check that the morphisms given realise the product structures correctly.
\end{proof}

Motivated by the above propositions, we make the following definition (recalling the definition of a cocartesian fibration from \ref{quasicategory-theory}):
\begin{defn}
A \emph{Lawvere symmetric monoidal structure} is a model of $\Span$ in quasicategories: a cocartesian fibration $\cCt\rightarrow\Span$ such that preimages of product diagrams in $\Span$ are product diagrams of categories.
\end{defn}

So, the propositions above assemble to:
\begin{thm}
An quasicategory with finite products gives a Lawvere symmetric monoidal structure.
\end{thm}
\begin{proof}
This is Propositions \ref{cocart-from-prods} and \ref{lifting-prop}.
\end{proof}

We need notions of functor too:
\begin{defn}
A \emph{symmetric monoidal functor} between two Lawvere symmetric monoidal structures $\cCt\stackrel{p}{\rightarrow}\Span$ and $\cDt\stackrel{q}{\rightarrow}\Span$ is a map of models of $\Span$ in quasicategories: that is, a functor taking $p$-cocartesian morphisms to $q$-cocartesian morphisms.

A \emph{lax symmetric monoidal functor} is one which takes $p$-cocartesian morphisms whose image in $\Span$ is collapsing to $q$-cocartesian morphisms.
\end{defn}

\subsection{Lawvere commutative algebra objects}

Now we seek to define an algebra object in a Lawvere symmetric monoidal category.

We say that an 1-cell in the category $\Span$ is \emph{collapsing} if it is of the form
$$\xymatrix{X&Z\iar[l]\ar[r]^\sim&Y},$$
or equivalently if it's isomorphic to a diagram of the form
$$\xymatrix{A\sqcup B&A\iar[l]\ar[r]^=&A}.$$
We can thus define a subquasicategory $\Spanc$ of $\Span$ containing all objects, all collapsing 1-cells and all higher cells all of whose edges are collapsing.

Recall that Lurie \cite{DAG-III}*{Section 1.9}, defines a morphism in $\Fins$ to be collapsing if all the preimage of every element except the basepoint has size exactly one. By the same process we can define a subquasicategory $\Finsc$, and it is quick to check we have a pullback diagram
$$\xymatrix{\Finsc\pb{315}\ar[d]\ar[r]&\Spanc\ar[d]\\\Fins\ar[r]_R&\Span.}$$
In other words: a morphism in $\Fins$ is collapsing if and only if it yields a collapsing span diagram.

We define an algebra object in a manner analogous to Lurie's definition: a \emph{Lawvere commutative algebra object} of a Lawvere symmetric monoidal category $\cCt\stackrel{p}{\longrightarrow}\Span$ consists of a section $f$ of $p$ which takes collapsing morphisms to $p$-cocartesian morphisms.

We can then define the quasicategory of Lawvere commutative algebra objects $\Algt(\cCt)$ to be the full subquasicategory of $\Fun(\Span,\cCt)$ on the Lawvere commutative algebra objects. When the monoidal structure is understood, which is most of the time, we write $\Algt(\cC)$ instead.

\section{A comparison of the Lawvere and Lurie approaches}

The purpose of this section is to show that Lurie's definition of a symmetric monoidal category in \cite{DAG-III} is equivalent to the one advanced in the preceding section.

We begin with some reasonably lightweight sections, which sketch straightforward simple arguments why $\Span$ should be thought of as the theory for commutative monoids in the setting of quasicategories.

\subsection{Discrete Lawvere commutative monoids}

Lurie shows \cite{HTT}*{1.2.3.1} that the nerve functor (of ordinary categories) from ordinary categories to quasicategories has a right adjoint, denoted $\Ho$, the \emph{homotopy category}. We use this theory briefly to understand how $\Span$ really is a quasicategorical version of the Lawvere theory of commutative monoids.

We can compute the homotopy category of $\Span$:

\begin{prop}
\label{ho-span}
 The homotopy category $\Ho(\Span)$ of $\Span$ has finite sets as objects and isomorphism classes of span diagrams as morphisms. Composition is by pullback.
\end{prop}
\begin{proof}
 The only check is that the composition in $\Span$ respects isomorphism classes; this is evident.
\end{proof}

This category is recognisable as the Lawvere theory for discrete commutative monoids.

This allows us to state the following:
\begin{thm}
\label{discrete-monoids}
Discrete Lawvere commutative monoid objects are the same as commutative monoids.
\end{thm}
\begin{proof}

Since, as mentioned above, $\Ho$ is the right adjoint of the nerve functor $\Cat\rightarrow\sSet$, for any quasicategory $\cC$ all product-preserving functors $\Span\rightarrow\cC$ factor uniquely through the product-preserving functor $\Span\rightarrow\Ho(\Span)$.

\end{proof}

This will be evident from the structure results proven later in this section, but this elementary argument is nevertheless informative, in our opinion.

\subsection{$\Span$ and the Barratt-Eccles operad}
\label{homs-in-span}

Another small piece of propaganda is provided by calculating the homspaces in the quasicategory $\Span$; this makes plausible much of the relationship between $\Span$ and various recognisably classical notions such as the Barratt-Eccles operad \cite{BarEcc}.

One should think of $\Span$ as a homotopy-theoretic elaboration of the theory of commutative monoids. Indeed, the theory of commutative monoids is the opposite of the full subcategory of $\Monoids$ on the objects $\N^r$, and there is a natural forgetful morphism $\Span\rightarrow\Th_\Monoids$.

An object $X\in\Span_0$ is sent to the object $\N^X$, and a span $X\stackrel{f}{\leftarrow} Z\stackrel{g}{\rightarrow} Y$ is sent to the map

\begin{align*}
\N^X&\longleftarrow\N^Y\\
y\in Y&\longmapsto \sum_{g(z)=y}f(z).
\end{align*}

Thus the theory $\Span$ consists of finitely-generated free commutative monoids, with some unusual autoequivalences on the morphisms.

Writing $\FSI$ for the category of finite sets and isomorphisms, we have
\begin{align*}
\Span(1,1)&=N(\FSI)\\
          &\isom\coprod_{X\in\operatorname{skeleton}(\Fin)}B\Aut(X)\\
          &\isom\coprod_n B\Sigma_n
\end{align*}
(where the sum on the second line is taken over a set of representatives of the isomorphism classes)
This is a famous model for the free $E_\infty$-monoid on one generator: it's that provided by the Barratt-Eccles operad.

The composition of two spans $1\leftarrow X\rightarrow 1$ and $1\leftarrow Y\rightarrow 1$ is given by $1\leftarrow X\times Y\rightarrow 1$, and on hom-spaces is given by the maps $B\Sigma_n\times B\Sigma_{n'}\rightarrow B\Sigma_{n\times n'}$ induced by the Cartesian product map $\Sigma_n\times\Sigma_{n'}\rightarrow\Sigma_{n\times n'}$.

Moreover, we also have
\begin{align*}
\Span(X,Y)&=N(\FSI_{/X,Y})\\
          &=N(\FSI_{/X\times Y})\\
          &=\left(\coprod_n B\Sigma_n\right)^{X\times Y},
\end{align*}
so we can think of the space of spans from $X$ to $Y$ as being $X$-by-$Y$ matrices with entries in the free $E_\infty$-monoid on one generator. Composition is then matrix multiplication, and it is easily seen that coproducts and products are given on homspaces by block sums of matrices.

\subsection{General comparisons}

Our construction has strictly more data than Lurie's construction:
\begin{thm}
Any Lawvere monoidal quasicategory has an ``underlying'' Lurie symmetric monoidal quasicategory $\cCot\rightarrow\Fins$.
\end{thm}
\begin{proof}
We define $\cCot$ to be the pullback 
\begin{displaymath}
 \xymatrix{\cCot\ar[r]\ar[d]&\cCt\ar[d]^p\\
           \Fins\ar[r]_L    &\Span}
\end{displaymath}
The left-hand arrow is a cartesian fibration, by \cite{HTT}*{Lemma 2.4.2.3}. The required product property is immediate, since unions in $\Fins$ get sent to product diagrams in $\Span$.
\end{proof}

We note that, by further restriction of structure, we can think of $p^{-1}(*)$ as the \emph{underlying quasicategory} of $\cCt$; we also get an underlying monoidal category in the sense of \cite{DAG-II}.

Moreover, maps of Lawvere symmetric monoidal quasicategories yield maps of their underlying Lurie symmetric monoidal quasicategories.

Similarly, commutative algebra objects for $\cCt$ yield commutative algebra objects in $\cCot$: given a section $a$ of $\cCt\rightarrow\Span$, we can pull back to obtain a section of $\cCot\rightarrow\Fins$:
\begin{displaymath}
 \xymatrix{\Fins\ar[r]\ar[d]\pb{315}&\Span\ar[d]\\
           \cCot\ar[r]\ar[d]\pb{315}&\cCt \ar[d]\\
           \Fins\ar[r]              &\Span.}
\end{displaymath}
This defines a functor $\theta:\Algt(\cCt)\rightarrow\Algot(\cCot)$.

Recall from Subsection \ref{span-structure} that we define a \emph{Lawvere monoid object} in a category $\cC$ to be a model of $\Span$ in $\cC$: that is, a product-preserving functor $\Span\rightarrow\cC$.

Similarly, for the present argument we define a \emph{Lurie monoid object} in $\cC$ to be a functor $\Fins\rightarrow\cC$ which takes the projection maps of disjoint unions to product diagrams in $\cC$. (This agrees with the definition in \cite{DAG-III}).

We can form quasicategories $\Mont(\cC)$ and $\Monot(\cC)$ of Lawvere and Lurie monoid objects respectively, as full subquasicategories of the functor categories $\Fun(\Span,\cC)$ and $\Fun(\Fins,\cC)$ on the monoid objects.

The good properties of the functor $L:\Fins\rightarrow\Span$ defines a functor $\varphi:\Mont(\cC)\rightarrow\Monot(\cC)$, defined by precomposition.

Now we have our two major comparison results:
\begin{thm}
 \label{lawvere-lurie-mons}
The natural functor $\varphi:\Mont(\cC)\rightarrow\Monot(\cC)$, as defined above, is an equivalence.
\end{thm}

\begin{thm}
 \label{lawvere-lurie-algs}
Let $q:\cCt\rightarrow\Span$ be a Lawvere monoidal category, and $\cCot\rightarrow\Fins$ be the underlying Lurie monoidal category.

Then the natural functor $\theta:\Algt(\cC)\rightarrow\Algot(\cC)$ between the corresponding quasicategories of algebras, as defined above, is an equivalence.
\end{thm}

Both will be proved in the next section; here is the most important corollary:
\begin{thm}
\label{lawvere-lurie-cats}
 The quasicategory of Lawvere symmetric monoidal categories is equivalent to the quasicategory of Lurie symmetric monoidal categories.
\end{thm}

\begin{proof}
 The quasicategory of cocartesian fibrations over $\Fins$ is equivalent to the quasicategory of functors $\Fins\rightarrow\Cinfty$; and the disjoint union property for cocartesian fibrations is equivalent to taking disjoint unions to products.

So the category of Lurie symmetric monoidal categories is equivalent to $\Monot(\Cinfty)$.

Similarly, the quasicategory of cocartesian fibrations over $\Span$ is equivalent to the category of functors $\Span\rightarrow\Cinfty$; and the product property is equivalent to being product-preserving.

So the category of Lawvere symmetric monoidal categories is equivalent to $\Mont(\Cinfty)$.

Theorem \ref{lawvere-lurie-mons} gives the required equivalence to prove the theorem.
\end{proof}

\subsection{The proofs of the comparison theorems}

This section merely contains the proofs of the two basic comparison results stated in the last section. The two arguments are very similar.

Both employ an auxiliary quasicategory $\cJ$, defined as follows:
$$\cJ_n=\coprod_{\substack{i+j+1=n,\\i,j\geq -1}}\left\{\text{$f:\Delta^j\rightarrow\Fins$, $g:\Delta^i\star\Delta^j\rightarrow\Span$, with $g|_{\Delta^j}=L\circ f$.}\right\},$$
where the face and degeneracy maps are obvious.

We can draw cells of $\cJ$ as span diagrams equipped with a fenced-off sub-span diagram on the right-hand side, where the fenced-off part is in the essential image of $L$:
\begin{displaymath}
\xymatrix@!R=6mm@!C=6mm
{&&&&X_{04}\pb{270}\ar[dr]\ar[dl]&&&&\\
&&&X_{03}\pb{270}\ar[dr]\ar[dl]&&X_{14}\pb{270}\ar[dr]\ar[dl]&\dl[dddlll]&&\\
&&X_{02}\pb{270}\ar[dr]\ar[dl]&&X_{13}\pb{270}\ar[dr]\ar[dl]&&X_{24}\pb{270}\ar[dr]\iar[dl]&&\\
&X_{01}\ar[dr]\ar[dl]&&X_{12}\ar[dr]\ar[dl]&&X_{23}\ar[dr]\iar[dl]&&X_{34}\ar[dr]\iar[dl]&\\
X_0&&X_1&&X_2&&X_3&&X_4}
\end{displaymath}

When $n=0$, we have $\Ob\cJ=\Ob(\Span)\sqcup\Ob(\Fins)$ (the summands are the contributions of $i=0, j=-1$ and $i=-1, j=0$ respectively).

The full subcategory $\cJ_\Span$ on the objects corresponding to $\Ob(\Span)$ consists of the contributions by $i=n, j=-1$, and this is a copy of $\Span$; the full subcategory $\cJ_\Fins$ on the objects corresponding to $\Ob(\Fins)$ consists of the contributions by $i=-1, j=n$, and this is a copy of $\Fins$.

Note that there are no 1-cells from $\cJ_\Fins$ to $\cJ_\Span$ in $\cJ$. Note also that the full subcategory embedding $\Span\rightarrow\cJ$ admits a retraction $\cL:\cJ\rightarrow\Span$ (defined by $L$ on $\cJ_\Fins$).

Equipped with this, we can prove the results:

\begin{proof}[Proof of Theorem \ref{lawvere-lurie-mons}]
This proof is inspired by \cite{DAG-II}*{Prop. 1.7.7}.

We write $\bMon(\cC)$ for the full subcategory of functors $\Map(\cJ,\cC)$ on the objects $f$ such that:
\begin{enumerate}[(i)]
 \item\label{m-equivs} For any set $A$, the canonical 1-cell of $\cJ$ defined by the constant maps $A_+:\Delta^0\rightarrow\Fins$ and $A:\Delta^0\star\Delta^0\rightarrow\Span$ is sent by $f$ to an equivalence in $\cC$,
 \item\label{m-lawvere} The restriction $f|_{\cJ_\Span}$ is a Lawvere monoid object.
 \item\label{m-lurie} The restriction $f|_{\cJ_\Fins}$ is a Lurie monoid object
\end{enumerate}

We notice that, in the presence of condition (\ref{m-equivs}), conditions (\ref{m-lawvere}) and (\ref{m-lurie}) are equivalent to one another. Indeed, both the sets of collapsing morphisms and the product properties match up under the given equivalences.

Also, condition (\ref{m-equivs}) is equivalent to saying that $f$ is a left Kan extension of $f|_{\Span}$ along $\Span\rightarrow\cJ$ (as defined in \cite{HTT}*{Section 4.3.2}).

Indeed, to say that the following diagram is a Kan extension diagram
\begin{displaymath}
 \xymatrix{\cJ_\Span\ar[r]^{f|_{\cJ_\Span}}\ar[d]&\cC\\
           \cJ\ar[ur]^f&}
\end{displaymath}
is to say that, for every object $X\in\Ob(\cJ)$, the diagram
\begin{displaymath}
 \xymatrix{(\cJ_\Span)_{/X}\ar[r]^{f}\ar[d]&\cC\\
           (\cJ_\Span)_{/X}\star 1\ar[ur]^f&}
\end{displaymath}
makes $f(X)$ a colimit of $f_{/X}$. Here $(\cJ_\Span)_{/X}$ is notation for $(\cJ_{/X})\timeso{\cJ}(\cJ_{\Span})$.

If $X\in\Ob(\cJ_\Span)$, then $(\cJ_\Span)_{/X}\isom\Span_{/X}$, and the diagram is vacuously a colimit.

If $X_+\in\Ob(\cJ_\Fins)$, then $(\cJ_\Span)_{/X}\isom\Span_{/X}$, and the diagram is a colimit if and only if the 1-cell from $A$ to $A_+$ is taken to an equivalence in $\cCt$.

We can show that every map $f_0:\Span\rightarrow\cC$ admits $f_0\circ\cL$ as a left Kan extension to a map $\cJ\rightarrow\cC$. Indeed, following the definition of a Kan extension along an inclusion, this amounts to showing for $X\in(\Fins)_0$ that the diagram
\begin{displaymath}
\xymatrix{\Span_{/X}\ar[d]\ar[rr]&&\Span\ar[r]^{f_0}&\cC\\
1\star\Span_{/X}\ar[r]&\cJ\ar[r]_{\cL}&\Span\ar[ur]_{f_0}&}
\end{displaymath}
is a colimit diagram of shape $\Span_{/X}$ in $\cC$. However, $\Span_{/X}$ has a terminal object given by the diagram of identities:
\begin{displaymath}
\xymatrix{&X\ar[dl]\ar[dr]&\dl[dl]\\
X&&X}
\end{displaymath}
Thus the colimit is given by $f_0(X)$, with colimiting structure maps described by $\cL$.

Hence we can use \cite{HTT}*{Prop 4.3.2.15} to deduce that the restriction functor $p:\bMon(\cC)\rightarrow\Mont(\cC)$ is acyclic Kan. 

Now, composition with $\cL$ defines a section of the functor $\bMon(\cC)\rightarrow\Mont(\cC)$, and $\theta$ is the composition of this with the restriction map $p':\bMon(\cC)\rightarrow\Monot(\cC)$.

Thus all we need to do is show that $p'$ is acyclic Kan, and \cite{HTT}*{Prop 4.3.2.15} says that this will follow from these two claims:
\begin{enumerate}[(a)]
\item\label{mak-a} Every $f_0\in\Monot(\cC)$ admits a right Kan extension $f$, as shown fitting into the following diagram:
\begin{displaymath}
  \xymatrix{\Fins\ar[r]^{f_0}\ar[d]&\cC\\
            \cJ\ar[ur]_f&}
\end{displaymath}
\item\label{mak-b} Given $f\in\sSet_\Span(\cJ,\cC)$ such that $f_0=f|_{\Fins}$ is a Lurie monoid object, $f$ is a right Kan extension of $f_0$ if and only if $f$ satisfies condition (\ref{m-equivs}) above.
\end{enumerate}

To prove (\ref{mak-a}), for any object $K\in\Span$, we consider the quasicategory $$\cK_K=\Fins\timeso{\Span}(\Span_{K/}).$$
We write $g_K$ for the composite $\cK_K\rightarrow\Fins\rightarrow\cC$.

According to \cite{HTT}*{Lemma 4.3.2.13}, it will suffice to to show that, for every $K$, $g_K$ has a colimit in $\cC$.

Since there is an injection $\Fins\rightarrow\Span$, there is an injection $\cK_K\rightarrow\Span_{K/}$; we thus write the objects of $\cK_K$ as morphisms $K\stackrel{a}{\leftarrow}Y\stackrel{b}{\rightarrow}Z$ of $\Span$. We let $\cK'_K$ denote the full subcategory on the objects where $b$ is an isomorphism and $a$ an injection.

The inclusion $\cK'_K\rightarrow\cK_K$ has a right adjoint. Indeed, one choice of right adjoint sends the object $K\stackrel{a}{\leftarrow}Y\stackrel{b}{\rightarrow}Z$ to $K\leftarrow \im(a)\rightarrow \im(a)$. Regarding an adjunction as a bicartesian fibration over $\Delta^1$, we need to provide cartesian lifts of the nontrivial 1-cell in $\Delta^1$; a lift for $K\stackrel{a}{\leftarrow} Y\stackrel{b}{\rightarrow}Z$ is given by
\begin{displaymath}
\xymatrix@!R=8mm@!C=8mm{
&&Y\ar[dl]\ar[dr]^=\pb{270}&&\\
&\im(a)\iar[dl]\ar[dr]^=&&Y\ar[dl]\ar[dr]&\\
K&&\im(a)&&Z;}
\end{displaymath}
it is readily checked that this is indeed cartesian.

Hence $(\cK'_K)^\op\rightarrow(\cK_K)^\op$ is cofinal, so we just need to show that $g'_K=g_K|_{\cK'_K}$ has a colimit in $\cC$.

We write $\cK''_K$ for the full subcategory of $\cK'_K$ on the objects $K\leftarrow\{k\}\rightarrow\{k\}$. Because of the product property of monoidal objects, $g''_K=g'_K|_{\cK''_K}$ is a Kan extension of $g'_K$.

Thus, using \cite{HTT}*{Lemma 4.3.2.7}, we merely need to show that $g''_K$ has a limit in $\cC$. But by the product property, $f$ exhibits $f(K)$ as a limit of $g''_K$, thus proving~(\ref{mak-a}).

The argument is reversible: it shows that $f$ is a right Kan extension of $f_0$ at $K_+$ if and only if $f$ induces an equivalence $f(K_+)\rightarrow f(K)$; this proves~(\ref{mak-b}).
\end{proof}

\begin{proof}[Proof of Theorem \ref{lawvere-lurie-algs}]
This proof is inspired by \cite{DAG-II}*{Prop. 1.7.15}.

We write $\bAlg(\cC)$ for the full subcategory of functors $\Map_\Span(\cJ,\cCt)$ on the objects $f$ such that $qf=\cL$, and:
\begin{enumerate}[(i)]
 \item\label{d-equivs} For any set $A$, the canonical 1-cell of $\cJ$ defined by the constant maps $A_+:\Delta^0\rightarrow\Fins$ and $A:\Delta^0\star\Delta^0\rightarrow\Span$ is taken to an equivalence in $\cCt$,
 \item\label{d-lawvere} The restriction $f|_{\cJ_\Span}$ is a Lawvere algebra object.
 \item\label{d-lurie} The restriction $f|_{\cJ_\Fins}$ is a Lurie algebra object, in the sense that the following diagram factors with a dotted arrow, as shown, to give one:
\begin{displaymath}
 \xymatrix{\cJ_\Fins\ar[d]\ar[r]^\sim&\Fins\dar[r]&\cCot\ar[d]\\
           \cJ\ar[rr]&&\cCt.}
\end{displaymath}
\end{enumerate}

We notice that, in the presence of condition (\ref{d-equivs}), conditions (\ref{d-lawvere}) and (\ref{d-lurie}) are equivalent to one another. Indeed, both the sets of collapsing morphisms and the product properties match up under the given equivalences.

Also, condition (\ref{d-equivs}) is equivalent to saying that $f$ is a $q$-Kan extension of $f|_{\Span}$ along $\Span\rightarrow\cJ$ (as defined in \cite{HTT}*{Section 4.3.2}).

Indeed, to say that the following diagram is a Kan extension diagram
\begin{displaymath}
 \xymatrix{\cJ_\Span\ar[r]^{f|_{\cJ_\Span}}\ar[d]&\cCt\ar[d]^q\\
           \cJ\ar[r]\ar[ur]^f&\Span}
\end{displaymath}
is to say that, for every object $X\in\Ob(\cJ)$, the diagram
\begin{displaymath}
 \xymatrix{(\cJ_\Span)_{/X}\ar[r]^{f}\ar[d]&\cCt\ar[d]^q\\
           (\cJ_\Span)_{/X}\star 1\ar[r]\ar[ur]^f&\Span}
\end{displaymath}
makes $f(X)$ a $q$-colimit of $f_{/X}$. Here $(\cJ_\Span)_{/X}$ is notation for $(\cJ_{/X})\timeso{\cJ}(\cJ_{\Span})$.

If $X\in\Ob(\cJ_\Span)$, then $(\cJ_\Span)_{/X}\isom\Span_{/X}$, and the diagram is vacuously a $q$-colimit.

If $X_+\in\Ob(\cJ_\Fins)$, then $(\cJ_\Span)_{/X}\isom\Span_{/X}$, and the diagram is a $q$-colimit if and only if the 1-cell from $A$ to $A_+$ is taken to an equivalence in $\cCt$.

Since every map $f_0:\Span\rightarrow\cCt$ has $f_0\circ\cL$ as a $q$-left Kan extension to a map $\cJ\rightarrow\cCt$, \cite{HTT}*{Prop 4.3.2.15} says that the restriction functor $p:\bAlg(\cC)\rightarrow\Algt(\cC)$ is acyclic Kan.

Now, composition with $\cL$ defines a section of the functor $\bAlg(\cC)\rightarrow\Algt(\cC)$, and $\theta$ is the composition of this with the restriction map $p':\bAlg(\cC)\rightarrow\Algot(\cC)$.

Thus all we need to do is show that $p'$ is acyclic Kan, and \cite{HTT}*{Prop 4.3.2.15} says that this will follow from these two claims:
\begin{enumerate}[(a)]
\item\label{ak-a} Every $f_0\in\Algot(\cC)$ admits a $q$-Kan extension $f$, as shown fitting into the following diagram:
\begin{displaymath}
  \xymatrix{\Fins\ar[r]^{f_0}\ar[d]&\cCot\ar[r]&\cCt\ar[d]\\
            \cJ\ar[rr]\ar[urr]_f&&\Span}
\end{displaymath}
\item\label{ak-b} Given $f\in\sSet_\Span(\cJ,\cCt)$ such that $f_0=f|_{\Fins}$ is a Lurie algebra object, $f$ is a $q$-right Kan extension of $f_0$ if and only if $f$ satisfies condition (\ref{d-equivs}) above.
\end{enumerate}

To prove (\ref{ak-a}), for any object $K\in\Span$, we consider the quasicategory $$\cK_K=\Fins\timeso{\Span}(\Span_{K/}).$$
We write $g_K$ for the composite $\cK_K\rightarrow\Fins\rightarrow\cCot$.

According to \cite{HTT}*{Lemma 4.3.2.13}, it will suffice to to show that, for every $K$, $g_K$ has a $q$-colimit in $\cCot$.

Since there is an injection $\Fins\rightarrow\Span$, there is an injection $\cK_K\rightarrow\Span_{K/}$; we thus write the objects of $\cK_K$ as morphisms $K\stackrel{a}{\leftarrow}Y\stackrel{b}{\rightarrow}Z$ of $\Span$. We let $\cK'_K$ denote the full subcategory on the objects where $b$ is an isomorphism and $a$ an injection.

As in the preceding proof of Theorem \ref{lawvere-lurie-mons}, we just need to show that $g'_K=g_K|_{\cK'_K}$ has a $q$-colimit in $\cCot$.

We write $\cK''_K$ for the full subcategory of $\cK'_K$ on the objects $K\leftarrow\{k\}\rightarrow\{k\}$. Because of the product property of monoidal objects, $g''_K=g'_K|_{\cK''_K}$ is a $q$-Kan extension of $g'_K$.

Thus, using \cite{HTT}*{Lemma 4.3.2.7}, we merely need to show that $g''_K$ has a $q$-limit in $\cCt$. But $f$ exhibits $f(K)$ as a $q$-limit of $g''_K$, and that proves~(\ref{ak-a}).

Similarly, this argument is reversible: $f$ is a $q$-right Kan extension of $f_0$ at $K_+$ if and only if $f$ induces an equivalence $f(K_+)\rightarrow f(K)$; this proves~(\ref{ak-b}).
\end{proof}

\section{Distributive laws}
\label{distributive-laws}

We wish to study the notion of distributivity: our chief motivation is to be able to describe rings as consisting of a multiplicative monoid, and an additive group, such that the multiplicative structure distributes over the additive.

\subsection{Motivation}

Without even having constructed the classical theory of semirings, we can write down some morphisms in it: they're merely natural ways of forming some objects in a semiring given others. For example:
$$(a,b,c)\longmapsto(a+b,abc,1,a(b+c)).$$
However, one axiom we impose on semirings is the distributive law. In the above morphism, that means we can replace $a(b+c)$ by $ab+ac$.

In general, the distributive law means that whenever we add and then multiply, we may instead multiply and then add. This says, in particular, that elements of a free semiring are sums of products of generators.

It also means that we can factor our operation as follows:
$$(a,b,c)\longmapsto(a,b,a,b,c,a,b,a,c)\longmapsto(a,b,abc,1,ab,ac)\longmapsto(a+b,abc,1,ab+ac).$$
Here the first map is a diagonal map, making copies of the variables; the second map multiplies the variables together to form monomials; the third map adds the monomials together to form our chosen semiring elements.

We can generalise this, producing natural operations on semirings systematically:
\begin{itemize}
\item For any map of sets $f:W\leftarrow X$ we have a pullback map, ``making copies'', $\Delta_f:R^W\rightarrow R^X$, defined by $(\Delta_fA)_x = A_{f(x)}$.
\item For any map of sets $g:X\rightarrow Y$ we have a multiplication map $\Pi_g:R^X\rightarrow R^Y$, defined by $(\Pi_gA)_y = \prod_{g(x)=y} A_x$.
\item For any map of sets $h:Y\rightarrow Z$ we have an addition map $\Sigma_h:R^Y\rightarrow R^Z$, defined by $(\Sigma_hA)_z = \sum_{h(y)=z} A_y$.
\end{itemize}
Naturally, these can be composed. Given a diagram of maps
$$W\stackrel{f}{\longleftarrow}X\stackrel{g}{\longrightarrow}Y\stackrel{h}{\longrightarrow}Z,$$
we can produce a composite $\Sigma_h\circ\Pi_g\circ\Delta_f$. Moreover, from our understanding of the distributive law as saying that all elements of free semirings are sums of monomials, this should constitute all natural operations in the theory of semirings.

Since they are putatively the morphisms of the theory of semirings, we should be able to compose them. The composition is itself an instance of the distributive law: we are taking a sum of products of sums of products, and need to rewrite it as a sum of products. To do this we must interchange the ``products of sums'' in our formula, to get a sum of sums of products of products, which is naturally a sum of products. This is not difficult, but there is no very pleasant expression for it. (We do give a precise account, using category theory, valid in greater generality later).

Regarding the morphisms as being composites of their three constituent parts, we wonder what we can interchange. Moving a pullback to the left of everything else is easy: rather than doing algebra and then making copies of the results, we can make copies of the starting ingredients and then do the same algebraic routine on each copy.

The issue is with sums and products. In fact, they don't interchange quite perfectly. As can be seen from the distributive law $a(b+c) = ab + ac$, in order to do products then sums, we may need to make some extra copies first: even though the left-hand-side involves no copying, the right-hand-side requires us to make a copy of the variable $a$ so we can multiply it into both $b$ and $c$.

The upshot is that we can consider two subcategories of the theory of semirings: the subcategory $\cM$ whose morphisms copy and multiply, corresponding to diagrams
$$W\longleftarrow X\longrightarrow Y\stackrel{=}{\longrightarrow} Y$$
and the subcategory $\cA$ whose elements consist of additions, corresponding to diagrams
$$Y\stackrel{=}{\longleftarrow}Y\stackrel{=}{\longrightarrow}Y\longrightarrow Z.$$
The former category is a copy of the theory of commutative monoids, the homotopy category of $\Span$, and the latter category is just a copy of the category of finite sets.

Moreover, they form a unique factorisation system, insofar as any morphism in the theory of semirings can be written uniquely as a morphism from $\cM$ followed by a morphism from $\cA$.

For semirings, the comments above can be fleshed out to provide a complete and functioning approach \cite{Rosebrugh-Wood}. But for $E_\infty$-semirings, where we demand all the axioms only up to coherent homotopies, there is much more to say.

We still expect subcategories $\cA$ and $\cM$, and for $a\in\cA$ and $m\in\cM$ still expect a composite $m\circ a$ (where we add then multiply) to be homotopic to a composite $a\circ m'$ (where we multiply then add). By replacing the centre $m_2\circ a_1$ in this way, this means we have composites up to homotopy for chains of two composable morphisms $a_2\circ m_2\circ a_1\circ m_1$.

But we have coherence issues when trying to compose three composable morphisms. Suppose, indeed, we have $a_3\circ m_3\circ a_2\circ m_2\circ a_1\circ m_1$. We could replace either $m_3\circ a_2$ or $m_2\circ a_1$, and then simplify, and then do the other. This gives us potentially two different composites, and naturally we want to ensure that there is no difference up to homotopy. Then problems with composing longer chains give coherence requirements on the homotopies.

As in the case of $A_\infty$ structures, one natural approach is to find a family of geometric shapes that contains all possible ways of performing these replacements: something akin to the \emph{associahedra} or \emph{Stasheff polytopes} (\cite{Stash}).

Whereas the associahedra are all just polytopes, we cannot just use polytopes in this theory: we have two different subcategories, $\cM$ and $\cA$, and must record which part of our shape represents which.

The theory of \emph{double categories} \cite{Ehresmann} is a primordial example of a type of category theory which distinguishes two sorts of morphisms. However, it does not have the spatial nature required: there are no higher cells to store higher coherence properties. We are led to consider bisimplicial sets: these occur as the nerves of double categories, and do have appropriate higher cells.

\subsection{Distributahedra}

Our first goal is to define a family of bisimplicial sets which we shall call \emph{distributahedra}. The $n$-th distributahedron $\Xi^n$ is intended to resemble an $n$-tuple of composable factorised morphisms in a quasicategorical theory with a distributive law, together with choices of all compositions, and cells giving coherence properties.

We shall use $(m,0)$-simplexes to represent $m$-simplices of the additive sub-quasicategory $\cA$, and $(0,n)$-simplices to represent $n$-simplices of the multiplicative sub-quasicategory $\cM$. An $(1,1)$-simplex represents a diagram
\begin{displaymath}
\xymatrix{\bullet\ar[r]^{\cA}\ar[d]_{\cM}&\bullet\ar[d]^{\cM}\\
          \bullet\ar[r]_{\cA}&\bullet}
\end{displaymath}
consisting of a composable pair of addition and multiplication (along the top and right) and a factorisation of their composite (along the left and bottom), and a homotopy between the composites of the two sides. In general, a $(m,n)$-simplex represents an $m$-simplex from $\cA$ and an $n$-simplex from $\cM$, together with a coherent collection of refactorisations.

Accordingly, we define $\Xi^0$ just to be a point (a $(0,0)$-simplex), and $\Xi^1$ to be a pair consisting of a $(0,1)$-simplex and a $(1,0)$-simplex, the latter starting where the former finishes.

The bisimplicial set $\Xi^2$ describes our preferred approach to composing two factorised morphisms: we refactorise the inner pair to get two multiplications and then two additions. Then we compose the two multiplications and the two additions. The following diagram depicts this, with the original pair of morphisms forming the diagonal faces, and the resulting composite along the bottom:
\begin{displaymath}
\xymatrix{&&\bullet\ar[dr]^\cM&&\\
&\bullet\ar[ur]^\cA\ar[dr]^\cM&&\bullet\ar[dr]^\cA&&\\
\bullet\ar[ur]^\cM\ar[rr]_\cM&&\bullet\ar[ur]^\cA\ar[rr]_\cA&&\bullet}
\end{displaymath}
The square is a $(1,1)$-simplex; the left cell is a $(0,2)$-simplex and the right cell a $(2,0)$-simplex.

In general, for $I$ a partially ordered set, we can define the $I$-th distributahedron $\Xi^I$ as a bisimplicial set. Given partially ordered sets $J_1$ and $J_2$, we define the $(J_1,J_2)$-cells of $\Xi^I$ to be given by
$$(\Xi^I)_{J_1,J_2}=\left\{\text{ordered maps $f:J_1\sqcup J_2\rightarrow I$}\right\},$$
where $J_1\sqcup J_2$ is defined to be the concatenation of $J_1$ and $J_2$. We will write $(i,j)$-cells of $\Xi^n$ as sequences of elements $a_0\cdots a_i|b_0\cdots b_j$ in $I$.

\begin{figure}[htbp]
\label{distributivity-pic}
\begin{center}
\includegraphics{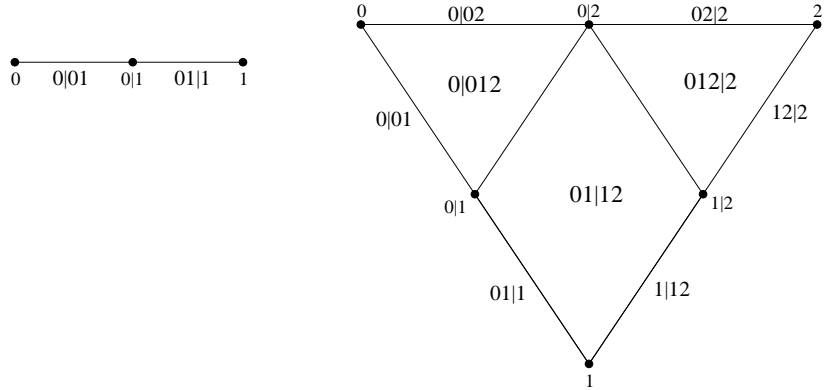}
\end{center}
\caption{The bisimplicial sets $\Xi^1$ and $\Xi^2$, with simplices labelled.}
\end{figure}

So, in general, the distributahedron $\Xi^n=\Xi^{0,\ldots,n}$ contains alternating $(0,1)$-simplices and $(1,0)$-simplices
$$(0|0)\stackrel{0|01}{\longrightarrow}(0|1)\stackrel{01|1}{\longrightarrow}(1|1)\stackrel{1|12}{\longrightarrow}\cdots \stackrel{(n-1)n|n}{\longrightarrow}(n|n).$$

As in the examples of figure \ref{distributivity-pic}, the diagonal simplicial set of $\Xi^n$ is a cell decomposition of the $n$-simplex $\Delta^n$. Geometrically, the space occupied by the simplex $a_0\cdots a_i|b_0\cdots b_j$ is the set of midpoints between points in the subsimplex $a_0\cdots a_i$ and the subsimplex $b_0\cdots b_j$.

Also, since by this description the $(J_1,J_2)$-simplices are covariant in the partially ordered set $I$, the bisimplicial sets $\Xi^I$ assemble to form a cosimplicial object in the category of bisimplicial sets. We shall need this later.

\subsection{Using distributahedra}

Now, by the philosophy above, given a theory in which there is a distributive law, we should be able to produce a bisimplicial set. This bisimplicial set should have $(m,0)$-simplices corresponding to $m$-simplices of the additive subcategory, and $(0,n)$-simplices corresponding to $n$-simplices of the multiplicative subcategory, and mixed simplices corresponding to commutative diagrams of mixed type.

Accordingly, we define:
\begin{defn}
A \emph{distributive law} is a bisimplicial set $D_{*,*}$ which has liftings with respect to all maps of the following forms:
\begin{itemize}
\item $\Lambda^{m,n}_{0,n}\rightarrow\Delta^{m,n}$ for all $m,n\geq 1$,
\item $\Lambda^{m,n}_{0,k}\rightarrow\Delta^{m,n}$ for all $n>k>0$ and $m\geq 0$, and 
\item $\Lambda^{m,n}_{k,n}\rightarrow\Delta^{m,n}$ for all $n\geq 0$ and $m>k>0$.
\end{itemize}

Given simplicial sets $X$ and $Y$, we write $X\boxtimes Y$ for the \emph{external product} bisimplicial set defined by $(X\boxtimes Y)_{m,n}=X_m\times Y_n$. Using that, the bisimplicial set $\Lambda^{a,c}_{b,d}$ is defined by 
$$\Lambda^{a,c}_{b,d}=(\Lambda^a_b\boxtimes\Delta^c)\mathop{\cup}_{(\Lambda^a_b\boxtimes\Lambda^c_d)}(\Delta^a\boxtimes\Lambda^c_d),$$
with the left and right simplicial directions represented by the left and right factors, and $\Delta^{m,n}$ is shorthand for $\Delta^m\boxtimes\Delta^n$.

We need to employ the reasonable convention that $\Lambda^0_0=\emptyset$.

We call this list of maps the \emph{inner bihorn inclusions}.
\end{defn}

Note that the case of lifting for $\Lambda^{m,0}_{k,0}=\Lambda^m_k\boxtimes *$ implies that the simplicial set $D_{*,0}$ is a quasicategory, and the lifting for $\Lambda^{0,n}_{0,k}=*\boxtimes\Lambda^n_k$ implies that the simplicial set $D_{0,*}$ is a quasicategory. We refer to $D_{*,*}$ as being a distributive law \emph{between} the quasicategories $D_{*,0}$ and $D_{0,*}$.

The lifting for $\Lambda^{1,1}_{0,1}\rightarrow\Delta^{1,1}$ allows us to take any pair of composable morphisms, with the first in $T$ and the second in $S$, and give a pair with homotopic composite with the first in $S$ and the second in $T$. This is the lowest-degree manifestation of distributivity, as we would recognise it.

Since the bisimplicial sets $\Xi^n$ have been defined so as to be the pictures of composable factorised morphisms, we can get a quasicategory from $D$ by considering maps from $\Xi^n$.

Thus, given $D$ a distributive law between the categories $S$ and $T$, we define the \emph{composite quasicategory} $S\circ_DT$ as a simplicial set by the formula $(S\circ_DT)_n=\bisSet(\Xi^n,D)$. This uses the cosimplicial structure on the collection $\{\Xi^n\}$, mentioned above: the cosimplicial structure for $\{\Xi^n\}$ becomes the simplicial structure for $S\circ_DT$.

This construction is known already as the Artin-Mazur codiagonal functor \cite{Artin-Mazur}; see \cite{Cordier-Porter} for a discussion of the pedigree of the idea in homotopy theory.

An easy consequence of this highbrow way of thinking is the following:
\begin{prop}
\label{maps-from-SDT}

Give a map of simplicial sets $S\circ_DT\rightarrow X$ is exactly the same as giving a bisimplicial set map $D\rightarrow\sSet(\Delta^\bullet\times\Delta^\bullet,X)$: that is, a collection of compatible maps $D_{m,n}\rightarrow\sSet(\Delta^m\times\Delta^n,X)$.

\end{prop}
\begin{proof}
This is exactly the universal property of the Artin-Mazur codiagonal.
\end{proof}

Another basic consequence we shall need is the following:
\begin{prop}
\label{comp-is-quasicat}
Let $S$ and $T$ be quasicategories, and let $D$ be a distributive law relating them. Then $S\circ_DT$ is a quasicategory.
\end{prop}
\begin{proof}
Let us suppose given an inner horn $\Lambda^n_k\rightarrow S\circ_D T$. We need to extend this to a simplex $\Delta^n\rightarrow S\circ_D T$, or equivalently a map $\Xi^n\rightarrow D$.

In terms of the latter, we have maps from all cells $a_0\cdots a_i|b_0\cdots b_j$ where $a_0,\ldots,a_i,b_0,\ldots b_j$ take fewer than $n$ of the values $0,\ldots,n$, and all those which take exactly $n$ of them, as long as they take the value $k$.

The nondegenerate missing cells are spanned by the objects $0\cdots a|a\cdots n$ for $0\leq a\leq n$.
First we extend it to $0\cdots k|k\cdots n$, which is an extension with respect to $\Lambda^{n-k,k}_{0,k}\rightarrow\Delta^{n-k,k}$

Then for increasing $i$, we extend to $0\cdots (k-i)|(k-i)\cdots n$ and $0\cdots (k+i)|(k+i)\cdots n$, using extensions along $\Lambda^{n-k+i,k-i}_{i,k-i} \rightarrow \Delta^{n-k+i,k-i}$ and $\Lambda^{n-k-i,k+i}_{0,k} \rightarrow \Delta^{n-k-i,k+i}$ respectively.
\end{proof}

We take the category of distributive laws to be the full subcategory of $\bisSet$ on the distributive laws; it is evident that the composite quasicategory construction is functorial in the distributive law $D$.

Given a quasicategory $S$, we can form a distributive law $\pi_1^*S$ by $(\pi_1^*S)_{m,n}=S_m$. This is indeed a distributive law, since a map $(\Lambda^a_b\boxtimes\Delta^c)\cup(\Delta^a\boxtimes\Lambda^c_d)\rightarrow\pi_1^*S$ gives a map $\Delta^a\rightarrow S$, which gives a map $\Delta^a\boxtimes\Delta^c\rightarrow\pi_1^S$.

This distributive law has $(\pi_1^*S)_{*,0}=S$ and $(\pi_1^*S)_{0,*}=S_0$.

Also, a map $\Xi^n\rightarrow \pi_1^*S$ corresponds to a map $\Xi^n_{*,0}\rightarrow S$. But $\Xi^n_{*,0}=\Delta^n_*$. Hence $S\circ_{\pi_1^*S}S_0=S$.

For any distributive law $D$ from $S$ to $T$, there is a map $\pi_1^*S\rightarrow D$, coming from the fact that $S_*=D_{*,0}$. This gives a map $S\rightarrow S\circ_DT$.

Similarly, for any quasicategory $T$, there is a distributive law $\pi_2^*T$, and functoriality of maps from this gives a map $T\rightarrow S\circ_DT$.

\subsection{Distributive laws and $\Span$}

In this subsection we give an example, which demonstrates that, for any category with pullbacks, $\Span(\cC)$ itself arises from a distributive law. Indeed, morphisms in $\Span$ consist of morphisms of $\Finop$ composed with morphisms of $\Fin$, and the higher structure merely encodes how to interchange them.

We define a bisimplicial set $D(\cC)$ as follows:
$$D(\cC)_{m,n}=\left\{\vcenter{\xymatrix{
X_{00}\ar[d]&X_{10}\ar[l]\ar[d]\pb{225}&\cdots\ar[l]&X_{m0}\pb{225}\ar[l]\ar[d]\\
X_{01}\ar[d]&X_{11}\ar[l]\ar[d]\pb{225}&\cdots\ar[l]&X_{m1}\pb{225}\ar[l]\ar[d]\\
\vdots\ar[d]&\vdots\ar[d]&                          &\vdots              \ar[d]\\
X_{0n}      &X_{1n}\ar[l]              &\cdots\ar[l]&X_{mn}        \ar[l]      \\
}}\right\}$$

This can be regarded as a sub-bisimplicial set of the bisimplicial set defined by $X_{m,n}=\sSet((\Delta^m)^\op\times\Delta^n,\cC)$, consisting of all the cells whose $(1,1)$-faces are pullbacks.

It follows from this description that
$$\bisSet(X\boxtimes Y,D(\cC))=\sSet(X^\op\times Y,\cC).$$

Evidently $D(\cC)_{*,0}\isom\cC^\op$, and $D(\cC)_{0,*}\isom\cC$.

\begin{prop}
\label{spans-bisset-is-dist}
The bisimplicial set $D(\cC)$ is a distributive law.
\end{prop}
\begin{proof}
 Since, by Proposition \ref{nerve-outer-horns}, a map $(\Lambda^m_k)^\op\times\Delta^n\rightarrow\cC$ extends uniquely to a map $(\Delta^m)^\op\times\Delta^n\rightarrow\cC$ if $m\geq 4$ or if $m>k>0$, we only need interest ourselves in the lifting problems where all horns are outer horns of dimension at most three.

 These are the lifting problems for $\Lambda^{m,n}_{0,n}\rightarrow\Delta^{m,n}$ with $1\leq m,n\leq 3$.

 A map $\Lambda^{1,1}_{0,1}\rightarrow D(\cC)$ gives us a diagram in $\cC$ as follows:
 \begin{displaymath}
   \xymatrix{X_{00}\ar[d]&\\X_{01}&X_{11}.\ar[l]}
 \end{displaymath}
 This can be extended to a map $\Delta^{1,1}\rightarrow D(\cC)$ by taking a pullback in $\cC$.
 
 A map $\Lambda^{2,1}_{0,1}\rightarrow D(\cC)$ can be extended as below to a map $\Delta^{2,1}\rightarrow D(\cC)$:
 \begin{displaymath}
   \xymatrix{X_{00}\ar[d]&X_{10}\pb{225}\ar[l]\ar[d]&X_{20}\pb{225}\dar[l]\ar@/_2ex/[ll]\ar[d]\\
             X_{01}      &X_{11}        \ar[l]      &X_{21}.        \ar[l]}
 \end{displaymath}
 The dotted arrow exists since $X_{10}$ is a pullback; the right-hand square is easily checked to be a pullback. The lifting problem for $\Lambda^{1,2}_{0,2}$ is a mirror image of this one.
 
 A map $\Lambda^{3,1}_{0,1}\rightarrow D(\cC)$ gives us a diagram in $\cC$ of shape
 \begin{displaymath}
  \xymatrix{X_{00}\ar[d]&X_{10}\pb{225}\ar[d]\ar[l]&X_{20}\pb{225}\ar[d]\ar[l]&X_{30}\pb{225}\ar[d]\ar[l]\ar@/_2ex/[ll]\\
            X_{01}&X_{11}\ar[l]&X_{21}\ar[l]&X_{31},\ar[l]}
 \end{displaymath}
 and the two morphisms $X_{30}\rightarrow X_{10}$ are equal since they are both projections of the same pullback. This gives us a diagram $\Delta^{3,1}\rightarrow D(\cC)$. The lifting problem for $\Lambda^{1,3}_{0,3}$ is a mirror image of this.

 A map $\Lambda^{2,2}_{0,2}\rightarrow D(\cC)$ gives us a diagram in $\cC$ of shape
 \begin{displaymath}
  \xymatrix{X_{00}\ar[d]\ar@/_3ex/[dd]&X_{10}\pb{225}\ar[d]\ar@/_3ex/[dd]\ar[l]&X_{20}\pb{225}\ar[d]\ar@/_3ex/[dd]\ar[l]\ar@/_3ex/[ll]\\
            X_{01}\ar[d]&X_{11}\pb{225}\ar[l]\ar[d]&X_{21}\pb{225}\ar[l]\ar@/_3ex/[ll]\ar[d]\\
            X_{02}&X_{12}\ar[l]&X_{22},\ar[l]\ar@/_3ex/[ll]}
 \end{displaymath}
 and all composites are equal, since they agree on the intersection $\Lambda^2_2\boxtimes\Lambda^0_2$.

 Maps from $\Lambda^{2,3}_{0,3}$, $\Lambda^{3,2}_{0,2}$ and $\Lambda^{3,3}_{0,3}$ similarly define all the required data coherently; this completes the checks.
\end{proof}

Now, we give a description of $\Span(\cC)$ as the composite of $\cC^\op$ and $\cC$ along $D(\cC)$. The philosophy is that $\Span(\cC)$ consists of morphisms from $\cC^\op$ and from $\cC$, which are interchanged by taking pullbacks.
\begin{prop}
 The quasicategory $\cC^\op\circ_{D(\cC)}\cC$ is equivalent to $\Span(\cC)$.
\end{prop}
\begin{proof}
 We need to show that a map $F:\Xi^n\rightarrow D(\cC)$ is an $n$-span diagram in $\cC$.

 The $(0,0)$-cells of $\Xi^n$ consist of $i|j$ for all $0\leq i\leq j\leq n$. Their images $F_{i,j}\in\Ob\cC$ are  naturally the objects of an $n$-span diagram.

 The $(1,0)$-cells give morphisms $F_{i,j}\rightarrow F_{i',j}$ for $0\leq i\leq i'\leq j\leq n$, and the $(0,1)$-cells give morphisms $F_{i,j}\rightarrow F_{i,j'}$ for $0\leq i\leq j'\leq j\leq n$: these are exactly the morphisms of an $n$-span diagram.

 The higher cells give the pullback property.
\end{proof}

This makes precise our intuition that product-preserving functors from $\Span$ consist of objects with diagonals and products, and a suitable compatibility between them.

We can derive an interesting collection of corollaries from this.

It is a familiar fact that a map of pointed finite sets $X_+\rightarrow Y_+$ decomposes uniquely as a composite $X_+\rightarrow X'_+\rightarrow Y_+$, where the first map is indexed by an inclusion $\xymatrix{X'\iar[r]&X}$ (where everything not in the image is sent to the basepoint), and the second map is indexed by a map of finite sets $X'\rightarrow Y$.

Motivated by this, we define a bisimplicial set $\hat D(\Fin)$ to be the sub-bisimplicial set of $D(\Fin)$ where all the 1-cells of $\Finop$ are monomorphisms, and deduce:
\begin{prop}
 The bisimplicial set $\hat D(\Fin)$ is a distributive law.
\end{prop}
\begin{proof}
 It is rapidly checked that all the lifting properties in the proof of Proposition \ref{spans-bisset-is-dist} give monomorphisms in the $\Finop$ direction if they are input them. In particular, the pullback of a monomorphism by a map of finite sets is another monomorphism.
\end{proof}

This describes $\Finsop$ as coming from a distributive law. In addition, we can realise the map $R:\Finsop\rightarrow\Span$ as being induced by the natural inclusion of distributive laws $\hat D(\Fin)\rightarrow D(\Fin)$, and the map $\Finop\rightarrow\Finsop$ arises from the inclusion of the left-hand factor of the distributive law.

Now we are able to prove the proposition first stated as Proposition \ref{finsop-is-universal}, that $\Finsop$ is the initial pointed theory.
\label{proof-finsop-is-universal}
\begin{proof}[Proof of Proposition \ref{finsop-is-universal}]
We show first that any pointed theory $T$ admits a morphism of theories from $\Finsop$.

Using the philosophy of Proposition \ref{maps-from-SDT}, to define a map $\Finsop\rightarrow T$ requires us merely to define compatible maps $\hat D(\Fin)_{m,n}\rightarrow\sSet(\Delta^m\times\Delta^n,T)$ for all $m$ and $n$.

The square
 \begin{displaymath}
  \xymatrix{X_{00}\iar[d]_i&X_{10}\pb{225}\iar[d]^j\ar[l]_g\\
            X_{01}&X_{11}\ar[l]^f}
 \end{displaymath}
can be sent to a square in $T$ given by
\begin{displaymath}
  \xymatrix{X_{00}\ar[d]_{i_*}\ar[r]^{g^*}&X_{10}\ar[d]^{j_*}\\
            X_{01}\ar[r]_{f^*}&X_{11}}
\end{displaymath}
where $f^*$ denotes the image of $f$ under the functor $\Finop\rightarrow T$, and $i_*$ denotes the map from $X_{00}$ to $X_{01}$ given by the identity on $X_{00}$ and the zero map elsewhere.

Since $T$ is a quasicategory, these can be composed to obtain appropriate maps for all $m,n$, thus defining $\Finsop\rightarrow T$. By construction, this map agrees upon restriction to $\Finop$ with the structure map $\Finop\rightarrow T$.

It is not difficult to check that, wherever a choice was made in the above construction, the space of choices is contractible. It follows that the morphism $\Finsop\rightarrow T$ is homotopy unique, and thence that $\Finsop$ is indeed initial.
\end{proof}

\subsection{The theory $\Spam$ for $E_\infty$-semirings}
\label{spam-e-infty}

In this section, we use the formalism of distributive laws to introduce a quasicategory $\Spam$, which will turn out to be the algebraic theory for $E_\infty$-semirings.

An $E_\infty$-semiring is equipped with a multiplicative and an additive monoidal structure, such that the multiplication distributes over the addition. This means, at the level of algebraic theories, that a distributive law exchanges the additive and multiplicative operations. This section constructs that distributive law.

Rather than distributing two copies of $\Span$ and finding some way of dealing with the way in which diagonals have been provided twice, we distribute one copy of $\Span$ (encoding diagonals and the multiplicative monoidal structure) and one copy of $\Fin$ (encoding the additive monoidal structure). When describing an operation in the theory of semirings, we need only take diagonals once.

We will need some machinery to describe this distributive law.

A morphism $f:X\rightarrow Y$ of finite sets induces a natural pullback functor $f^*:\Fin_{/Y}\rightarrow\Fin_{/X}$. This has a left adjoint $f_!$, defined by composition. It also has a right adjoint $f_*$, the fibrewise sections functor, given on fibres by $(f_*A)_y=\left\{\text{sections of $A_y\rightarrow X_y$}\right\}$.

We shall make heavy use of these in what follows, and pause to prove a basic lemma.

\begin{prop}[Mackey property]
  \label{sets-mackey}
  A pullback diagram of finite sets induces a commutative diagram (up to natural isomorphism) of functors:
  \begin{displaymath}
    \xymatrix{W\pb{315}\ar[r]^f\ar[d]_\alpha&X\ar[d]^\beta&&\Fin_{/W}\ar[r]^{f_*}&\Fin_{/X}\\
              Y\ar[r]_g&Z&&\Fin_{/Y}\ar[u]^{\alpha^*}\ar[r]_{g_*}&\Fin_{/Z}.\ar[u]_{\beta^*}}
  \end{displaymath}
\end{prop}
\begin{proof}
  This is an easy fibre-by-fibre check: given $A\in\Fin_{/Y}$, we have
 \begin{align*}
   (f_*\alpha^*A)_x&=\left\{\text{sections of $(\alpha^*A)_x\rightarrow W_x$}\right\}\\
                   &=\left\{\text{sections of $\coprod_{f(w)=x}(\alpha^* A)_w\rightarrow W_x$}\right\}\\
                   &=\prod_{f(w)=x}A_{\alpha(w)};
 \end{align*}
 while $(g_*A)_z=\Fin(Y_z,A_z)$, giving
 \begin{align*}
   (\beta^*g_*A)_x&=(g_*A)_{\beta(x)}\\
                  &=\left\{\text{sections of $A_{\beta(x)}\rightarrow Y_{\beta(x)}$}\right\}\\
                  &=\prod_{g(y)=\beta(x)}A_y.
 \end{align*}
 These are isomorphic by the universal property of the pullback.
\end{proof}

In the sequel, we shall have much use for rectangular diagrams of the shape
\begin{displaymath}
\xymatrix{
A\ar[d]&B\ar[l]\ar[r]\ar[d]&C\ar[d]\\
X&Y\ar[l]\ar[r]&Z.}
\end{displaymath}
By the \emph{edge} of such a diagram, we mean the subdiagram
\begin{displaymath}
\xymatrix{
A\ar[d]&&\\
X&Y\ar[l]\ar[r]&Z.}
\end{displaymath}
We shall say that such a diagram is \emph{precromulent} if the right-hand square is a pullback, and \emph{cromulent} if it is terminal among precromulent diagrams with the same edge.

The following result shows that cromulent diagrams are in good supply:
\begin{lemma}
\label{cromulents-exist}
Any edge can be extended to a cromulent diagram
   \begin{displaymath}
     \xymatrix{A\ar[d]&g^*g_*f^*A\ar[l]\ar[r]\ar[d]&g_*f^*A\ar[d]\\
               X&Y\ar[l]^f\ar[r]_g&Z.}
   \end{displaymath}
Here the top left-hand map is made from the counit map $g^*g_*f^*A\rightarrow f^*A$ composed with the structure map for the pullback $f^*A$.
\end{lemma}
\begin{proof}
The right-hand square is a pullback by construction; thus the diagram is precromulent. We must show it terminal among precromulent diagrams.

Given another precromulent diagram
\begin{displaymath}
\xymatrix{
A\ar[d]&B\pb{315}\ar[l]\ar[r]\ar[d]&C\ar[d]\\
X&Y\ar[l]^f\ar[r]_g&Z,}
\end{displaymath}
we have a unique map $g^*C=B\rightarrow f^*A$ factoring through $A$ and $Y$. This is adjoint to a map $C\rightarrow g_*f^*A$; and we get the corresponding map $B\rightarrow g^*g_*f^*A$ from the universal property of the pullback.
\end{proof}
Since to be cromulent is to satisfy a universal property, this extension is essentially unique. It follows that any cromulent diagram is isomorphic to one of this form.

We now build a bisimplicial set $\DR$. The set of $(m,n)$-simplices $\DR_{m,n}$ is defined to be the set of diagrams $F:C^m\times\Delta^n\rightarrow\Set$ (where $C^m$ is the poset of subintervals of $\{0,\ldots,m\}$ as before), such that:
\begin{itemize}
 \item for any $0\leq k\leq n$, the restriction $C^m\times\{k\}\rightarrow\Set$ satisfies the pullback property, and
 \item for any $0\leq i\leq j\leq m$ and $0\leq k\leq l\leq n$, the restriction
   \begin{displaymath}
     \xymatrix{F(i,k)\ar[d]&F(ij,k)\ar[l]\ar[r]\ar[d]&F(j,k)\ar[d]\\
               F(i,l)      &F(ij,l)\ar[l]\ar[r]      &F(j,l)}
   \end{displaymath}
   is cromulent.
\end{itemize}

\begin{thm}
 \label{DR-is-distrib}
 The bisimplicial set $\DR$ is a distributive law.
\end{thm}
\begin{proof}
 Since $\DR_{*,0}=\Span$ and $\DR_{0,*}=\Fin$ are both quasicategories, the lifting problems for $\Lambda^{n,0}_{k,0}$ and $\Lambda^{0,n}_{0,k}$ are soluble.

 We proceed to solve the low-dimensional lifting problems one-by-one, interspersed by technical lemmas where required.

 \begin{innerproof}{Lifting for $\Lambda^{1,1}_{0,1}$}
 A lifting problem for $\Lambda^{1,1}_{0,1}$ merely consists of extending a diagram
 \begin{displaymath}
     \xymatrix{A_0\ar[d]&&\\
               X_0&X_{01}\ar[l]^f\ar[r]_g&X_1,}
 \end{displaymath}
 to a cromulent rectangle, which is possible by Lemma \ref{cromulents-exist}.
 \end{innerproof}

 \begin{innerproof}{Lifting for $\Lambda^{2,1}_{0,1}$}
 A lifting problem for $\Lambda^{2,1}_{0,1}$ gives us a diagram
 \begin{displaymath}
 \xymatrix{
   &&A_{02}\ar@/_2ex/[ddll]\ar[ddrr]\ar@/^2.5ex/[ddd]&&\\
   &A_{01}\ar[dl]\ar[dr]\ar[ddd]&&&\\
   A_0\ar[ddd]&&A_1\ar@/^2.5ex/[ddd]&&A_2\ar[ddd]\\
   &&X_{02}\ar[dl]_{f_{02}}\ar[dr]^{g_{02}}&&\\
   &X_{01}\ar[dl]_{f_{01}}\ar[dr]^{g_{01}}&&X_{12}\ar[dl]_{f_{12}}\ar[dr]^{g_{12}}&\\
   X_0&&X_1&&X_2,}
 \end{displaymath}
  with
 \begin{align*}
  A_{01}&={g_{01}}^*{g_{01}}_*{f_{01}}^*A_0,\\
   A_{1}&={g_{01}}_*{f_{01}}^*A_0\\
  A_{02}&={(g_{12}g_{02})}^*{(g_{12}g_{02})}_*{(f_{01}f_{02})}^*A_0,\quad\text{and}\\
   A_{2}&={(g_{12}g_{02})}_*{(f_{01}f_{02})}^*A_0
 \end{align*}
  with all maps the appropriate ones.

  We can fill in this diagram by using
  \begin{displaymath}
     \xymatrix{A_1\ar[d]&{g_{12}}^*{g_{12}}_*{f_{12}}^*A_1\ar[l]\ar[r]\ar[d]&{g_{12}}_*{f_{12}}^*A_1\ar[d]\\
               X_1&X_{12}\ar[l]^{f_{12}}\ar[r]_{g_{12}}&X_2,}
  \end{displaymath}
  and define $A_{12}$ accordingly. We must merely check that this is consistent.

  The first check is that $A_2={(g_{12}g_{02})}_*{(f_{01}f_{02})}^*A_0$ is isomorphic to ${g_{12}}_*{f_{12}}^*A_1={g_{12}}_*{f_{12}}^*{g_{01}}_*{f_{01}}^*A_0$; this follows from the Mackey property.

  Now, the map $A_{02}\rightarrow A_0$ factors through the map $A_{01}\rightarrow A_0$: the map $A_{02}\rightarrow A_{01}$ is given by the following chain of maps:
    \begin{align*}
        A_{02}=&{(g_{12}g_{02})}^*{(g_{12}g_{02})}_*{(f_{01}f_{02})}^*A_0\\
              =&{g_{02}}^*{g_{12}}^*{g_{12}}_*{g_{02}}_*{f_{02}}^*{f_{01}}^*A_0\\
\longrightarrow&{g_{02}}^*{g_{02}}_*{f_{02}}^*{f_{01}}^*A_0\qquad\text{(using the counit)}\\
              =&{g_{02}}^*{f_{12}}^*{g_{01}}_*{f_{01}}^*A_0\qquad\text{(using the Mackey property)}\\
              =&{f_{02}}^*{g_{01}}^*{g_{01}}_*{f_{01}}^*A_0\\
\longrightarrow&{g_{01}}^*{g_{01}}_*{f_{01}}^*A_0\qquad\text{(using a structure map of the pullback)}\\
              =&A_{01}.
    \end{align*}

  Also, the map $A_{02}\rightarrow A_2$ factors through $A_{12}\rightarrow A_2$. The map $A_{02}\rightarrow A_{12}$ is given by the following chain of maps:
    \begin{align*}
        A_{02}=&{g_{02}}^*{g_{12}}^*{g_{12}}_*{g_{02}}_*{f_{02}}^*{f_{01}}^*A_0\\
\longrightarrow&{g_{12}}^*{g_{12}}_*{g_{02}}_*{f_{02}}^*{f_{01}}^*A_0\qquad\text{(using a structure map of the pullback)}\\
              =&{g_{12}}^*{g_{12}}_*{f_{12}}^*{g_{01}}_*{f_{01}}^*A_0\qquad\text{(using the Mackey property)}\\
              =&{g_{12}}^*{g_{12}}_*{f_{12}}^*A_1\\
              =&A_{12}
    \end{align*}

  Finally, we observe that the top square is a pullback, completing the check required to produce an element of $D_{2,1}$.
 \end{innerproof}

 \begin{innerproof}{Lifting for $\Lambda^{2,1}_{1,1}$}
   A lifting problem for $\Lambda^{2,1}_{1,1}$ consists of being given the following:
   \begin{displaymath}
    \xymatrix{
   &A_{01}\ar[ddd]\ar[dl]\ar[dr]&&A_{12}\ar[ddd]\ar[dl]\ar[dr]&\\
   A_0\ar[ddd]&&A_1\ar@/^2.5ex/[ddd]&&A_2\ar[ddd]\\
   &&X_{02}\ar[dl]_{f_{02}}\ar[dr]^{g_{02}}&&\\
   &X_{01}\ar[dl]_{f_{01}}\ar[dr]^{g_{01}}&&X_{12}\ar[dl]_{f_{12}}\ar[dr]^{g_{12}}&\\
   X_0&&X_1&&X_2,}
   \end{displaymath}
   consisting of two $(1,1)$-cells and a $(2,0)$-cell glued together. This readily extends to a full $(2,1)$-cell in an evident way.
 \end{innerproof}

 Using this, we can demonstrate a worthwhile reduction principle:
 \begin{lemma}
 \label{reduction-of-amalgs}
 Given $(1,1)$-cells of $\DR$
 \begin{displaymath}
   \vcenter{\xymatrix{X'\ar[d]&U'_1\ar[d]\ar[l]\ar[r]&Y'_1\ar[d]\\X&U_1\ar[l]^{f_1}\ar[r]_{g_1}&Y_1}}
     \qquad\text{and}\qquad
   \vcenter{\xymatrix{X'\ar[d]&U'_2\ar[d]\ar[l]\ar[r]&Y'_2\ar[d]\\X&U_2\ar[l]^{f_2}\ar[r]_{g_2}&Y_2,}}
 \end{displaymath}
 there is a $(1,1)$-cell
 \begin{displaymath}
   \xymatrix{X'\ar[d]&U'_1\sqcup U'_2\ar[d]\ar[l]\ar[r]&Y'_1\sqcup Y'_2\ar[d]\\X&U_1\sqcup U_2\ar[l]^{f_1\sqcup f_2}\ar[r]_{(g_1,g_2)}&Y_1\sqcup Y_2.}
 \end{displaymath}
 Moreover, any $(1,1)$-cell arises as a sum of cells of the form
 \begin{displaymath}
   \xymatrix{X'\ar[d]&U'\ar[d]\ar[l]\ar[r]&Y'\ar[d]\\X&U\ar[l]\ar[r]&{*}}
 \end{displaymath}
 in this sense.
 \end{lemma}
 \begin{innerproof}{Proof of lemma}
 The first part is a straightforward check.
 
 The second part follows from the lifting results: given a $(1,1)$-cell
 \begin{displaymath}
 \xymatrix{X'\ar[d]&U'\ar[d]\ar[l]\ar[r]&Y'\ar[d]\\X&U\ar[l]\ar[r]&Y,}
 \end{displaymath}
 and an element $y\in Y$, we can extend it to a diagram in $\DR$ of shape $(\Delta^1\boxtimes\Delta^1)\cup(\Delta^2\boxtimes\Lambda^1_1)$ by extending the base to
 \begin{displaymath}
 \xymatrix{&&U_y\pb{270}\ar[dl]\ar[dr]&&\\&U\ar[dl]\ar[dr]&&{*}\ar[dl]_y\ar[dr]&\\X&&Y&&{*},}
 \end{displaymath}
 then extend this to a diagram of shape $\Lambda^{2,1}_{1,1}$ by a lifting of type $\Lambda^{1,1}_{0,1}\rightarrow\Delta^{1,1}$. Then we can use lifting for $\Lambda^{2,1}_{1,1}$ to get a complete $(2,1)$-cell.
 
 It is straightforward to check that the resulting $(1,1)$-faces over $X\leftarrow U_y\rightarrow *$ have sum isomorphic to the original $(1,1)$-cell, where we let $y$ range over all elements of $Y$.
 \end{innerproof}

 \begin{innerproof}{Lifting for $\Lambda^{1,2}_{0,1}$}
  A lifting problem for $\Lambda^{2,1}_{0,1}$ consists of a diagram
   \begin{displaymath}
    \xymatrix{
   X_2\ar[d]&U_2\ar[l]  \ar[d]\ar[r]    &Y_2\ar[d]\\
   X_1\ar[d]&U_1\ar[l]_F\ar[d]\ar[r]^G  &Y_1\ar[d]\\
   X_0      &U_0\ar[l]_f      \ar[r]^g  &Y_0;}
   \end{displaymath}
  the problem is to show that the top row is given by
  $$X_2\longleftarrow g^*g_*f^*X_2\longrightarrow g_*f^*X_2.$$

  Using Lemma \ref{reduction-of-amalgs}, it suffices to consider the situation where $Y_0=*$. In this case all the information simplifies considerably.

  Indeed, $Y_1=\Set_{/X_0}(U_0,X_1)$ and $U_1=U_0\times\Set_{/X_0}(U_0,X_1)$. The map $U_1\rightarrow X_1$ is the evaluation map, and the map $U_1\rightarrow Y_1$ the projection.

  Then we have
  \begin{align*}
   \Set_{/Y_1}(T,Y_2)&=\Set_{/X_1}(F_!G^*T,X_2)\\
                     &=\Set_{/X_1}(T\times U_0,X_2)\\
                     &=\Set_{/X_0}(T\times U_0,X_2).
  \end{align*}
  and $U_2$ is the pullback $Y_2\timeso{Y_1}U_1=Y_2\times U_0$. These mean that the top row has the required form.
 \end{innerproof}

 \begin{innerproof}{Lifting for $\Lambda^{1,2}_{0,2}$}
   As before, we use Lemma \ref{reduction-of-amalgs}, and thus obtain a diagram isomorphic to the following:
   \begin{displaymath}
    \xymatrix{X_2\ar[d]&U_0\times\Set_{/X_0}(U_0,X_2)\ar@/^2ex/[dd]\ar[l]\ar[r]&\Set_{/X_0}(U_0,X_2)\ar@/^2ex/[dd]\\
              X_1\ar[d]&U_0\times\Set_{/X_0}(U_0,X_1)\ar[d]        \ar[l]\ar[r]&\Set_{/X_0}(U_0,X_2)\ar[d]\\
              X_0&U_0\ar[l]\ar[r]&{*}}
   \end{displaymath}
   The map $X_2\rightarrow X_1$ induces vertical maps making the diagram commute, and it's immediate to check that the top half is a $(1,1)$-cell, making the whole diagram a $(1,2)$-cell.
 \end{innerproof}

 \begin{innerproof}{All higher liftings}
   By definition, a cell of $\DR$ is uniquely determined by its $(2,0)$-faces and its $(0,1)$-faces, and, given a set of putative such faces, they extend to a full cell of $\DR$ if there exists a compatible set of $(2,1)$-faces.

   It is straightforward to check that all lifting problems $\Lambda^{m,n}_{i,j}\rightarrow\Delta^{m,n}$ with $m\geq 2$ or $n\geq 2$ provide this information.
 \end{innerproof}

This completes the checks required to prove the theorem.
\end{proof}

Accordingly, we define:
\begin{defn}
The \emph{multiplicative Span} category $\Spam$ is defined to be the composite quasicategory $\Span\circ_\DR\Fin$.
\end{defn}

Notice that 0-cells of $\Spam$ are finite sets, while 1-cells from $X$ to $Y$ can be written as \emph{ringlike span diagrams}
$$\xymatrix{X&\ar[l]_\Delta^f U\ar[r]^\Pi_g&V\ar[r]^\Sigma_h&Y,}$$
the first two arrows forming a 1-cell in $\Span$ and the third a 1-cell in $\Fin$.

To save space, in this section we shall sometimes depict a ringlike span diagram from $X$ to $Y$ as an arrow $\xymatrix{X\rar[r]&Y}$.

Moreover, we have the following:
\begin{prop}
\label{spam-two-cat}
  $\Spam$ is a $(2,1)$-category.
\end{prop}
\begin{proof}
  By considering the definition of $\Span$ as a composite quasicategory, and the proof of Proposition \ref{comp-is-quasicat}, it suffices to show that liftings for all inner bihorn inclusions $\Lambda^{m,n}_{i,j}\rightarrow\Delta^{m,n}$ are unique if $m+n>2$.

  But this is evident from the proof of Theorem \ref{DR-is-distrib} above.
\end{proof}

We now elucidate the structure of $\Spam$ further.

There are two natural maps of $(2,1)$-categories $\Span\rightarrow\Spam$:
\begin{itemize}
\item The inclusion $I_\Pi$ of the multiplicative monoid structure given by
$$\Span\isom\Span\circ_{(\DR)'}\Fin_0\longrightarrow \Span\circ_\DR\Fin=\Spam,$$
where $\Fin_0$ is the discrete simplicial set on objects $\Fin_0$, and $(\DR)'$ is the restriction of $\DR$ to $\Fin_0$, and
\item The inclusion $I_\Sigma$ of the additive monoid structure given by
$$\Span\isom\Finop\circ_{D(\Fin)}\Fin\longrightarrow\Span\circ_\DR\Fin=\Spam.$$
\end{itemize}

In particular, we have a commutative diagram
\begin{displaymath}
 \xymatrix{\Finsop\ar[r]\ar[d]&\Span\ar[d]^{I_\Pi}\\
           \Span\ar[r]_{I_\Sigma}&\Spam,}
\end{displaymath}
with the following property:
\begin{thm}
 This map $\Finsop\rightarrow\Spam$ makes $\Spam$ into a theory.
\end{thm}
\begin{proof}
 The aim is to show that $\Spam$ has finite products, given by disjoint unions of sets.

 It is straightforward to show that $\Spam$ has $\emptyset$ as terminal object. Actually, more is true: the projection $\Spam_{/\emptyset}\rightarrow\Spam$ is an isomorphism, and is thus clearly acyclic Kan.

 Indeed, an $n$-cell of $\Spam_{/\emptyset}$ corresponds to an $(n+1)$-cell of $\Spam$ with $(n+1)$-st vertex sent to the empty set. But all the vertices except those of the $(n+1)$-st face admit morphisms to the $(n+1)$-st vertex, and are thus sent to the empty set themselves. Thus everything but this face carries trivial information.

 We now turn our attention to binary products.

 Suppose given finite sets $X_1$ and $X_2$, regarded as a diagram $D:2\rightarrow\Spam$. We define a cone $C:1\star 2\rightarrow\Spam$ via the projection maps
\begin{align*}
\rspan{X_1\sqcup X_2}{}{X_1}{}{X_1}{}{X_1}&
    \qquad\text{and}\\
\rspan{X_1\sqcup X_2}{}{X_2}{}{X_2}{}{X_2}&
\end{align*}

 We have to show that $\Spam_{/C}\rightarrow\Spam_{/D}$ is acyclic Kan; according to Propositions \ref{spam-two-cat} and \ref{acyclic-kan}, we have only four checks to make.

 \begin{innerproof}{Lifting for $\emptyset\rightarrow *$}
 A diagram
 \begin{displaymath}
 \xymatrix{\emptyset\ar[r]\ar[d]&{*}\ar[d]\\
           \Spam_{/C}\ar[r]&\Spam_{/D},}
 \end{displaymath}
 gives a cone over $D$ consisting of diagrams
\begin{displaymath}
\rspan{A}{}{A'_1}{}{A_1}{}{X_1}
    \qquad\text{and}\qquad
\rspan{A}{}{A'_2}{}{A_2}{}{X_2}.
\end{displaymath}
 
 We now show that this extends to a cone over $C$. We let the $A$ factor through the product via the diagram
  $$\rspan{A}{}{A'_1\sqcup A'_2}{}{A_1\sqcup A_2}{}{X_1\sqcup X_2}.$$
 Now all we need to do is check that this really does extend to a diagram $1\star 2\rightarrow\Spam$. However, it is easily checked that a composite
  $$\xymatrix{A&A'_1\sqcup A'_2\ar[l]_\Delta\ar[r]^\Pi&A_1\sqcup A_2\ar[r]^\Sigma&X_1\sqcup X_2&X_1\ar[l]_\Delta\ar[r]^\Pi&X_1\ar[r]^\Sigma&X_1}$$
 is given as required by the diagram
  $$\rspan{A}{}{A'_1}{}{A_1}{}{X_1},$$
 and similarly for $X_2$.
 \end{innerproof}
 
 \begin{innerproof}{Lifting for $\partial\Delta^1\rightarrow\Delta^1$}
 We suppose given a diagram
 \begin{displaymath}
 \xymatrix{\partial\Delta^1\ar[r]\ar[d]&\Delta^1\ar[d]\\
           \Spam_{/C}\ar[r]&\Spam_{/D}.}
 \end{displaymath}
  The map $\Delta^1\rightarrow\Spam_{/D}$ gives us two objects $A$ and $B$, equipped with a ringlike span diagram $\rspan{A}{}{U}{}{V}{}{B}$, and ringlike span diagrams from each of $A$ and $B$ to each of $X_1$ and $X_2$ looking like $\rspan{A}{}{A'_1}{}{A_1}{}{X_1}$, such that the given maps $\xymatrix{A\rar[r]&X_1}$ and $\xymatrix{A\rar[r]&X_2}$ are composites of $\xymatrix{A\rar[r]&B\rar[r]&X_1}$ and $\xymatrix{A\rar[r]&B\rar[r]&X_2}$ respectively. 

  The map $\partial\Delta^1\rightarrow\Spam_{/C}$ gives us ringlike spans $\xymatrix{A\rar[r]&X_1\sqcup X_2}$ and $\xymatrix{B\rar[r]&X_1\sqcup X_2}$. These agree with the maps given above upon composition with the maps $\xymatrix{X_1\sqcup X_2\rar[r]&X_1}$ and $\xymatrix{X_1\sqcup X_2\rar[r]&X_2}$.

  But, by the definition of the composition, this just means that they can be taken to be
  \begin{align*}
    \rspan{A}{}{A'_1\sqcup A'_2}{}{A_1\sqcup A_2}{}{X_1\sqcup X_2}&\quad\text{and}\\
    \rspan{B}{}{B'_1\sqcup B'_2}{}{B_1\sqcup B_2}{}{X_1\sqcup X_2}&, 
  \end{align*}
  respectively.

  It's then easy to check that the ringlike span $\xymatrix{A\rar[r]&B}$ commutes with those down to $X_1\sqcup X_2$, which fills in the diagram as necessary.
 \end{innerproof}

 \begin{innerproof}{Lifting for $\partial\Delta^2\rightarrow\Delta^2$ and $\partial\Delta^3\rightarrow\Delta^3$}
   Similar arguments using setwise coproducts works here too.
 \end{innerproof}
\end{proof}

\subsection{Homspaces in $\Spam$}

By analogy with subsection \ref{homs-in-span}, we seek a description of the homspaces of $\Spam$.

We have
$$\Spam(1,1)=\coprod_{f\in\Arr(\Fin)}B\Aut(f)=\coprod_{m_1,m_2,\ldots}\prod_kB(\Sigma_{m_k}\wr\Sigma_k),$$
where the sum is taken over all sequences $\{m_i\}$ of nonnegative integers with finite support. This is because an object in $\Arr(\Fin)$ is classified up to isomorphism by the number of fibres of each size, and the automorphisms can permute the collections of fibres of the same size and also permute the elements of each fibre.

Indeed, more generally,
$$\Spam(X,Y)=\coprod_{f\in\Arr(\Fin)_{/(X\times Y\rightarrow Y)}}B\Aut(f),$$
(where the automorphism groups are taken in the category $\Arr(\Set)_{/(X\times Y\rightarrow Y)}$).

\subsection{Multilinear maps}

The machinery of ringlike span diagrams allows approaches to other questions of multiplicative structure. In this section we sketch an approach to multilinear maps between Lawvere monoid objects.

Given Lawvere monoid objects $X,Y,Z:\Span\rightarrow\cC$ in a category $\cC$, a bilinear map $X\otimes Y\rightarrow Z$ ought to consist of a natural transformation
\begin{displaymath}
 \xymatrix{\Span\times\Span\ar[d]_{\Pi}\ar[rrr]^{X\times Y}&\Ar[d]&&\cC,\\
           \Span\ar@/_4ex/[urrr]_Z&&}
\end{displaymath}
where the vertical map is induced by taking Cartesian products of sets, and the top horizontal map sends $(A,B)$ to $X(A)\times Y(B)$.

This is an imperfect definition. One problem is that neither Cartesian products of sets nor (usually) products of objects in $\cC$ are defined on the nose. Another is that natural transformations of maps of quasicategories are a nuisance to study in the world of quasicategories.

So, instead, one could expect to describe the collection of triples of Lawvere monoids $(X,Y,Z)$ equipped with a bilinear map $X\otimes Y\rightarrow Z$ as models of a three-sorted theory.

We can describe this theory; to motivate it we write down what the operations are in the case of bilinear maps between abelian groups; other monoid objects in an ordinary category would do.

A natural map from $(X^{A_0}\times Y^{B_0}\times Z^{C_0})$ to $(X^{A_1}\times Y^{B_1}\times Z^{C_1})$ can be described by its action on its generators:
$$x_i\longmapsto\sum_{i'\in I_i}x_{i'},\qquad
  y_j\longmapsto\sum_{j'\in J_j}y_{j'},\qquad
  z_k\longmapsto\sum_{k'\in K_k}z_{k'}+\sum_{l'\in L_k}(x_{l'}\otimes y_{l'}),$$
for suitable collections of sets.

Equivalently, it is described by a span diagram of the form
$$A_0\sqcup B_0\sqcup C_0\stackrel{f}{\longleftarrow} A_{01}\sqcup B_{01}\sqcup C_{01}\sqcup E\sqcup E\stackrel{g}{\longrightarrow}A_{01}\sqcup B_{01}\sqcup C_{01}\sqcup E\stackrel{h}{\longrightarrow}A_1\sqcup B_1\sqcup C_1,$$
where:
\begin{itemize}
 \item $f$ maps $A_{01}$ and the left-hand $E$ to $A_0$, $B_{01}$ and the right-hand $E$ to $B_0$, and $C_{01}$ to $C_0$;
 \item $g$ is the identity on $A_{01}$, $B_{01}$ and $C_{01}$ and the fold map on $E$; and
 \item $h$ maps $A_{01}$ to $A_1$, $B_{01}$ to $B_1$, and $C_{01}$ and $E$ to $C_0$.
\end{itemize}

It is obvious from the origins discussed above, and otherwise a simple categorical check, that composites of two ringlike spans of this form is again of this form.

Thus ringlike span diagrams with all 1-cells of this form make up a subquasicategory $\Spam(\bullet\otimes\bullet\rightarrow\bullet)$ of $\Spam\timeso{\Spam_0}\Spam_0^3$, the quasicategory of ringlike span diagrams with partitions of their vertices into three.

Moreover, there are three inclusions $\Spam\longrightarrow\Spam(\bullet\otimes\bullet\rightarrow\bullet)$, representing each of $X$, $Y$ and $Z$.

It is clear that a precisely similar discussion is possible for multilinear maps $X_1\otimes\cdots\otimes X_n\rightarrow Z$.

\subsection{Vector spaces as a model of $\Spam$ in categories}
\label{Vect-as-model-of-Spam}

The purpose of this subsection is to demonstrate that, despite the abstraction, it is not too difficult to build models of $\Spam$.

We give the category $\Vect$ of vector spaces (over some fixed field) the structure of a model of $\Spam$ in the $(2,1)$-category $\Cat$ of ordinary categories, functors and natural isomorphisms. This can be viewed as a quasicategory thanks to the results of Section \ref{two-one-categories}, and as a quasicategory it has a natural functor to $\Cinfty$, given by the nerve.

An alternative treatment of the bimonoidal properties of $\Vect$ has been given by Elmendorf and Mandell \cite{Elmendorf-Mandell}. Their approach is tricky: it requires careful analysis of a skeleton of $\Vect$. This approach is rather more forgiving.

Since both $\Spam$ and $\Cat$ are $(2,1)$-categories, they can be specified by giving the images of the $0$-simplices, $1$-simplices and $2$-simplices, and then it suffices to check that this respects the $3$-simplices.

A $0$-cell of $\Spam$ is a finite set $X$, and to it we associate the category $\Vect^X$ of $X$-indexed collections of vector spaces.

A $1$-cell of $\Spam$ consists of a diagram $X\stackrel{f}\leftarrow A\stackrel{g}\rightarrow B\stackrel{h}\rightarrow Y$, and to it we associate the functor
$$\Vect^X\stackrel{\Delta_f}\longrightarrow\Vect^A\stackrel{\otimes_g}\longrightarrow\Vect^B\stackrel{\oplus_h}\longrightarrow\Vect^Y.$$

A $2$-cell of $\Spam$ consists of $1$-cells $\xymatrix{X\rar[r]&Y}$, $\xymatrix{Y\rar[r]&Z}$ and a composition of them. In detail, given ringlike span diagrams $X\leftarrow A\rightarrow B\rightarrow Y$ and $Y\stackrel{f}\leftarrow C\stackrel{g}\rightarrow D\rightarrow Z$, the composite $\xymatrix{X\rar[r]&Z}$ must supply the dotted lines in a diagram as below:
\begin{displaymath}
\xymatrix @!=1.3pc {
&&\bullet\pb{270}\ar[dl]\ar[dr]\dar@/_4ex/[ddll]\dar@/^4ex/[ddrr]&&\\
&A\ar[dl]\ar[dr]&&g^*g_*f^*B\ar[dl]\ar[dr]\ar[d]&\\
X&&B\ar[d]&C\ar[dl]^f\ar[dr]_g&g_*f^*B\ar[d]\dar@/^4ex/[dd]\\
&&Y&&D\ar[d]\\
&&&&Z}
\end{displaymath}
It is now necessary to supply a natural isomorphism between the two functors $\Vect^X\rightarrow\Vect^Z$ we have obtained: $\xymatrix{X\rar[r]&Y\rar[r]&Z}$ and $\xymatrix{X\rar[r]&Z}$. However, one can readily be produced by pasting the natural isomorphisms obtained from each of the cells in the diagram above: each one describes a well-known isomorphism in the theory of vector spaces (associativity and naturality of direct sums and of tensor products, and distributivity of direct sums and tensor products).

Then it is a straightforward (though tedious) check that $3$-cells of $\Spam$ are sent to $3$-cells of $\Cat$; it reduces to the coherence properties of the standard isomorphisms used in constructing the $2$-cells.

\section{Ends and coends in quasicategories}
\label{ends-and-coends}

Ends and coends are a familiar tool in classical category theory: they are treated in \cite{MacLane}. They abstract the familiar concept of systematically gluing objects together along common interfaces: for example, producing the geometric realisation of a simplicial set by gluing topological $n$-simplices along their boundaries as prescribed by the structure maps of the simplicial set.

In this section we introduce a theory of ends and coends for bifunctors between quasicategories. While not directly related to the other constructions in this thesis, it is deeply suggestive that this apparently natural approach uses a kind of formalism of span diagrams.

\subsection{Dinatural maps}

Let $\cC$ and $\cD$ be categories, and suppose we are given two functors $F,G:\cC^\op\times\cC\rightarrow\cD$.

Traditionally, a \emph{dinatural transformation} consists of a morphism $F(x,x)\rightarrow G(x,x)$ for every $x\in\Ob\cC$, such that for every $f:x\rightarrow y$ in $\cC$, the following diagram commutes:
\begin{displaymath}
 \xymatrix{&F(x,x)\ar[r]&G(x,x)\ar[dr]&\\
           F(y,x)\ar[ur]\ar[dr]&&&G(x,y)\\
           &F(y,y)\ar[r]&G(y,y)\ar[ur]&}
\end{displaymath}

Now suppose $\cC$ and $\cD$ are quasicategories. It is only meaningful to demand that these hexagons commute up to homotopy; then we shall want to impose coherence conditions on the resulting homotopies.

To model these, for any totally ordered set $K$, we define a poset $\Din(K)$ as follows. The objects consist of $I_*$ and $I^*$, where $I$ is a nonempty subinterval of $K$, and they have the following partial ordering:
\begin{align*}
 I_*&\leq J_*\quad\text{if $I\supseteq J$;}&\qquad I_*&\leq J^*\quad\text{if $I\cap J\neq\emptyset$;}\\
 I^*&\leq J_*\quad\text{ never;}           &\qquad I^*&\leq J^*\quad\text{if $I\subseteq J$.}
\end{align*}
It can readily be checked that this relation defines a poset, and is functorial in $K$.

To illustrate, here is a Hasse diagram of $\Din(\{0,1,2,3\})$, with intervals labelled by their startpoints and endpoints:
\begin{displaymath}
 \xymatrix @!=0.8pc {
   &&&0_*\ar[r]&0^*\ar[dr]&&&\\
   &&01_*\ar[ur]\ar[dr]&&&01^*\ar[dr]&&\\
   &02_*\ar[ur]\ar[dr]&&1_*\ar[r]&1^*\ar[ur]\ar[dr]&&02^*\ar[dr]&\\
   03_*\ar[ur]\ar[dr]&&12_*\ar[ur]\ar[dr]&&&12^*\ar[ur]\ar[dr]&&03^*\\
   &13_*\ar[ur]\ar[dr]&&2_*\ar[r]&2^*\ar[ur]\ar[dr]&&13^*\ar[ur]&\\
   &&23_*\ar[ur]\ar[dr]&&&23^*\ar[ur]&&\\
   &&&3_*\ar[r]&3^*\ar[ur]&&&
 }
\end{displaymath}

The poset $\Din(K)$ has two sub-posets, $\Din_*(K)$ and $\Din^*(K)$, consisting of the objects $\{I_*\}$ and the objects $\{I^*\}$ respectively; there are natural maps
\begin{align*}
 \Din_*(K)\longrightarrow&(\Delta^K)^\op\times(\Delta^K)  &  \Din^*(K)\longrightarrow&(\Delta^K)^\op\times(\Delta^K)\\
    (ij)_*\longmapsto&(j,i)                               &  (ij)^*\longmapsto&(i,j).
\end{align*}

If $\cC$ is a simplicial set, we define the simplicial set $\Din(\cC)$ to be the coend of the functor $\cC_{-}\times\Din(-):\Delta^\op\times\Delta\rightarrow\Set$.

Naturality of the $\Din_*$ and $\Din^*$ constructions defines subsimplicial sets $\Din_*(\cC)$ and $\Din^*(\cC)$, together with maps of simplicial sets
$$\Din_*(\cC)\longrightarrow\cC^\op\times\cC,\qquad\Din^*(\cC)\longrightarrow\cC^\op\times\cC.$$
Both $\Din_*(\cC)$ and $\Din^*(\cC)$ are full subsimplicial sets, in the sense that any simplex in $\Din(\cC)$ whose vertices are all in either one, is also in.

Accordingly, we now make the following definition:
\begin{defn}
 Let $\cC$ and $\cD$ be quasicategories.

 A \emph{dinatural transformation} between two functors $F,G:\cC^\op\times\cC\rightarrow\cD$ consists of a map of simplicial sets $\Din(\cC)\rightarrow\cD$, such that the following two diagrams of simplicials sets commute:
\begin{displaymath}
 \xymatrix{
   \Din_*(\cC)\ar[r]\ar[d]&\Din(\cC)\ar[d]\\
   \cC^\op\times\cC\ar[r]_F&\cD,}
\qquad\text{and}\qquad
 \xymatrix{
   \Din^*(\cC)\ar[r]\ar[d]&\Din(\cC)\ar[d]\\
   \cC^\op\times\cC\ar[r]_G&\cD.}
\end{displaymath}
\end{defn}

By the universal property of the coend, this is equivalent to having, a collection of maps $\Din(K)\rightarrow\cD$ for all $K$ and all $\alpha\in\cC_K$ compatible under faces and degeneracies, such that the diagrams 
\begin{displaymath}
\xymatrix{
   \Din_*(K)\ar[r]\ar[d]&\Din(K)\ar[d]\\
   (\Delta^K)^\op\times(\Delta^K)\ar[r]_F&\cD}
\qquad\text{and}\qquad
 \xymatrix{
   \Din^*(K)\ar[r]\ar[d]&\Din(K)\ar[d]\\
   (\Delta^K)^\op\times(\Delta^K)\ar[r]_G&\cD}
\end{displaymath}
both commute.

We have $\Din(\cC)_0=\Din_*(\cC)\sqcup\Din^*(\cC)$, and we can describe the 0-cell sets of each. Any 0-cell in $\Din_*(\cC)$ either occurs as a $\Din(\Delta^0)$, or is the $01$ part of some $\Din(\Delta^1)$ (since, by inspection, any object in any $\Din(\Delta^n)$ arises as one of these). Hence $\Din_*(\cC)_0=\cC_0\sqcup\cC_1$, and $\Din^*(\cC)$ is of course similar.

We can describe $k$-cells of $\Din_*(\cC)$ explicitly.

To start with, of course, a $k$-cell of $\Din_*(K)$ consists of a sequence of intervals
$$\emptyset\neq I_0\subset I_1\subset\cdots\subset I_n\subset K.$$

However, if any $x\in K$ does not feature as an endpoint of some interval $I_i$, then this $k$-cell occurs already as a $k$-cell of $\Din_*(K\backslash\{x\})$.

So we have
$$\Din_*(\cC)_n=\coprod_{\{I_i\}\in\cI_n}\cC_{|I_n|},$$
where the sum is taken over all the set $\cI_n$ of systems of intervals
$$\emptyset=I_{-1}\subset I_0\subset\cdots\subset I_n,$$ which have the property that each inclusion of $I_i$ into $I_{i+1}$ consists merely of adding at most two objects (at most one on the left of $I_i$ and at most one on the right).

Notice that $|\cI_n|=2^{2n+1}$; in a more down-to-earth fashion, for counting purposes, we can write
$$\Din_*(\cC)_n=\coprod_{S\subset\{-n,\ldots,n\}}\cC_{S}.$$

We show how to recover a family of intervals from such a subset $S\subset\{-n,\ldots,n\}$; it will be evident that this is an equivalence.
We take:
\begin{displaymath}
I_0 = 
\begin{cases}
\{\bullet,\bullet\}, & \text{if $0\in S$;}\\
\{\bullet\},         & \text{otherwise.}
\end{cases}
\end{displaymath}
Then, writing $\sqcup$ for concatenation of intervals, we take 
\begin{displaymath}
I_k = 
\begin{cases}
\{\bullet\}\sqcup I_{k-1}\sqcup \{\bullet\}, & \text{if $-k,k\in S$;}\\
\{\bullet\}\sqcup I_{k-1},                   & \text{if $-k\in S$, $k\notin S$;}\\
\hskip 27pt       I_{k-1}\sqcup \{\bullet\}, & \text{if $-k\notin S$, $k\in S$;}\\
\hskip 27pt       I_{k-1},                   & \text{if $-k,k\notin S$.}
\end{cases}
\end{displaymath}

More concisely still, we can write
$$\Din_*(\cC)_J=\coprod_{S\subset J^\op\vee J}\cC_S,$$
where $\vee$ denotes a concatenation of intervals that identifies adjacent endpoints.

\subsection{Dinatural transformations to constants}
\label{dinatural-to-constants}

By definition, a dinatural transformation from a functor $F:\cC^\op\times\cC\rightarrow\cD$ to a constant bifunctor is the same as a map $\Din(\cC)\rightarrow\cD$ which restricts to $F$ under the map $\cC^\op\times\cC\rightarrow\Din_*(\cC)$, and is constant on $\Din^*(\cC)$.

These are equivalent to functors $\Din(\cC)\star 1\rightarrow\cD$ which agree with $F$ on $\Din_*(\cC)$ and are constant on $\Din^*(\cC)\star 1$. These, in turn, are equivalent to functors $\Din_*(\cC)\star 1\rightarrow\cD$ which restrict to $F$ on $\Din_*(\cC)$.

Similarly we can argue that a dinatural transformation from a constant bifunctor to $F$ is a map $1\star\Din^*(\cC)\rightarrow\cD$ which restricts to $F$.

Accordingly, we make the following definitions:
\begin{defn}
\label{end-coend-defn}
  Let $F:\cC^\op\times\cC\rightarrow\cD$ be a bifunctor.

  An \emph{end} of $F$ is a limit of the composite
    $$\Din^*(\cC)\longrightarrow\cC^\op\times\cC\longrightarrow\cD,$$
  while a \emph{coend} of $F$ is a colimit of the composite
    $$\Din_*(\cC)\longrightarrow\cC^\op\times\cC\longrightarrow\cD.$$
\end{defn}

We would like to study this definition, in the case where $\cC$ and $\cD$ are ordinary categories. First we prove the following:
\begin{prop}
 If $\cC$ is an ordinary category, then so are $\Din_*(\cC)$ and $\Din^*(\cC)$.
\end{prop}
\begin{proof}
 We demonstrate it for $\Din_*(\cC)$; the argument for $\Din^*(\cC)$ is dual.

 It follows from the description
$$\Din_*(\cC)_n=\coprod_{\{I_i\}\in\cI_n}\cC_{|I_n|}$$
 from the preceding section. Indeed, a chain of intervals can be recovered from its increments, and the resulting element of $\cC_{|I_n|}$ can be recovered from its successive corresponding faces too, thus demonstrating as required that
$$\Din_*(\cC)_n=\Din_*(\cC)_1\timeso{\Din_*(\cC)_0}\cdots\timeso{\Din_*(\cC)_0}\Din_*(\cC)_1.$$
\end{proof}

We now exploit this approach to describe this category explicitly.

We write $\Aug(M)$ for the \emph{augmentation} of a monoid $M$: the monoid formed by adjoining a unit to $M$. As a set, this is $\{1\}\sqcup M$; we write $\tilde m$ for the elements $m\in M$ regarded as an element of $\Aug(M)$. This has the original monoid structure on $M$, given by $\widetilde{mn}=\tilde m\tilde n$, together with $1\cdot 1=1$ and $1\cdot \tilde m=\tilde m\cdot 1=\tilde m$ for all $m\in M$.

As a category is a monoid with multiple objects, we can similarly define the augmentation $\Aug(\cC)$ of a category $\cC$. The objects are the same as those of $\cC$, but the morphisms are given by:
\begin{align*}
\Aug(\cC)(x,x)&=\Aug(\cC(x,x))=\{1_x\}\sqcup\cC(x,x),\\
\Aug(\cC)(x,y)&=\cC(x,y)\qquad\text{if $x\neq y$}.
\end{align*}
Again, we write $\tilde f$ for the morphisms of $\cC$, regarded as morphisms in $\Aug(\cC)$. We omit the subscripts on the $1$'s when it can do no harm.

We shall need the evident property of $\Aug(\cC)$ that it has no isomorphisms except the arrows $1_x$.

We also introduce the \emph{category of factorisations} $\Fact(\cC)$ of a category $\cC$ (the nomenclature is as in \cite{Baues}). This has, as objects, arrows of $\cC$, and a morphism from $f$ to $g$ consists of a commutative diagram
\begin{displaymath}
 \xymatrix{\bullet\ar[d]_f\ar[r]^a&\bullet\ar[d]^g\\
           \bullet&\bullet\ar[l]^b.}
\end{displaymath}
Composition is defined by concatenation of squares.

Now, using these, we can characterise $\Din_*(\cC)$ where $\cC$ is an ordinary category.
\begin{prop}
  If $\cC$ is an ordinary category, then
    $$\Din_*(\cC)\isom\Fact(\Aug(\cC)).$$
\end{prop}
\begin{proof}
We construct an explicit isomorphism.

Firstly, since we have $\Din_*(\cC)_0=\Ob\cC\sqcup\Arr\cC$, we can identify the former kind with the objects $1_x\in\Ob\Fact(\Aug(\cC))$, and the latter kind with the objects $a(f)\in\Ob\Fact(\Aug(\cC))$

Then there are eight sorts of morphism, corresponding to the eight isomorphism classes of nested intervals. The correspondence is as follows:
\begin{itemize}
\item For the nested intervals $((\bullet))$, an object $x$ in $\cC$ yields the morphism $\id 1_x$:
\begin{displaymath}
 \xymatrix{x\ar[d]_{1_x}\ar[r]^{1_x}&x\ar[d]^{1_x}\\
           x&x\ar[l]^{1_x}}
\end{displaymath}
\item For the nested intervals $(\bullet(\bullet))$, a morphism $x\stackrel{f}{\rightarrow}y$ in $\cC$ yields the following morphism:
\begin{displaymath}
 \xymatrix{x\ar[d]_{\tilde f}\ar[r]^{\tilde f}&y\ar[d]^{1_y}\\
           y&y\ar[l]^{1_y}}
\end{displaymath}
\item For the nested intervals $((\bullet)\bullet)$, a morphism $x\stackrel{f}{\rightarrow}y$ in $\cC$ yields the following morphism:
\begin{displaymath}
 \xymatrix{x\ar[d]_{\tilde f}\ar[r]^{1_x}&x\ar[d]^{1_x}\\
           y&x\ar[l]^{\tilde f}}
\end{displaymath}
\item For the nested intervals $(\bullet(\bullet)\bullet)$, a sequence $x\stackrel{f}{\rightarrow}y\stackrel{g}{\rightarrow}z$ in $\cC$ yields the following morphism:
\begin{displaymath}
 \xymatrix{x\ar[d]_{\widetilde{gf}}\ar[r]^{\tilde f}&y\ar[d]^{1_y}\\
           z&y\ar[l]^{\tilde g}}
\end{displaymath}
\item For the nested intervals $((\bullet\bullet))$, a morphism $x\stackrel{f}{\rightarrow}y$ in $\cC$ yields the morphism $\id\tilde f$:
\begin{displaymath}
 \xymatrix{x\ar[d]_{\tilde f}\ar[r]^{1_x}&x\ar[d]^{\tilde f}\\
           y&y\ar[l]^{1_y}}
\end{displaymath}
\item For the nested intervals $(\bullet(\bullet\bullet))$, a sequence $x\stackrel{f}{\rightarrow}y\stackrel{g}{\rightarrow}z$ in $\cC$ yields the following morphism:
\begin{displaymath}
 \xymatrix{x\ar[d]_{\widetilde{gf}}\ar[r]^{\tilde{f}}&y\ar[d]^{\tilde{g}}\\
           z&z\ar[l]^{1_z}}
\end{displaymath}
\item For the nested intervals $((\bullet\bullet)\bullet)$, a sequence $x\stackrel{f}{\rightarrow}y\stackrel{g}{\rightarrow}z$ in $\cC$ yields the following morphism:
\begin{displaymath}
 \xymatrix{x\ar[d]_{\widetilde{gf}}\ar[r]^{1_x}&x\ar[d]^{\tilde f}\\
           z&y\ar[l]^{\tilde g}}
\end{displaymath}
\item For the nested intervals $(\bullet(\bullet\bullet)\bullet)$, a sequence $x\stackrel{f}{\rightarrow}y\stackrel{g}{\rightarrow}z\stackrel{h}{\rightarrow}w$ in $\cC$ yields the following morphism:
\begin{displaymath}
 \xymatrix{x\ar[d]_{\widetilde{hgf}}\ar[r]^{\tilde f}&y\ar[d]^{\tilde g}\\
           w&z\ar[l]^{\tilde h}}
\end{displaymath}
\end{itemize}
It is straightforward to check that this correspondence at the level of morphisms agrees with composition, and thus defines an isomorphism of categories.
\end{proof}

Also, in this framework it is easy to check that the map $\Din_*(\cC)\rightarrow\cC^\op\times\cC$ is the composite 
$$\Fact(\Aug(\cC))\longrightarrow\Fact(\cC)\longrightarrow\cC^\op\times\cC.$$
Here the first functor is associated to the functor $\Aug(\cC)\rightarrow\cC$ sending $\tilde f\mapsto f$ and $1_x\mapsto\id_x$. The second functor associates to each arrow its source and target.

This lets us prove:
\begin{prop}
 If $\cC$ and $\cD$ are ordinary categories, and $F:\cC^\op\times\cC\rightarrow\cD$ a bifunctor, then an end or coend for $F$ as in Definition \ref{end-coend-defn} above is the same as the usual notion of end or coend for categories.
\end{prop}
\begin{proof}
 We prove it for coends; the argument for ends is dual. Our argument is similar to \cite{MacLane}*{Exercise IX.6.3}.

 There is a category $\MLend{\cC}$, with
\begin{itemize}
 \item $\Ob\MLend{\cC}=\Ob\cC\sqcup\Arr\cC$, and
 \item morphisms consisting of the identities, plus morphisms $x\leftarrow f\rightarrow y$ for every morphism $f:x\rightarrow y$ in $\Arr\cC$.
\end{itemize}

 There is a natural map $\MLend{\cC}\rightarrow\cC^\op\times\cC$ (sending $x$ to $(x,x)$, and $f$ to $(x,y)$, where $f$ is the object of $\MLend{\cC}$ corresponding to $f:x\rightarrow y$). In \cite{MacLane} it is proved that a coend for $F$ is a colimit of the composite
$$\MLend{\cC}\rightarrow\cC^\op\times\cC\rightarrow\cD.$$

 We shall exhibit a functor $\MLend{\cC}\rightarrow\Din_*(\cC)$ commuting with the maps to $\cC^\op\times\cC$ and prove it to be cofinal.

 Since all $k$-cells of $\MLend{\cC}$ are degeneracies for $k>1$ (which means, as a category, that $\MLend{\cC}$ has no nontrivial compositions), we must merely give compatible destinations for the objects and morphisms.

 There are natural such destinations:
\begin{itemize}
 \item $x\in\cC_0\subset\Ob\MLend{\cC}$ is sent to $1_x$,
 \item $f\in\cC_1\subset\Ob\MLend{\cC}$ is sent to $\tilde f$,
 \item The morphism $f\rightarrow x$ in $\Arr\MLend{\cC}$ corresponding to $f:x\rightarrow y$ is sent to the morphism
\begin{displaymath}
 \xymatrix{x\ar[d]_{\tilde f}\ar[r]^{1_x}&x\ar[d]^{1_x}\\
           y&x\ar[l]^{\tilde f}}
\end{displaymath}
 \item The morphism $f\rightarrow y$ in $\Arr\MLend{\cC}$ corresponding to $f:x\rightarrow y$ is sent to the morphism
\begin{displaymath}
 \xymatrix{x\ar[d]_{\tilde f}\ar[r]^{\tilde f}&y\ar[d]^{1_y}\\
           y&y,\ar[l]^{1_y}}
\end{displaymath}
\end{itemize}

 This functor evidently commutes with the maps from $\MLend{\cC}$ and $\Din_*(\cC)$ down to $\cC^\op\times\cC$.

 According to \cite{MacLane}, to prove that the functor $i:\MLend{\cC}\rightarrow\Din_*(\cC)$ is cofinal, we must verify that for every $\alpha\in\Din_*(\cC)_0$, the comma category
$$\alpha\downarrow i=\MLend{\cC}\timeso{\Din_*(\cC)}\Din_*(\cC)_{\alpha/}$$
is nonempty and connected.

Since $\Ob\MLend{\cC}=\Ob\Din_*(\cC)$, the objects of the comma category are just the morphisms of $\Din_*(\cC)$ with source $\alpha$. Thus the nonemptiness condition is automatically satisfied: the identity morphism on $\alpha$ provides an object.

If $\alpha=1_x\in\Ob\Din_*(\cC)$, then, as the only morphism out of $1_x$ in $\Din_*(\cC)$ is the identity, the comma category is obviously connected.

It is only in the cases $\alpha=\tilde f$ that we have some work to do. We will connect everything to the identity morphism $\id \tilde f\in\Ob(\alpha\downarrow i)$, with a case-by-case approach.
\begin{enumerate}
 \item The morphisms
\begin{displaymath}
  \vcenter{\xymatrix{\ar[d]_{\tilde f}\ar[r]^{1_x}&\ar[d]^{1_x}\\
                     &\ar[l]^{\tilde f}}}
\qquad\text{and}\qquad
  \vcenter{\xymatrix{\ar[d]_{\tilde f}\ar[r]^{\tilde f}&\ar[d]^{1_y}\\
                     &,\ar[l]^{1_y}}}
\end{displaymath}
are in the image $\im(\alpha\downarrow i)$, and are thus obviously connected to the identity in the comma category. To be explicit, we can connect the former as follows:
\begin{displaymath}
\xymatrix{
&\ar[dl]_{1_x}\ar[dd]^(.7){\tilde f}\ar[dr]^{1_x}&\\
\ar[dd]_{\tilde f}\ar[rr]^(.7){1_x}&&\ar[dd]^{1_x}\\
&&\\
\ar[ur]^{1_x}&&\ar[ll]^{\tilde f}\ar[ul]_{\tilde f}
}
\end{displaymath}
The left-hand face is $\id\tilde f$, the right-hand face is the aforementioned morphism, and the front face is the connecting morphism in the image of $\MLend{\cC}$.

\item We now connect the morphisms
\begin{displaymath}
 \vcenter{\xymatrix{\ar[d]_{\widetilde{gf}}\ar[r]^{1_x}&\ar[d]^{\tilde f}\\
           &\ar[l]^{\tilde g}}}
 \qquad\text{and}\qquad
 \vcenter{\xymatrix{\ar[d]_{\widetilde{gf}}\ar[r]^{\tilde{f}}&\ar[d]^{\tilde{g}}\\
           &\ar[l]^{1_z}}}
\end{displaymath}
The former are connected to the morphisms
\begin{displaymath}
  \vcenter{\xymatrix{\ar[d]_{\tilde f}\ar[r]^{1_x}&\ar[d]^{1_x}\\
                     &\ar[l]^{\tilde f}}}
\end{displaymath}
studied in the previous paragraph as follows:
\begin{displaymath}
\xymatrix{
&\ar[dl]_{1}\ar[dd]^(.7){\widetilde{gf}}\ar[dr]^{1}&\\
\ar[dd]_{1}&&\ar[dd]^{\tilde f}\ar[ll]_(.7){1}\\
&&\\
\ar[ur]^{\widetilde{gf}}\ar[rr]_{\tilde f}&&\ar[ul]_{\tilde g}
}
\end{displaymath}
The latter can be connected in a symmetrical manner.

\item We now connect the morphisms
\begin{displaymath}
 \xymatrix{x\ar[d]_{\widetilde{gf}}\ar[r]^{\tilde f}&y\ar[d]^{1_y}\\
           z&y\ar[l]^{\tilde g}}
\end{displaymath}
to the morphisms
\begin{displaymath}
 \xymatrix{x\ar[d]_{\widetilde{gf}}\ar[r]^{1_x}&x\ar[d]^{\tilde f}\\
           z&y\ar[l]^{\tilde g}}
\end{displaymath}
studied in the previous paragraph as follows:
\begin{displaymath}
\xymatrix{
&\ar[dl]_{1}\ar[dd]^(.7){\widetilde{gf}}\ar[dr]^{\tilde f}&\\
\ar[dd]_{\tilde f}\ar[rr]^(.7){\tilde f}&&\ar[dd]^{1}\\
&&\\
\ar[ur]^{\tilde g}&&\ar[ll]^{1}\ar[ul]_{\tilde g}
}
\end{displaymath}

\item Lastly, we connect the morphisms
\begin{displaymath}
 \xymatrix{x\ar[d]_{\widetilde{hgf}}\ar[r]^{\tilde f}&y\ar[d]^{\tilde g}\\
           w&z\ar[l]^{\tilde h}}
\end{displaymath}
to the morphisms studied in the previous paragraph as follows:
\begin{displaymath}
\xymatrix{
&\ar[dl]_{\widetilde{gf}}\ar[dd]^(.7){\widetilde{hgf}}\ar[dr]^{\tilde f}&\\
\ar[dd]_{1}&&\ar[ll]_(.7){\tilde g}\ar[dd]^{\tilde g}\\
&&\\
\ar[ur]^{\tilde h}\ar[rr]_{1}&&\ar[ul]_{\tilde h}
}
\end{displaymath}
\end{enumerate}
This exhausts all the types of morphism.

 As this functor is cofinal, we have shown that there is an isomorphism from the colimit of $\MLend{\cC}$, the coend, to the colimit of $\Din_*(\cC)$, meaning that the two notions of end agree.
\end{proof}

\subsection{Nerves of weak simplicial objects}

Classically, given a simplicial set $X:\Delta^\op\rightarrow\Set$, we use the realisation functor $N:\Delta\rightarrow\Spaces$ to form a bifunctor $X\times N:\Delta^\op\times\Delta\rightarrow\Spaces$, and the coend is the geometric realisation of $X$.

Sometimes we can naturally produce something akin to a simplicial object, but where the relations between the faces and degeneracies only hold coherently up to higher homotopies. Accordingly, we make the following definition.

\begin{defn}
A \emph{weak simplicial object} in a quasicategory $\cC$ is a map of quasicategories $\Delta^\op\rightarrow\cC$.
\end{defn}

The value for us of subsection \ref{dinatural-to-constants}, above, is that this provides a means to take the realisation of a weak simplicial object, in a manner precisely comparable to the usual realisation.

Observe firstly that Theorem \ref{colimit-functor} gives us functorial tensorings in a finitely complete quasicategory $\cC$.

Indeed, we can define the tensoring $\otimes$ to be the composite
$$\otimes:\FinSpaces\times\cC\rightarrow\FinCatInfty\times\cC\rightarrow(\FinCatInfty)_{/\cC}\rightarrow\cC.$$
Here the first map is induced by the inclusion of Kan complexes into inner Kan complexes, the second map sends $(A,\cD)$ to the constant $A$-valued diagram $\cD\rightarrow\cC$, and the third map is a colimit functor as provided by Theorem \ref{colimit-functor}.

Indeed, given a complete category $\cC$ and a weak simplicial object $X:\Delta^\op\rightarrow\cC$, we define the \emph{realisation} of $X$ to be the coend of the bifunctor $X\otimes N:\Delta^\op\times\Delta\rightarrow\cC$. Here $N$ is the nerve functor $\Delta\rightarrow\Spaces$; and $\otimes$ denotes the tensoring over $\Spaces$ that arises in a complete category ($A\otimes S$ is the limit in $\cC$ of the constant diagram of shape $S$ with value $A$). Realisation is functorial in an evident sense.

Of course, if $\cC$ is cocomplete, then it admits all coends and so in particular has realisations of all weak simplicial objects. Moreover, if $\cC$ is a discrete category, then the realisations coincide with the usual ones.

\section{Group completions and connective spectra}

There are several approaches to forming group-completions of the theories we've discussed. The one we'll follow is to define a theory $\GrSpan$, the group completion of the theory of spans, by hand. This will come equipped with a morphism $\Span\rightarrow\GrSpan$.

Then the group-completion $\Gr T$ of an arbitrary theory $T$ equipped with a functor $\Span\rightarrow T$ will be the pushout
\begin{displaymath}
 \xymatrix{\Span\ar[r]\ar[d]&\GrSpan\ar[d]\\
           T\ar[r]&\pb{135}\Gr T}
\end{displaymath}

We will show (Proposition \ref{grspan-really-is-grouplike}) that $\GrSpan$ is the theory whose models are grouplike $E_\infty$-monoids, which are well known to be equivalent to infinite loop spaces. In particular, we shall see (at the end of Subsection \ref{completed-theories}) that $\GrSpan(*,*)\isom QS^0=\Omega^\infty\Sigma^\infty S^0$ is the group-completion of the free $E_\infty$-monoid on one point.

The space $QS^0$ is famously hard to work with. Thus, in my view, it would be optimistic to hope for an explicit, discrete description of $\GrSpan$ along the lines of that given earlier for $\Span$. Our description thus employs categorical machinery.

In fact, it should properly be regarded as a strength of the present approach that we are able to describe a theory with so little work and in such a natural manner.

\subsection{Group-completed spans}

In this subsection, we define $\GrSpan$.

It is easy to show that an ordinary monoid $M$ is a group if and only if the map $M^2\rightarrow M^2$ defined by $(a,b)\mapsto(a+b,b)$ is invertible. Thus an $E_\infty$-monoid is grouplike if and only if this map is weakly invertible. It is slightly complicated to study invertibility of such maps directly, so we do it by stealth.

First we define a category $T_1$, intended to be the theory of objects with endomorphisms, as follows:
\begin{defn}
  The objects of $T_1$ are finite sets. Also, we set
 $$T_1(X,Y)=\left\{(f:Y\rightarrow X, \alpha:Y\rightarrow\N)\right\}.$$
  These compose according to
    $$(g,\beta)(f,\alpha)=(fg,\beta+g^*\alpha).$$
\end{defn}

So the part $f$ records how to use the diagonals and unit maps (just as in the theory $\Finop$), and $\alpha$ records the number of times that each element has had the endomorphism applied to it.

Similarly, we define $T_1^\sim$, the theory of objects with automorphisms, in exactly the same way, only we allow maps into $\Z$ rather than $\N$ (hence an object can have the automorphism applied a negative number of times).

\begin{prop}
 The categories $T_1$ and $T_1^\sim$ are both theories, with the inclusion functor from $\Finop$ the one sending $f$ to $(f,0)$.
\end{prop}
\begin{proof}
 It is easy to check that disjoint unions are products (with projection maps the projections of $\Finop$ together with the zero map); so the inclusion functor is evidently product-preserving and is also surjective on objects.
\end{proof}

This coincides with any other reasonable way of defining objects with endomorphisms or automorphisms. For example, it can be shown that (for $\cU$ a category with finite products):
$$\Mod(T_1,\cU)\isom\Fun(B\N,\cU)\isom\Fun(\Delta^1/\partial\Delta^1,\cU).$$
and
$$\Mod(T_1^\sim,\cU)\isom\Fun(B\N,J_c(\cU))\isom\Fun(B\Z,\cU),$$
where $J_c(\cU)$ denotes the largest sub-Kan complex of $\cU$ (the $\infty$-groupoid of objects and invertible maps between them); this will be studied briefly in Section \ref{components-and-units}.

Now, we actually are concerned with objects $X$ such that $X\times X$ is equipped with an endomorphism or automorphism. We thus define $T_2$ and $T^\sim_2$ as pushouts in the quasicategory of quasicategories with products and product-preserving functors:
\begin{displaymath}
 \vcenter{\xymatrix{\Finop\ar[r]^{\times 2}\ar[d]&\Finop\ar[d]\\
          T_1\ar[r]&\pb{135}T_2}}
  \qquad\qquad
 \vcenter{\xymatrix{\Finop\ar[r]^{\times 2}\ar[d]&\Finop\ar[d]\\
          T^\sim_1\ar[r]&\pb{135}T^\sim_2.}}
\end{displaymath}
Proposition \ref{cinftypp-cocomplete} argues that these pushouts both exist.

\begin{prop}
The categories $T_2$ and $T^\sim_2$ are both theories, when equipped with the vertical structure maps $\Finop\rightarrow T_2$ and $\Finop\rightarrow T^\sim_2$.
\end{prop}
\begin{proof}
 By construction the structure maps are product-preserving. Being pushouts of essentially surjective maps, they are also essentially surjective.
\end{proof}

Moreover, Proposition \ref{models-colimits-to-limits} tells us that the functor $\Mod(-,\cU)$ takes pushouts to pullbacks: pushouts of theories model things with two compatible structures.

Thus a model of $T_2$ is indeed a model $X$ of $\Fins$ equipped with an endomorphism of $X\times X$, and a model of $T^\sim_2$ is a model $X$ of $\Fins$ equipped with an automorphism of $X\times X$.

It is possible to give generators and relations for $T_2$ and $T^\sim_2$ as a category. $\Ob T_2=\Ob T^\sim_2=\Ob\Finop$, and the morphisms have the following generators:
\begin{itemize}
 \item $f^*:Y^\op\rightarrow X^\op$ for $f:X\rightarrow Y$ a map of sets, and
 \item $\tau(x,y):X^\op\rightarrow X^\op$ for $x\neq y\in X$ (together with inverses $\tau(x,y)^{-1}$ in the case of $T^\sim_2$).
\end{itemize}
 These are subject to the following relations:
\begin{itemize}
 \item $(gf)^*=f^*g^*$,
 \item $\tau(x,y)\tau(w,z)=\tau(w,z)\tau(x,y)$ for any distinct $x,y,z,w$.

 \item If there is an isomorphism $f^{-1}(x)=(x_1,\ldots,x_n)\isom f^{-1}(y)=(y_1,\ldots,y_n)$, we have $\tau(x_1,y_1)\cdots\tau(x_n,y_n)f^*=f^*\tau(x,y)$.

\end{itemize}
However, in practice this is rather unpleasant, so we use the indirect approach.

There is an evident inclusion $T_2\rightarrow T_2^\sim$, arising as the pullback of the evident inclusion $T_1\rightarrow T_1^\sim$. To complete the structure, we need the following:
\begin{prop}
There is also a natural map $T_2\rightarrow\Span$, modelling $(a,b)\mapsto(a+b,b)$
\end{prop}
\begin{proof}
 It occurs as the universal map associated to the following diagram (in the quasicategory of quasicategories with products and product-preserving maps):
\begin{displaymath}
\xymatrix{\Finop\ar[r]^{\times 2}\ar[d]&\Finop\ar[d]\ar[ddr]&\\
          T_1\ar[r]\ar[drr]&\pb{135}T_2\dar[dr]&\\
          &&\Span.}
\end{displaymath}
Here the map $\Finop\rightarrow\Span$ is the structure map of $\Span$. The map $T_1\rightarrow\Span$ is defined as follows.

To each morphism $(f:X\leftarrow Y,\alpha:Y\rightarrow\N)$ we associate a diagram
$$X\stackrel{f}{\longleftarrow}Y\stackrel{e}{\longleftarrow}E_{\alpha}$$
where $E_\alpha=\coprod_{y\in Y}\{1,\ldots,\alpha(y)\}$ and the map $E_\alpha\rightarrow Y$ is the evident one.

Now we define the functor. To an object $X\in\Ob T_1$ we associate $X\sqcup X\in\Span_0$. To a morphism $(f:X\leftarrow Y,\alpha:Y\rightarrow\N)$ we associate the span diagram
\begin{displaymath}
 \xymatrix{&Y\sqcup Y\sqcup E_\alpha\ar[dl]_{f_1\sqcup f_2\sqcup e_1}\ar[dr]^{\id_1\sqcup\id_2\sqcup e_2}&\\
           X\sqcup X&&Y\sqcup Y}
\end{displaymath}
(where, for example, $f_i$ is supposed to denote taking the map $f$ into the $i$th summand).

This extends to a functor in a unique fashion. Indeed, given two morphisms $(f:X\leftarrow Y, \alpha:Y\rightarrow\N)$, and $(g:Y\leftarrow Z, \beta:Z\rightarrow\N)$ the span diagram
\begin{displaymath}
\xymatrix @!=1.3pc {
&&Z\sqcup Z\sqcup g^*E_\alpha\sqcup E_\beta\pb{270}\ar[dl]\ar[dr]&&\\
&Y\sqcup Y\sqcup E_\alpha\ar[dl]\ar[dr]&&Z\sqcup Z\sqcup E_\beta\ar[dl]\ar[dr]&\\
X\sqcup X&&Y\sqcup Y&&Z\sqcup Z}
\end{displaymath}
realises the compatibility of composition.
\end{proof}

Now we can define $\GrSpan$ to be the pushout
\begin{displaymath}
 \xymatrix{T_2\ar[r]\ar[d]&T^\sim_2\ar[d]\\
           \Span\ar[r]&\pb{135}\GrSpan.}
\end{displaymath}

It has the basic properties we expect:
\begin{prop}
 $\GrSpan$ is a theory.
\end{prop}
\begin{proof}
 The argument is the same as that for $T_2$, above.
\end{proof}

Now, the following proposition justifies our work introducing $\GrSpan$:
\begin{prop}
\label{grspan-really-is-grouplike}
The quasicategory of models of $\GrSpan$ is the quasicategory of grouplike models of $\Span$. 
\end{prop}
\begin{proof}
By Proposition \ref{models-colimits-to-limits}, the models of $\GrSpan$ are those models of $\Span$ where the natural map $(x,y)\mapsto(x+y,y)$ is weakly invertible. This is a statement of grouplikeness.
\end{proof}

This means in particular that discrete models of $\GrSpan$ are just abelian groups. Further, since models of $\GrSpan$ in $\Spaces$ are grouplike $E_\infty$-monoids, we can think of them as infinite loop spaces, or as connective spectra.

\subsection{Group-completed theories and models}
\label{completed-theories}

Now, given a theory $T$ and a morphism of theories $\Span\rightarrow T$, we define the group-completion $\Gr T$ of $T$ to be the pushout in $\Theories$ of
\begin{displaymath}
 \xymatrix{\Span\ar[r]\ar[d]&\GrSpan\ar[d]\\
           T\ar[r]_{i_{\Gr}}&\pb{135}\Gr T.}
\end{displaymath}

Proposition \ref{models-colimits-to-limits} ensures that this is indeed the quasicategory of models of $T$ whose underlying monoid is grouplike.

In particular, models of $\Gr\Spam$ are grouplike $E_\infty$-semiring spaces, and can thus should thought of as \emph{connective ring spectra}.

The morphism $T\rightarrow\Gr T$ of theories gives a ``forgetful'' pullback functor $\Mod(\Gr T,\cU)\rightarrow\Mod(T,\cU)$. Proposition \ref{models-pushforward} equips this with a left adjoint we call $Q$, such that as mapping spaces
$$\Mod(\Gr T,\cU)(QA,B)=\Mod(T,\cU)(A,B).$$
From this description we see are entitled to regard $QA$ as the group completion of $A$. In particular, for $T=\Span$, the functor $Q$ is a quasicategorical left adjoint to the inclusion functor of quasicategories from grouplike $E_\infty$-spaces to all $E_\infty$-spaces.

Thus, in particular, for $T=\Span$, our construction $QA$ is isomorphic to the traditional constructions of group completion, such as $\Omega BA$.

\subsection{Components and units}
\label{components-and-units}

The functor $\pi_0$ is a product-preserving functor from the quasicategory $\Spaces$ to the ordinary category $\Set$, which we call the \emph{components functor}. As a result, any model of $\Span$ in $\Spaces$ has an underlying monoid of components, any model of $\GrSpan$ in $\Spaces$ has an underlying group of components, and so on.

Also, regarding $\pi_0$ as valued instead in discrete spaces, there is a natural map $X\rightarrow\pi_0 X$.

It is quick to check that if $S\subset\pi_0 X$, then the homotopy pullback $S\timeso{\pi_0 X}X$ is given by those components of $X$ corresponding to $S$. This is clear from the standard model for a homotopy pullback of spaces (or alternatively, by oldfashioned means: the map $X\rightarrow\pi_0X$ is a fibration for all fibrant $X$).

Given a model $M:\Span\rightarrow\Spaces$, we define the \emph{monoid of units} $M^\times:\Span\rightarrow\Spaces$ via the pullback (in models of $\Span$)
\begin{displaymath}
\xymatrix{M^\times\pb{315}\ar[r]\ar[d]&M\ar[d]\\
(\pi_0M)^\times\ar[r]&\pi_0M.}
\end{displaymath}
Since limits are computed pointwise, we see in particular that the underlying space of $M$ is built in the same way:
\begin{displaymath}
\xymatrix{|M^\times|\pb{315}\ar[r]\ar[d]&|M|\ar[d]\\
(\pi_0M)^\times\ar[r]&\pi_0M.}
\end{displaymath}

The model $A^\times$ is grouplike by construction, and thus admits a homotopy unique lift to a model of $\GrSpan$.

Evidently, $\pi_0$ is an isomorphism on discrete spaces. So if we had started with a discrete model of $\Span$: an ordinary monoid, then we would have produced its group of units, in the usual sense.

The basic result showing that this construction has good properties is the following:
\begin{prop}
There is an adjunction, where the left adjoint is given by $I^*_{\Gr}:\Mod(\Gr\Span)\rightarrow\Mod(\Span)$, and with the right adjoint given by restricting to those connected components which are invertible on $\pi_0$.
\end{prop}
\begin{proof}
We define a simplicial set $\MonGr$ together with a natural map $\MonGr\rightarrow\Delta^1$ by
$$\MonGr=(\Delta^1\times\Mod(\Span))\timeso{(\{0\}\times\Mod(\Span))}(\{0\}\times\Mod(\Gr\Span)).$$

This is readily shown to be an inner Kan fibration using that $\Mod(\Span)$ is a quasicategory and $\Mod(\Gr\Span)$ is equivalent to an a subquasicategory (the subquasicategory of grouplike objects).

It is also not difficult to show that, for $M\in\Mod(\Gr\Span)$, the morphism $M\rightarrow I^*_{\Gr}M$ is cocartesian.

Given $M\in\Mod(\Span)$ given by a product-preserving functor $\Span\rightarrow\Spaces$, we postcompose with the product-preserving functor $\pi_0:\Spaces\rightarrow\Set$ to get a model $\pi_0M$ of $\Span$ on $\Set$: a commutative monoid.

We can then define $M^\times$ (which we also write $\GL_1M$ to be obtained by restricting to the subspaces of $M$ whose image in $\pi_0M$ lies in $(\pi_0M)^\times$; this forms the pullback (in models of $\Span$) as follows:
\begin{displaymath}
\xymatrix{M^\times\pb{315}\ar[r]\ar[d]&M\ar[d]\\
(\pi_0M)^\times\ar[r]&\pi_0M.}
\end{displaymath}

It is also quick to show that for any grouplike commutative monoid $A$, the morphism $A^\times\rightarrow A$ is a cartesian morphism in $\MonGr$, so that the projection to $\Delta^1$ is bicartesian and so classifies an adjunction.
\end{proof}

One particularly interesting application of this is to the units of ring spectra. We recall from section \ref{spam-e-infty} that there is a natural map $I_\Pi:\Span\rightarrow\Spam$, the \emph{inclusion of the multiplicative monoid structure} given by

$$\Span\isom\Span\circ_{(\DR)'}\Fin_0\longrightarrow \Span\circ_\DR\Fin=\Spam,$$
where $\Fin_0$ is the discrete simplicial set on objects $\Fin_0$, and $(\DR)'$ is the restriction of $\DR$ to $\Fin_0$.

Accordingly, we define the \emph{units} of a model of $\Spam$ to be the units of the underlying multiplicative monoid.

Using Proposition \ref{models-pushforward}, this fits into a chain of adjunctions
\begin{displaymath}
\xymatrix{\Mod(\Gr\Span)\ar@<2pt>[rr]^{I^*_{\Gr}}&&\Mod(\Span)\ar@<2pt>[ll]^{\mathrm{units}}\ar@<2pt>[rr]^{I_*}&&\Mod(\Gr\Spam)\ar@<2pt>[ll]^{I^*}}
\end{displaymath}

Here $I$ is the diagonal of the following commutative square of theories:
\begin{displaymath}
\xymatrix{\Span\ar[r]\ar[d]\ar[dr]^I&\Gr\Span\ar[d]\\
          \Spam\ar[r]&\Gr\Spam.}
\end{displaymath}

The right adjoint, associating to a connective ring spectrum the units of its underlying multiplicative monoid, is the correct notion of the units of a connective ring spectrum. The left adjoint takes a connective spectrum and performs a kind of topological groupring construction, which models $\Sigma^\infty_+\Omega^\infty_+$. This adjunction has been achieved already in \cite{ABGHR}, by very different means.

\section{$K$-theory of categories}

Recall that the inclusion $\Grp\rightarrow\Mon$ admits a left adjoint, the \emph{Grothendieck construction}. This gives the universal functorial group-valued invariant of a monoid. Thus, for example, when we try to define $K_0$ of a ring $R$ to be the universal group-valued invariant of $R$-modules which is additive on direct sums, we are led to define it to be the Grothendieck construction of the monoid of isomorphism classes of $R$-modules (with the monoid operation given by direct sum).

Similarly, there is an evident inclusion functor $\Spaces\rightarrow\Cinfty$. We show below that it admits a left adjoint; this can be thought of as a quasicategorical Grothendieck construction.

Evidently this is in some natural sense the universal space-valued invariant of a quasicategory, and thus we should expect it to extend our classical notions of $K$-theory spaces of categories.

\subsection{Groupoid completion}

There is a natural inclusion functor $i:\Spaces\rightarrow\Cinfty$ extending the inclusion of the Kan complexes into the inner Kan complexes.

In this section, we study a kind of group-completion with many objects, as being the left adjoint to $i$. We'll do this in a roundabout manner.

\begin{prop}
\label{Spaces-Cinfty-right-adj}
The functor $i:\Spaces\rightarrow\Cinfty$ admits a \emph{right} adjoint, defined by taking the space of objects and equivalences between them.
\end{prop}
\begin{proof}
First of all, there is an adjunction between the ordinary categories $\Kan$ and $\wKan$ of Kan complexes and of inner Kan complexes respectively. The left adjoint $i_c$ is the evident inclusion; the right adjoint $J_c$ associates to a space its unique maximal Kan subsimplicial set: $J_c(X)$ is the union of the images of all maps from Kan complexes into $X$.

We recall that the quasicategory $\Cinfty$ is defined to be $N^\coh(\Cinfty^\Delta)$, the coherent nerve of the simplicial category of weak Kan complexes and the mapping spaces between them.

We recall likewise that the quasicategory $\Spaces$ is defined by $N^\coh(\Spaces^\Delta)$, where $\Spaces^\Delta$ is the full subsimplical category of $\Cinfty$ whose objects are the Kan complexes.

Now, we seek to boost the adjunction between $i_c$ and $J_c$ to an adjunction between $\Spaces$ and $\Cinfty$. We shall exhibit a quasicategory $\cA$ and bicartesian functor $\cA\rightarrow\Delta^1$ to classify this adjunction.

We define it as follows. For any $f:X\rightarrow\Delta^1$, we take
$$\sSet_{/\Delta^1}(X,\cA)=\left\{\text{maps $X\longrightarrow\Cinfty$ with $\im f^{-1}(0)\subset\Spaces$}\right\}.$$

It is routine to show that this is cocartesian (and that a space $X\in\Spaces_0$ has cocartesian lift over the nondegenerate 1-simplex of $\Delta^1$ given by $X\rightarrow iX$). So we show that it is cartesian.

Given a quasicategory $\cC\in(\Cinfty)_0$, we intend to show that the 1-simplex $J_c\cC\rightarrow\cC$ is a cartesian lift. This entails showing that the functor
$$\cA_{/(J_c\cC\rightarrow\cC)}\longrightarrow \cA_{/\cC}\timeso{\cA}\Spaces$$
is acyclic Kan.

As in previous arguments of this sort, a lifting problem of shape $\partial\Delta^n\rightarrow\Delta^n$ can be decoded to give an extension problem
\begin{displaymath}
\xymatrix{\Lambda^{n+2}_{n+2}\ar[r]\ar[dr]&\Delta^{n+2}\dar[d]\\
&\cA.}
\end{displaymath}
Here the final two vertices are sent to $J_c\cC$ and $\cC$ respectively. Since $J_c\cC$ is adjoint in the category of ordinary categories, all the maps to $\cC$ actually factor through $J_c\cC$. Thus we can recover a map $\Lambda^{n+2}_{n+2}\rightarrow\Spaces$, sending the final two vertices to $J_c\cC$.

But, since the final two vertices are exchangeable, this is equivalent to a map $\Lambda^{n+2}_{n+1}\rightarrow\Spaces$, which we can extend to a map $\Delta^{n+2}\rightarrow\Spaces$. We then swap the final two vertices back to get the required map.
\end{proof}

Now we can show what we need:
\begin{prop}
\label{Spaces-Cinfty-left-adj}
The functor $i:\Spaces\rightarrow\Cinfty$ has a left adjoint $B$.
\end{prop}
\begin{proof}
We apply the Adjoint Functor Theorem \cite{HTT}*{5.5.2.9}. Lurie shows \cite{HTT}*{Example 5.5.1.8} that $\Spaces$ is presentable, and $\Cinfty$ is also presentable \cite{LurieBicat}*{Remark 1.2.11}.

The preceding Proposition \ref{Spaces-Cinfty-right-adj} shows that $i$ has a right adjoint, and so preserves all small colimits. Hence it is $\omega$-continuous in particular.

In order to apply the Adjoint Functor Theorem, it remains to show that $i$ preserves small limits; as usual it suffices to consider products and pullbacks. Quasicategorical products are modelled by products of simplicial sets in both $\Spaces$ and $\Cinfty$. And homotopy pullbacks (of diagrams $X\rightarrow Z\leftarrow Y$) are modelled in both categories by the simplicial set
$$(X\times Y)\times_{(Z\times Z)}\Fun(E2,Z)$$
where $E2$ is the standard contractible simplicial set on two vertices (with one $n$-cell for each $(n+1)$-tuple of vertices).
\end{proof}

By definition, for any quasicategory $\cC$ and Kan complex $X$ this functor $B$ satisfies
$$\Spaces(B(\cC),X) \isom \Cinfty(\cC,X).$$
Since the homspaces are computed the same way, this means $B(\cC)$ is a space weakly equivalent to $\cC$; thus if a model for $B$ is needed, any fibrant replacement functor will do (such as Kan's $\Ex^\infty$ \cite{Goerss-Jardine}, or the singular complex of the geometric realisation).

For our needs, we must also show:
\begin{prop}
\label{B-preserves-prods}
The functor $B:\Cinfty\rightarrow\Spaces$ preserves products.
\end{prop}
\begin{proof}

For any space $X$ we have the following equivalence of mapping spaces:

$$\Spaces(B(1),X)\isom\Cinfty(1,X)\isom X.$$
Hence $B(1)\isom 1$.

A similar but more painful argument deals with binary products:
\begin{align*}
      &\Spaces(B(\cC\times\cD),X)\\
 \isom&\Cinfty(\cC\times\cD,X)\\
 \isom&\Cinfty(\cC,\Fun(\cD,X))\\
 \isom&\Spaces(B(\cC),\Fun(\cD,X))\\
 \isom&\Cinfty(B(\cC),\Fun(\cD,X))\\
 \isom&\Cinfty(\cD,\Fun(B(\cC),X))\\
 \isom&\Spaces(B(\cD),\Fun(B(\cC),X))\\
 \isom&\Cinfty(B(\cD),\Fun(B(\cC),X))\\
 \isom&\Cinfty(B(\cC)\times B(\cD),X)\\
 \isom&\Spaces(B(\cC)\times B(\cD),X),
\end{align*}
from which we can conclude that $B(\cC\times\cD)\isom B(\cC)\times B(\cD)$.
\end{proof}

\subsection{$K$-theory}

The subsection above gives us a product-preserving ``$\infty$-groupoid completion'' functor. We have provocatively named it $B$; we combine it with the group-completion functor of theories to give a model of $K$-theory.

Given a monoidal theory $T$ and a model $A:T\rightarrow\Cinfty$, we define $K(A)$ to be the group-completion $Q(B(A))$ obtained by $\infty$-groupoid-completing $A$, and then pushing forward along the map of theories $T\rightarrow\Gr T$ to obtain a model of $\Gr T$.
\begin{displaymath}
\xymatrix{T\ar[d]\ar[r]^{A}&\Cinfty\ar[r] ^B&\Spaces\\
          \Gr T\dar[urr]_{K(A)}&&}
\end{displaymath}

This provides a framework for extending the $K$-theory of permutative categories to symmetric monoidal categories (as models of $\Span$) and the $K$-theory of bipermutative categories to models of $\Spam$.

Since $B$ is a fibrant replacement functor and $Q$ is the normal group-completion, this agrees with the classical constructions of $K$-theory.

For example, this allows us to conclude that the structure $\Vect:\Spam\rightarrow\Cat$ on the category of vector spaces described in Section \ref{Vect-as-model-of-Spam} gives us a $K$-theory $K(\Vect):\Gr\Spam\rightarrow\Spaces$. This provides a rival model to the multiplicative structure of the $K$-theory of \cite{Elmendorf-Mandell}. While we have not shown that our multiplicative structure agrees with theirs, we feel confident in stating it as a conjecture.

\end{document}